\documentclass[11pt]{article}

\usepackage[a4paper,margin=2.4cm]{geometry}
\usepackage[T1]{fontenc}
\usepackage[utf8]{inputenc}
\usepackage{lmodern}
\usepackage{microtype}
\usepackage{amsmath,amssymb,amsthm,mathtools,bm}
\usepackage{algorithm}
\usepackage{algpseudocodex}
\usepackage{booktabs}
\usepackage{tabularx}
\usepackage{pdflscape}
\usepackage{graphicx}
\usepackage{tikz}
\usetikzlibrary{arrows.meta}
\usepackage{siunitx}
\usepackage{xcolor}
\usepackage[colorlinks=true,allcolors=blue!55!black]{hyperref}
\usepackage[capitalise,noabbrev]{cleveref}
\usepackage[backend=biber,style=numeric-comp,sorting=none,giveninits=true,maxbibnames=9]{biblatex}
\newtheorem{theorem}{Theorem}
\newtheorem{proposition}{Proposition}
\newtheorem{corollary}{Corollary}
\newtheorem{lemma}{Lemma}
\theoremstyle{definition}
\newtheorem{definition}{Definition}
\newtheorem{remark}{Remark}

\newcommand{\R}{\mathbb{R}}
\newcommand{\infconv}{\mathbin{\square}}
\newcommand{\conj}[1]{#1^{*}}
\newcommand{\argminop}{\operatorname*{arg\,min}}
\newcommand{\E}{\mathbb{E}}                 %
\newcommand{\Qexp}{Q}                       %
\newcommand{\cvar}[1]{\operatorname{CVaR}_{#1}} %
\newcommand{\pkg}{\texttt{FLToptim}}

\newcommand{\annualEightNodeGurobiSpeedup}{1.50}
\newcommand{\annualGapMax}{0.1045\%}

\newcommand{\annualMemFactorThreeToTenRange}{$3.09$--$5.93$}
\newcommand{\annualMultiEnergyFltRss}{300}
\newcommand{\annualMultiEnergyFltWall}{2.45}
\newcommand{\annualMultiEnergyGap}{+0.0793\%}

\newcommand{\annualMultiEnergyGurobiRss}{680}

\newcommand{\annualMultiEnergyGurobiWall}{2.17}

\newcommand{\annualMultiEnergyHighsRss}{884}

\newcommand{\annualMultiEnergyHighsSpeedup}{6.38}
\newcommand{\annualMultiEnergyHighsWall}{15.62}

\newcommand{\annualMultiEnergyMemFactorRange}{$2.27$--$2.95$}
\newcommand{\annualOneNodeFltRss}{59}
\newcommand{\annualOneNodeGurobiSpeedup}{16.1}
\newcommand{\annualOneNodeHighsSpeedup}{17.3}
\newcommand{\annualTenNodeGurobiSpeedup}{1.77}
\newcommand{\annualThirtyNodeFltRss}{550}
\newcommand{\annualThirtyNodeFltWall}{15.23}
\newcommand{\annualThirtyNodeGap}{+0.1045\%}
\newcommand{\annualThirtyNodeGurobiMemFactor}{6.68}
\newcommand{\annualThirtyNodeGurobiRss}{3{,}673}

\newcommand{\annualThirtyNodeGurobiSpeedup}{1.07}
\newcommand{\annualThirtyNodeGurobiWall}{16.35}
\newcommand{\annualThirtyNodeHighsMemFactor}{5.92}
\newcommand{\annualThirtyNodeHighsRss}{3{,}256}

\newcommand{\annualThirtyNodeHighsWall}{195.38}

\newcommand{\annualThirtyNodeMemFactorRange}{$5.92$--$6.68$}
\newcommand{\annualThreeNodeGurobiSpeedup}{5.62}
\newcommand{\annualThreeToTenGapMax}{0.0459\%}
\newcommand{\catLawCampaignDate}{2026-08-13}
\newcommand{\catLawCampaignVersion}{0.27.0}
\newcommand{\catLawDoubleCat}{233}
\newcommand{\catLawDoubleOneDGurobiTime}{609}
\newcommand{\catLawDoubleOneDRatio}{25.0}
\newcommand{\catLawDoubleOneDTime}{24.3}
\newcommand{\catLawDoubleTwoDGurobiTime}{1.75}
\newcommand{\catLawDoubleTwoDRatio}{13.6}
\newcommand{\catLawDoubleTwoDTime}{23.8}
\newcommand{\catLawFCheckCatHigh}{488}
\newcommand{\catLawFCheckCatLow}{159}
\newcommand{\catLawFCheckMeasuredHigh}{243{,}542}
\newcommand{\catLawFCheckMeasuredLow}{25{,}760}
\newcommand{\catLawFCheckPredictedHigh}{238{,}144}
\newcommand{\catLawFCheckPredictedLow}{25{,}281}
\newcommand{\catLawLosslessCampaignDate}{2026-08-15}

\newcommand{\catLawLossyFaceRatioRange}{$13.9$--$14.8$}

\newcommand{\catLawLossyTimeFactor}{70.1}
\newcommand{\catLawMemWallCatHigh}{949}
\newcommand{\catLawMemWallCatLow}{600}

\newcommand{\catLawMemWallMachineGb}{38}

\newcommand{\catLawOneDCatMax}{488}
\newcommand{\catLawOneDCatMin}{5}
\newcommand{\catLawOneDExactnessMax}{$6.9\,10^{-15}$}

\newcommand{\catLawOneDGainRange}{$15.1$--$30.2$}
\newcommand{\catLawRefusedCat}{363}
\newcommand{\catLawRefusedSubopt}{$6.4\,10^{-9}$}
\newcommand{\catLawTwoDCatMax}{488}
\newcommand{\catLawTwoDCatMin}{159}
\newcommand{\catLawTwoDExactnessMax}{$1.1\,10^{-11}$}

\newcommand{\catLawTwoDLossRange}{$13.5$--$107$}

\newcommand{\crossEightZoneGapFactor}{0.79}
\newcommand{\crossEightZoneMonolithSeconds}{24.4}
\newcommand{\crossEightZoneNewSlopesRound}{1218}
\newcommand{\crossEightZoneRoundsExtrap}{93}
\newcommand{\crossEightZoneRoundsProbed}{11}
\newcommand{\crossEightZoneSeenOnceShare}{59.8\%}

\newcommand{\crossEightZoneUnits}{22}

\newcommand{\crossThreeZoneRounds}{9}

\newcommand{\ddimAmpFourFive}{$2.6\times$}
\newcommand{\ddimAmpThreeFour}{$2.9\times$}
\newcommand{\ddimAmpTwoThree}{$3.8\times$}
\newcommand{\ddimByteGateCells}{25}

\newcommand{\ddimDGeqThreeSolveHigh}{$0.155\times$}
\newcommand{\ddimDGeqThreeSolveLow}{$1.3\,10^{-5}$}
\newcommand{\ddimDTwoYearBSBest}{$2.39\times$ ($k=12$)}
\newcommand{\ddimDTwoYearBSWorst}{$0.822\times$ ($k=6$)}

\newcommand{\ddimFmaxDFive}{444}
\newcommand{\ddimFmaxDFour}{168}
\newcommand{\ddimFmaxDThree}{57}
\newcommand{\ddimFmaxDTwo}{15}
\newcommand{\ddimGateBCells}{19}
\newcommand{\ddimGenReductionHigh}{$239\times$}
\newcommand{\ddimGenReductionLow}{$4.6\times$}
\newcommand{\ddimProbeCount}{1600}

\newcommand{\ddimWallCount}{4628}

\newcommand{\fltScheduleVsMdaRange}{$6.54$--$7.14$}

\newcommand{\fltValueFunctionVsMdaRange}{$2.62$--$3.17$}

\newcommand{\jitterBuiltFacesHigh}{243{,}542}
\newcommand{\jitterBuiltFacesLow}{25{,}760}
\newcommand{\jitterBuiltOverRealizedHigh}{$4.1\,10^{4}$}
\newcommand{\jitterBuiltOverRealizedLow}{$4.3\,10^{3}$}
\newcommand{\jitterDiverseGroupingCostPct}{23\%}
\newcommand{\jitterEconLevelsHigh}{18}
\newcommand{\jitterEconLevelsLow}{7}

\newcommand{\jitterGateExactMax}{$8.2\,10^{-11}$}
\newcommand{\jitterRealizedFHigh}{6}
\newcommand{\jitterRealizedFLow}{4}

\newcommand{\jointSizingFltFractionRange}{$0.24$--$0.32$}
\newcommand{\jointSizingFltSpeedupRange}{$3.1$--$4.1$}
\newcommand{\jointSizingInstances}{4}
\newcommand{\jointSizingOracleSpeedupRange}{$3.8$--$16.2$}
\newcommand{\jointSizingRelAgreement}{$\le5.8\,10^{-7}$}
\newcommand{\kCurveBarrierPureAdvantageRange}{$3.5$--$5.8\times$}
\newcommand{\kCurveBarrierPureObjAgreement}{$4.2\,10^{-9}$--$2.4\,10^{-7}$}
\newcommand{\kCurveBarrierPureSecondsKSixteen}{16.6}
\newcommand{\kCurveBarrierPureSecondsKSixtyFour}{61.2}
\newcommand{\kCurveBarrierPureSecondsKThirtyTwo}{35.1}

\newcommand{\kCurveCrossoverKAllSeeds}{16}
\newcommand{\kCurveCrossoverPushIters}{1{,}392{,}517}

\newcommand{\kCurveCutsRangeAllSeeds}{13--21}
\newcommand{\kCurveDefaultBarrierIters}{24}
\newcommand{\kCurveDefaultBarrierSeconds}{33.9}

\newcommand{\kCurveFrontalFailModeKSixtyFour}{\texttt{internal\_solver\_error}}

\newcommand{\kCurveFrontalSolver}{Gurobi}

\newcommand{\kCurveGapTarget}{$5\,10^{-7}$}
\newcommand{\kCurveGurobiVersion}{13.0.2}

\newcommand{\kCurvePerSolveRange}{0.242--0.316}
\newcommand{\kCurveRatioKFour}{$0.501\times$}
\newcommand{\kCurveRatioKSixteen}{$1.35\times$}

\newcommand{\kCurveRatioKThirtyTwo}{$1.92\times$}

\newcommand{\kCurveSeedsCount}{3}
\newcommand{\kCurveWallSeedsReproduced}{3}
\newcommand{\kCurveWarmColdGapKSixtyFour}{$1.7\,10^{-7}$}
\newcommand{\kernelCertifiedExactHigh}{13.4}
\newcommand{\kernelCertifiedExactLow}{12.1}
\newcommand{\kernelCertifiedFloatHigh}{12.3}
\newcommand{\kernelCertifiedFloatLow}{11.2}
\newcommand{\kernelHorizonHigh}{15{,}000}
\newcommand{\kernelHorizonLow}{5{,}000}
\newcommand{\kernelRelGap}{$1.1\,10^{-15}$}
\newcommand{\kernelScalingMaxN}{100{,}000}
\newcommand{\kernelScalingMaxTime}{0.0631}
\newcommand{\kernelScalingSpan}{100}
\newcommand{\kernelScalingUsHigh}{0.631}
\newcommand{\kernelScalingUsLow}{0.576}
\newcommand{\kernelSpeedupHigh}{93}
\newcommand{\kernelSpeedupLow}{35}
\newcommand{\kernelSpeedupRange}{$35$--$93$}

\newcommand{\mdaHorizonHigh}{15{,}000}
\newcommand{\mdaHorizonLow}{5{,}000}

\newcommand{\mdaVsGurobiRange}{$4.89$--$13.7$}

\newcommand{\mrEightDsmBands}{8}
\newcommand{\mrEightDsmPinError}{$3.8\,10^{-10}$}
\newcommand{\mrEightDsmRho}{10}
\newcommand{\mrEightDsmValue}{5.7}

\newcommand{\mrEightDsmWindows}{365}
\newcommand{\mrEightGap}{1.97}

\newcommand{\mrEightLevelViolation}{$2.5\,10^{-7}$}

\newcommand{\mrEightStorages}{34}

\newcommand{\pairAlphabetHHigh}{16}
\newcommand{\pairAlphabetHLow}{3}
\newcommand{\pairAlphabetVHigh}{18}
\newcommand{\pairAlphabetVLow}{3}
\newcommand{\pairFaceOverCellHigh}{11.3}
\newcommand{\pairFaceOverCellLow}{2.7}

\newcommand{\pkgCampaignVersion}{0.27.0}
\newcommand{\pkgVersion}{0.27.1}

\newcommand{\rampFwGap}{$+1.05\%$}
\newcommand{\rampFwWall}{9.1}
\newcommand{\rampGurobiWall}{0.56}
\newcommand{\rampHighsWall}{1.7}

\newcommand{\rampJointGap}{$1.1\,10^{-14}$}

\newcommand{\rampJointWall}{0.37}
\newcommand{\rampStorageOnlyWall}{0.03}
\newcommand{\reuseCurveSolves}{61}
\newcommand{\reuseCurveVsExactGrid}{3.1}

\newcommand{\reuseFamilyK}{100}
\newcommand{\reuseGridExactSize}{161}

\newcommand{\reuseLayers}{60}
\newcommand{\reuseRelGapMax}{$6.5\,10^{-15}$}
\newcommand{\reuseSpeedupCold}{57.9}
\newcommand{\reuseSpeedupWarm}{3.06}
\newcommand{\sandwichActiveAlphabetNetwork}{9}
\newcommand{\sandwichFvsKSatHigh}{244}
\newcommand{\sandwichFvsKSatLow}{208}
\newcommand{\sandwichGateWorst}{$1.9\,10^{-15}$}
\newcommand{\sandwichLossyBuiltHigh}{25{,}516}
\newcommand{\sandwichLossyBuiltLow}{7{,}696}
\newcommand{\sandwichLossyBuiltOverRealizedHigh}{473}
\newcommand{\sandwichLossyBuiltOverRealizedLow}{160}
\newcommand{\sandwichLossyRampShareHigh}{21\%}
\newcommand{\sandwichLossyRampShareLow}{18\%}
\newcommand{\sandwichLossyRealizedFHigh}{54}
\newcommand{\sandwichLossyRealizedFLow}{48}
\newcommand{\sandwichLossyStressCat}{1{,}852}
\newcommand{\sandwichLossyStressFMax}{406}
\newcommand{\sandwichLossyTwinFactorHigh}{$2.14\times$}
\newcommand{\sandwichLossyTwinFactorLow}{$1.73\times$}
\newcommand{\sandwichRoundsControl}{2}
\newcommand{\sandwichStressActiveShare}{40\%}
\newcommand{\statEdgeAuditTotal}{$3.74$}

\newcommand{\treeExactRelGapMax}{$3.8\,10^{-15}$}

\newcommand{\treeWallHigh}{13}
\newcommand{\treeWallLow}{5.1}
\newcommand{\twoStageEnergyCap}{4}
\newcommand{\twoStageHorizon}{8{,}760}
\newcommand{\twoStageMaxK}{500}
\newcommand{\twoStageMaxRelGap}{$3.1\,10^{-14}$}
\newcommand{\twoStageMaxSigmaGap}{0}
\newcommand{\twoStagePowerCap}{1}
\newcommand{\twoStagePriceGen}{i.i.d.\ $\mathcal{U}(1,100)$ per hour, seed 7}
\newcommand{\twoStageRecords}{5}
\newcommand{\twoStageSplitTOne}{4{,}380}
\newcommand{\twoStageSplitTTwo}{4{,}380}

\newcommand{\twoStoreBlockSpeedupRange}{$5.5$--$43$}
\newcommand{\twoStoreFacesLossless}{42}
\newcommand{\twoStoreFacesLossy}{762}
\newcommand{\twoStoreFamilyMbLossless}{9.8}
\newcommand{\twoStoreFamilyMbLossy}{178}

\newcommand{\twoStoreGurobiSpeedupRange}{$2.1$--$15.3$}
\newcommand{\twoStoreGurobiSpeedupWithBuild}{61}

\newcommand{\twoStoreReadoutInstances}{51}

\newcommand{\twoStoreUlpMax}{5}

\newcommand{\waterValueEnsembleSize}{11}
\newcommand{\waterValueExactGapSingleReservoir}{$1.4\,10^{-13}$}
\newcommand{\waterValueExpectedGap}{+0.250\%}

\newcommand{\waterValueLooGap}{+3.07\%}

\newcommand{\waterValueMyopicGap}{+57.7\%}

\newcommand{\waterValueWetYearGap}{+69.7\%}

\title{\bfseries Exact value functions and sensitivities for energy storage
by fast Legendre transforms}

\author{Robin Girard\thanks{Mines Paris -- PSL University, Centre PERSEE. \texttt{robin.girard@minesparis.psl.eu}}}
\date{October 2026}

\begin{document}
\maketitle

\begin{abstract}
Long chains of decisions coupled only through a bounded cumulative state ---
the operation of an energy-storage device is the canonical example --- form
convex programs that generic solvers treat as arbitrary sparse models. We
present an exact solver for this longitudinal class that propagates the
complete convex cost-to-reach function through restrictions and infimal
convolutions, each evaluated by fast Legendre--Fenchel conjugation as an
ordered merge of slopes: no state grid, no value approximation, and a
\emph{function}, not a single primal--dual pair, as output. Price-taking
storage instances of \kernelHorizonLow{} to \kernelHorizonHigh{} hourly
steps, whose per-step cost keeps a bounded number of pieces, solve in
milliseconds, \kernelSpeedupRange{} times faster than Gurobi --- and the
complete value function is returned faster than our verified
reimplementation of the specialised RAP--NC algorithm returns one schedule.
The output is what pays: capacity marginals read in one differentiated pass,
exactly where finite differences fail on piecewise-linear objectives and LP
duals are set-valued; recourse assembled, not sampled, on scenario trees;
and a certified sizing loop whose \kCurveGapTarget{} certificate closes in a
cut count independent of the number of weather years and of the seed. We
then widen the carried state. Two storages at one node can be
\emph{eliminated} rather than coordinated, by carrying the value function of
the pair, and the object that carries is combinatorially rigid: we prove
that every breakline the recursion realizes is vertical, horizontal, or
borne by a line through the origin \emph{exactly}, with no genericity and no
restriction on losses or horizon --- corroborated by an exact-rational audit
of \statEdgeAuditTotal{} million realized edges. For any number of
co-located stores sharing one charge--discharge coefficient the same
argument gives a $d$-dimensional theorem: every realized wall is a position
hyperplane or a type-$A$ indifference hyperplane $\{s_i=s_j\}$ through the
origin, nothing else, so the carried subdivision never leaves one
arrangement of fixed directions. Speed is not claimed uniformly but mapped
by a catalogue law: one-dimensional chains win the whole measured axis,
ahead of the frontal solve by \catLawOneDGainRange{}$\times$ and growing,
while the planar block pays the squared slope catalogue it carries and loses
it back by \catLawTwoDLossRange{}$\times$ on rich traced network curves ---
the boundary is not a crossover but the state dimension of the reader. Those
curves are themselves exact: a bounded-variable parametric dual simplex
traces the hourly injection cost of a network dispatch, and one such trace
followed by one solve is globally exact for a single storage inside the
network. The open-source package \pkg{}, every benchmark driver and all data
accompany the paper; every table and every headline figure is generated from
a versioned campaign.
\end{abstract}


\section{Introduction}
\label{sec:introduction}

Operating an energy-storage device against a known sequence of prices is one
of the simplest optimisation problems in power systems: charge or discharge
within a power limit while the cumulative energy stays between zero and
capacity. With linear conversion losses this is a linear program, routinely
passed to general-purpose solvers such as Gurobi or
HiGHS~\cite{Gurobi,HiGHS}. Yet this elementary instance exposes a large
performance gap. In the deliberately simple experiment of
\cref{tab:intro-storage}, \pkg{} and Gurobi solve the same sparse
price-taking storage problem with continuous random prices and identical
physical bounds: \pkg{} reproduces the optimum to machine precision and is
already \kernelSpeedupRange{} times faster between \kernelHorizonLow{} and
\kernelHorizonHigh{} periods, the
advantage still growing with the horizon. The gain involves no
approximation, no state-space discretisation, no relaxed physical model and
no specialised hardware. It is, however, a property of the instance as much
as of the method: a price taker's per-step cost keeps a bounded number of
pieces, and \cref{sec:experiments} measures how the comparison reverses when
the cost curves themselves grow with the horizon.

\begin{table}[t]
\centering
\caption{Motivating synthetic storage benchmark. Entries are median end-to-end times; Gurobi uses the natural sparse state-of-charge formulation and \pkg{} the flat-array value-function backend. The maximum relative objective difference is $1.1\,10^{-15}$. The protocol, and the specialised comparator planned alongside the two generic LP solvers, are described in \cref{subsec:kernel-bench}. (\pkg{} 0.27.0, machine \texttt{cyrius-2.home}.)}
\label{tab:intro-storage}
\begin{tabular}{lrrr}
\toprule
 & $T=5{,}000$ & $T=10{,}000$ & $T=15{,}000$ \\
\midrule
FLToptim [s] & 0.00324 & 0.00654 & 0.00926 \\
Gurobi [s] & 0.1134 & 0.4222 & 0.8614 \\
Gurobi/FLToptim & $\mathbf{35.1\times}$ & $\mathbf{64.5\times}$ & $\mathbf{93.0\times}$ \\
\bottomrule
\end{tabular}
\end{table}

The gap has a structural explanation. The storage LP is not an arbitrary
LP: its variables form a long chain coupled by one bounded cumulative
state. If the convex cost-to-reach function of that state is represented
explicitly, the Bellman step is, up to interval restrictions, an infimal
convolution, and Legendre--Fenchel conjugation turns convolution into
addition:
\[
(f\infconv g)^{*}=\conj f+\conj g .
\]
For convex piecewise-linear functions on the real line, conjugation
exchanges ordered breakpoints and slopes, so each Bellman step becomes an
ordered merge of two slope lists --- fast algebra on explicit functions
instead of a new optimisation. The core recursion was introduced in our
conference paper~\cite{Girard2014} for the sole optimisation of storage operation.
The present paper develops this simple idea into a complete method for more
realistic settings --- operation and investment in multi-energy systems with
several spatial nodes, and operation under uncertainty on scenario trees ---
and presents the complete open-source implementation, with case studies,
interfaces to energy-system models, and a benchmark against open-source and
commercial solvers.

An important feature of the method is that the solver returns a function, not
only a schedule. The slopes of that function are marginal values of stored
energy; forward and backward value functions expose the exact sensitivity
to capacity and yield a sizing oracle; and the same representation accepts,
as time-varying cost data, the exact injection-cost curves of an hourly
network dispatch, traced by a parametric dual simplex. Because all these
objects are convex piecewise-linear functions of one variable, they
compose: temporal oracles consume spatial curves, and coordination
algorithms exchange both.

The method is exact precisely where a scalar cumulative state separates
past from future --- the longitudinal class formalised in
\cref{sec:longitudinal} --- and the natural question is what happens when
the separator is no longer scalar. The paper answers it for the first step
up. Two storages at one node can be \emph{eliminated} rather than
coordinated, by carrying the value function of the pair; the carried object
is then a subdivision of a price plane, and the paper's main structural
results say that this subdivision is combinatorially rigid --- it never
leaves one \emph{fixed} arrangement of lines whose directions are known in
advance, and the same holds in dimension $d$ for co-located stores sharing
one charge--discharge coefficient (\cref{sec:pair}). What the widening costs
is measured rather than asserted: a catalogue law separates the regime in
which a one-dimensional chain wins the entire measured axis from the one in
which the planar block, paying a face count quadratic in the slope
catalogue it carries, loses on the rich curves that a real network produces.
The boundary between them is not a crossover but the state dimension of the
reader.

Beyond one intertemporal unit per coupled group, exactness stops and the
composition becomes a coordination scheme whose iterates are feasible with
measured gaps. That layer is not this paper's subject --- it buys
correctness rather than speed, and against a frontal barrier on annual
systems it does not win time --- so it is documented in full in
\cref{app:coordination} and the body stops where exactness stops.

\paragraph{Contributions.}
\begin{enumerate}
\item \emph{A fast Legendre-transform engine for explicit convex functions}
(\cref{sec:flt}): canonical piecewise-linear and piecewise-quadratic
representations, five closed operations with proven closure and
per-operation cost, and an arithmetic--degeneracy contract that never
resorts to perturbation.
\item \emph{An exact solver for longitudinal convex programs}
(\cref{sec:longitudinal}): the restriction--convolution recursion, a
deterministic backward-recovery rule valid on flat optimal faces, explicit
worst-case and output-sensitive costs, and a precise position relative to
the optimal serial algorithm of Tseng and Luo and to nested resource
allocation.
\item \emph{Widening the carried state, and two rigidity theorems}
(\cref{sec:pair}): two storages at one node eliminated by carrying the value
function of the pair, an hourly step that is two directional gradient clamps
and one addition with no planar transform inside the loop, and the
structural results that make it tractable. A
\emph{directional-stationarity} theorem: every realized breakline of the
carried subdivision is vertical, horizontal, or borne by a line through the
origin \emph{exactly}, under no genericity assumption and for arbitrary
losses and horizons, so the subdivision is permanently a sub-complex of one
fixed line arrangement --- only the position alphabet is left to
per-instance measurement, and we prove that this asymmetry is structural
rather than an unfinished proof. And, for any number of co-located stores
sharing one charge--discharge coefficient, a $d$-dimensional
\emph{realized-walls} theorem, monotonicity hypothesis lifted: position
hyperplanes and the type-$A$ indifference hyperplanes $\{s_i=s_j\}$ through
the origin, nothing else. Both are corroborated in exact rational
arithmetic. The block's capacity marginals follow by forward-mode
differentiation of the conjugate recursion.
\item \emph{Value functions as investment and network oracles}
(\cref{sec:investment,sec:injection}): an exact forward--backward
factorisation, a marginal-capacity theorem evaluated directly on
value-function slopes, sizing as a Legendre transform in the capacity
variable, joint energy--power sizing by exact Benders cuts in a level
bundle --- whose certificate closes in a cut count independent of the number
of weather years once the master carries the one feasibility frontier that
optimality cuts cannot supply --- and exact hourly injection-cost curves
traced by a bounded-variable parametric dual simplex with basis transport,
one such trace plus one solve being globally exact for a single storage
inside a network.
\item \emph{Expectation is closed in the class, so recourse is exact ---
at one cut and on a tree} (\cref{sec:stochastic}): a positively weighted sum
of cost-to-go functions is again an explicit convex piecewise-linear function
obtained by one ordered merge, which turns a two-stage stochastic operation
problem with $K$ scenarios into $K+1$ value-function passes and one merge ---
no cuts, no state grid, no sampling of the linking state, no convergence loop
--- and returns the per-scenario marginal water values at the optimal boundary
state as a by-product. On a scenario tree the same argument survives its own
repetition: expectation and elimination alternate down the horizon and one
backward sweep leaves the exact conditional cost-to-go at every node
(\cref{prop:tree-exact}), which is the point at which the outer approximation
of stochastic dual dynamic programming becomes unnecessary rather than merely
accurate. A corollary of the same factorisation
(\cref{cor:rolling-exact}) makes those curves operational: a window-by-window
policy driven by the true cost-to-go has \emph{exactly zero} truncation
error, so the gap of a rolling policy is an information error and nothing
else.
\item \emph{An open-source implementation and a progressive benchmark}
(\cref{sec:package,sec:experiments}): the \pkg{} package
and a comparison, under both end-to-end and solve-only accountings, against
sparse HiGHS and Gurobi formulations --- time, peak memory, objective
accuracy, feasibility --- a verified reimplementation of the specialised
RAP--NC comparator, and a catalogue-size law that predicts, instance by
instance, which reader regime wins.
\end{enumerate}

\paragraph{Organisation.}
The paper is organised by the dimension of the carried state. \emph{Part I}
develops the one-dimensional transform: \cref{sec:flt} builds the engine and
\cref{sec:longitudinal} the longitudinal solver, after which three sections
read the value function it returns --- \cref{sec:investment} turns it into
investment oracles for the energy and power capacity of a small number of
devices, \cref{sec:stochastic} averages finitely many futures inside the same
class and derives an exact recourse algorithm, at one cut and then on a
scenario tree, and \cref{sec:injection} supplies the hourly cost curves the
recursion reads, by tracing the exact injection value of a network dispatch.
\emph{Part II} widens the carried state: \cref{sec:pair} is the bivariate
block, its directional-stationarity theorem, its $d$-dimensional extension
and its differentiated capacity marginals. \emph{Part III} is implementation
and evidence: \cref{sec:package} describes the package,
\cref{sec:experiments} reports the computational study --- grouped by the
same two regimes, and closing on the catalogue law that decides between them
--- and \cref{sec:conclusion} concludes. The appendices carry the proofs of
the body's results (\cref{app:proofs,app:stat,app:statd}), the coordination
layer and the system-scale study that this paper does not claim
(\cref{app:coordination}), the parametric-simplex engineering and the
campaign and measurement conventions.

\subsection{Related work and positioning}
\label{subsec:related}

\Cref{tab:positioning} summarises how the method relates to its closest
neighbours; the paragraphs below give the detail.

\begin{table}[t]
\centering
\caption{Positioning among exact and generic alternatives for long
temporally coupled convex programs.}
\label{tab:positioning}
\footnotesize
\begin{tabularx}{\linewidth}{@{}p{2.6cm}Xp{2.5cm}Xp{2.6cm}@{}}
\toprule
method & problem class & complexity & output & key limitation\\
\midrule
\pkg{} (this paper) & longitudinal convex, PWL/PWQ costs
& $O(m{+}n)$ per convolution; output-sensitive totals
(\cref{rem:tseng-luo}) & complete value functions, trajectory, marginal and
parametric curves & scalar cumulative state; convexity\\
\addlinespace
serial algorithm of Tseng--Luo~\cite{Tseng1996} & serially nested
inf-convolutions; series--parallel flows & $\Theta(N\log N)$, optimal &
optimal value & no implementation reported; no parametric output\\
\addlinespace
MDA~\cite{Vidal2019} & RAP--NC, continuous and integer & $O(n\log m)$
linear/quadratic & optimal allocation for fixed bounds & no value
functions\\
\addlinespace
tabulated DP~\cite{Maly1995,Ferrero1998,Brown2008a} & general scalar-state
& $O(T\cdot\text{grid})$ & approximate values & discretisation error\\
\addlinespace
conjugate DP on grids~\cite{Kolarijani2020} & input-affine dynamics,
separable costs & $\tilde O(X{+}Y{+}Z)$ per step & sampled cost-to-go
& discretisation error; policy by enumeration\\
\addlinespace
sparse simplex, interior point~\cite{HiGHS,Gurobi} & general LP/QP & --- &
one primal--dual pair & global model held in memory\\
\bottomrule
\end{tabularx}
\end{table}

\paragraph{Storage scheduling and its generic formulations.}
Deterministic arbitrage and planning models commonly optimise charging,
discharging and state of charge against a known price or system
trajectory~\cite{Korpaas2003,Sioshansi2009}, and linear programming is
attractive because it is exact for the usual price-taking model and easily
absorbs engineering constraints. Modern revised-simplex and interior-point
solvers exploit sparse state equations, presolve and advanced linear
algebra; HiGHS provides a state-of-the-art open-source dual revised
simplex~\cite{HiGHS} and Gurobi is the natural commercial
reference~\cite{Gurobi}. Their drawback is not ignored sparsity --- the
path equations $s_t-s_{t-1}=x_t$ already make a fair LP formulation linear
in the horizon --- but that a general solver must retain and process a
global model to return one primal--dual solution. When the same temporal
structure extends over very long horizons, is solved for thousands of
capacities or sites, or becomes one block of a larger decomposition, the
factorisations and model storage remain substantial, and the complete
cost-to-reach and marginal-value curves are not native outputs of one
solve. Classical dynamic programming attacks the temporal dimension
directly but tabulates the state, with cost and accuracy set by the
grid~\cite{Maly1995,Ferrero1998,Brown2008a}. Approximate and stochastic
dynamic programming address the sequential regime --- bidding or control
under imperfect information along the whole horizon~\cite{Jiang2015,Powell2011};
recent analytical stochastic methods exploit piecewise-linear marginal value
functions but still approximate them on a grid over short operating
horizons~\cite{Zheng2022}. Our earlier conference paper~\cite{Girard2014}
showed that the fast Legendre transform manipulates an exact convex
piecewise-linear representation of the value function instead. The
uncertainty treated in \cref{sec:stochastic} is narrower than the sequential
regime and is solved exactly: recourse over a finite scenario set, whose
expected cost-to-go is assembled from the very value functions that
representation already returns. That places the section next
to the decomposition family rather than the tabulation family --- the
L-shaped method builds the recourse function from below by iterated
subgradient cuts, and stochastic dual dynamic
programming~\cite{PereiraPinto1991} extends the same outer approximation to
multistage trees --- with the approximation becoming unnecessary on a scalar
cumulative state. That last clause is not left as an assertion:
\cref{prop:tree-exact} carries it through the alternation of expectation and
elimination that a tree forces, and \cref{subsec:tree-bench} measures it.
Cut-based methods keep their advantage exactly where ours stops: vector
states, and stages whose block constraints break the longitudinal structure.

\paragraph{Nested resource allocation.}
The longitudinal class coincides with the resource-allocation problem with
nested lower and upper constraints (RAP--NC) studied by Vidal, Gribel and
Jaillet~\cite{Vidal2019}. Their monotonic divide-and-conquer algorithm
(MDA) needs no gradients, differentiability or strict convexity, treats
continuous and integer variables, attains $O(n\log m)$ for linear and
quadratic costs, and solves million-variable instances in seconds; for
general convex oracle costs it returns an $\epsilon$-proximate continuous
solution in $O(n\log m\log(nB/\epsilon))$. It is therefore a direct
specialised competitor for the longitudinal kernel, with a mechanism ---
recursively fixing boundary sums and exploiting monotonicity of the
resulting allocations --- distinct from Bellman elimination and infimal
convolution, and broader than the present implementation for integer
decisions. The distinction is the output: MDA returns an optimal allocation
for fixed nested bounds, whereas \pkg{} propagates the complete forward and
backward value functions that the investment and coordination layers of
\cref{sec:investment,sec:coordination} consume. Wherever only an optimal
schedule is needed, the two are interchangeable
(\cref{rem:oracle-agnostic}); a direct benchmark is specified in
\cref{subsec:kernel-bench}.

\paragraph{Computational convex analysis.}
Fast conjugation predates its use here as an optimisation recursion, and
the field divides into two frameworks that differ in what they store.
\emph{Fast transform} algorithms store sampled points: Lucet's linear-time
Legendre transform takes the finite-difference slopes of a convex hull and
merges them with a prescribed slope set, attaining $O(n+m)$ on the line
and, by factorising the supremum one coordinate at a time, in $\R^d$ as
well~\cite[\S4]{Lucet1997}. What such an algorithm returns is the
\emph{discrete} transform: it converges to the conjugate at rate $O(n^{-2})$
on twice-differentiable data, and the slopes at which the conjugate is
evaluated must be chosen before the computation. These primitives serve in
distance transforms, Hamilton--Jacobi equations, convex envelopes, max-plus
algebra and multifractal analysis~\cite{Lucet2010}. They have also been
turned into dynamic programming. Conjugate duality was recognised early as
a way to trade the Bellman minimisation for an
addition~\cite{KleinMorin1991}; Carpio and Kamihigashi used the linear-time
transform for a fast value iteration~\cite{CarpioKamihigashi2020}; and
Kolarijani and Mohajerin Esfahani~\cite{Kolarijani2020} built a discrete
conjugate dynamic program for input-affine dynamics whose step costs
$\tilde O(X+Y+Z)$ in the sizes of the primal and dual grids. Storage
operation, which our conference paper~\cite{Girard2014} already treated by
exact conjugation, is exactly their separable setting in dimension one, and
their recursion is then the restriction--convolution recursion of
\cref{sec:longitudinal}; what differs is the representation. Their
cost-to-go is sampled on prescribed grids, the result carries
discretisation error bounds, and the policy is recovered afterwards by
enumeration over the inputs --- its extraction from the recursion is listed
there as open, and is what \cref{lem:recovery} supplies. The \emph{piecewise
linear-quadratic} framework stores the function itself --- a finite list of
pieces, each at most quadratic --- and computes the same transforms exactly
and in linear time, with the dual domain emerging from the computation
instead of being prescribed~\cite{LucetBauschkeTrienis2009}; related work
treats infimal convolutions of structured convex quadratic
functions~\cite{Bayon2012}. \Cref{sec:flt} belongs to the second framework,
and the distinction is not cosmetic: no grid, sample or tolerance enters
the definition of any result below. Our contribution is not the first exact
conjugation but its integration --- with restrictions interleaved at every
step, deterministic trajectory recovery, degeneracy handling and retained
value functions --- into a complete longitudinal optimisation engine.

\paragraph{The optimal serial algorithm.}
Tseng and Luo~\cite{Tseng1996} studied nested sums and infimal convolutions
of convex piecewise-linear functions, obtained an $O(N\log N)$ algorithm
for serial nesting with a matching comparison lower bound, and extended the
construction to series--parallel linear-cost network flows. This is the
closest algorithmic ancestor of our work, and it has received little
attention in energy optimisation, where storage is still predominantly
approached through generic LP/MILP formulations, state-discretised DP or
approximate value functions. \Cref{rem:tseng-luo} states precisely what we
claim relative to their bound; the contribution here is
software-and-composition, not asymptotic novelty --- Vidal et
al.~\cite{Vidal2019} match the bound in the linear case with $m=O(T)$.

\paragraph{Monotropic programming and parametric linear programming.}
The class sits within Rockafellar's monotropic optimisation and convex
network-flow framework~\cite{Rockafellar1984}, which clarifies both scope
and limitation: the method exploits a serial elimination order with a
scalar separator, and meshed or multi-state couplings require either a
higher-dimensional value function or a coordination algorithm around
several one-dimensional oracles. On the spatial side, tracing a value
function along a scalar right-hand side is classical parametric linear
programming~\cite{Gal1995,AdlerMonteiro1992}; parametric network flows have
been studied from shortest-path augmentation~\cite{AhujaBatraGupta1984} to
complete parametric minimum-cost-flow paths under strictly convex
quadratics~\cite{KlimmWarode2022}. The number of critical regions can be
exponential for general LPs~\cite{Murty1980}, while total unimodularity
keeps the transport traces used here short and cheap to
update~\cite{Schrijver1986}. In power systems the resulting object is not
new: the locational marginal price is a piecewise-constant, non-decreasing
function of the load, with breakpoints at the \emph{critical load levels}
where the optimal basis changes~\cite{BoLi2009}, and multiparametric
programming partitions the parameter space into the corresponding critical
regions~\cite{JiThomasTong2017}. What is contributed here is therefore not
the curve but its \emph{scale and use}: tracing thousands of closely related
bounded hourly LPs by basis transport, so that a full year of exact
injection-cost curves becomes an input to the dynamic program rather than a
by-product of a single market clearing --- developed in \cref{sec:injection}.

\part{The transform in dimension one}

\section{A fast Legendre-transform engine for explicit convex functions}
\label{sec:flt}

Every algorithm in this paper manipulates closed proper convex functions of
one real variable as explicit, finite data. This section fixes the
representations, the five operations applied to them, and their cost; one
corollary records that positively weighted combinations --- the form taken
by an expectation over finitely many scenarios --- stay in the class at
merge cost. The
arithmetic contract under which the word \emph{exact} carries over to an
implementation is recorded with the package in \cref{sec:package}. The
following sections compose these operations into a longitudinal solver
(\cref{sec:longitudinal}), investment and injection oracles
(\cref{sec:investment,sec:injection}) and coordination algorithms
(\cref{sec:coordination}).

\subsection{Representations and conjugation}
\label{subsec:representations}

Let $f:\R\to\R\cup\{+\infty\}$ be closed, proper and convex. Its
Legendre--Fenchel conjugate is
\begin{equation}
  \conj f(p)=\sup_{x\in\R}\{px-f(x)\}.
  \label{eq:conjugate}
\end{equation}
For such functions $f^{**}=f$, and the subdifferentials are inverse
relations: $p\in\partial f(x)$ if and only if $x\in\partial f^*(p)$. In one
dimension, convexity orders both breakpoint positions and slopes. This
monotonicity is the computational resource used throughout.

A convex piecewise-linear (PWL) function is represented, up to one anchor
value, by its domain endpoints and the ordered jumps of its derivative;
equivalently, one stores a monotone relation between positions and slopes.
Under conjugation the roles of the two coordinates are exchanged:
breakpoints of $f$ become affine pieces of $\conj f$, affine pieces become
breakpoints, and the endpoints of $\operatorname{dom}f$ become the
asymptotic slopes of $\conj f$ (\cref{fig:flt-pipeline}). Repeated slopes
and vertical parts of the subgradient graph are merged canonically; they are
degeneracies of the description, not events to be removed by perturbing the
problem. Canonical merging leaves the represented function unchanged --- a
closed convex function is determined by its epigraph, and merging duplicate
records rewrites the description, not the epigraph.

For convex piecewise-quadratic (PWQ) functions, each open piece carries an
affine derivative $p=ax+b$ with $a\ge0$. Conjugation inverts this monotone
affine relation on quadratic pieces and exchanges linear pieces with slope
jumps. The PWL representation is recovered when every $a$ vanishes.

Both are the convex members of the piecewise linear-quadratic (PLQ) class
--- functions whose domain is a finite union of polyhedra on each of which
the function is at most quadratic --- which is closed under addition,
positive scaling, conjugacy and the Moreau
envelope~\cite[Thm.~11.14]{RockafellarWets1998} and for which exact
linear-time algorithms were introduced by Lucet, Bauschke and
Trienis~\cite{LucetBauschkeTrienis2009}. Two properties of the
representation used here are worth naming, because the operator costs of
\cref{subsec:complexity} rest on them. First, convexity is an
\emph{invariant} rather than a restored property: a function is stored by
the non-negative jumps of its derivative, every operation below adds
non-negative jumps, and no convex hull is ever computed --- the step that
opens the linear-time Legendre transform~\cite{Lucet1997} has no
counterpart in \pkg{}. Second, the representation is \emph{self-dual}:
conjugation transposes the two coordinates of each derivative record, so it
is an in-place linear pass rather than a rebuild.

\subsection{Five closed operations}
\label{subsec:operations}

The longitudinal algorithms use only the following operations:
\begin{align}
\text{pointwise sum:}\quad &(f+g)(x)=f(x)+g(x),\\
\text{restriction:}\quad &f|_{[a,b]}(x)=f(x)+\delta_{[a,b]}(x),\\
\text{reflection/shift:}\quad &(R_yf)(x)=f(y-x),\\
\text{infimal convolution:}\quad &(f\infconv g)(z)=\inf_x\{f(z-x)+g(x)\},\\
\text{minimisation:}\quad &\argminop_x f(x).
\end{align}
Here $\delta_{[a,b]}$ is zero on the interval and $+\infty$ outside.

The central identity is
\begin{equation}
  f\infconv g=(\conj f+\conj g)^{*},
  \label{eq:flt-infconv}
\end{equation}
valid under a standard closedness condition that is automatic for the
polyhedral functions used here whenever the convolution is proper. An
apparently continuous allocation problem therefore becomes: conjugate the
operands, merge their ordered derivative descriptions under addition, and
conjugate back. \Cref{fig:flt-pipeline} follows one convolution through this
pipeline. This is the fast Legendre-transform kernel; at no point does a
grid, a sample or a tolerance enter the definition of the result.

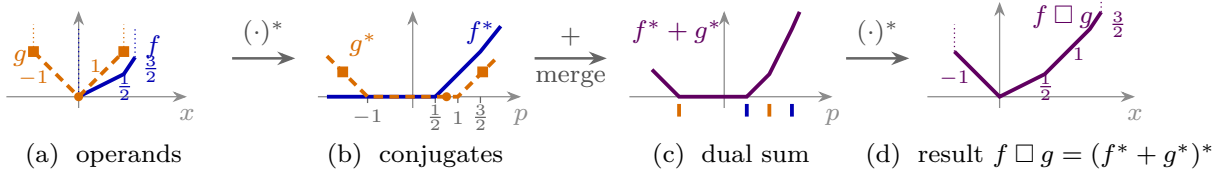
\begin{figure}[t]
\centering
\resizebox{\linewidth}{!}{%
\begin{tikzpicture}[
  x=0.5cm, y=0.5cm,
  axisline/.style={-{Stealth[length=1.6mm]}, thin, black!45},
  fplot/.style={very thick, blue!70!black},
  gplot/.style={very thick, orange!85!black, dash pattern=on 3pt off 2pt},
  hplot/.style={very thick, violet!75!black},
  wall/.style={densely dotted, thin},
  bp/.style={circle, fill, inner sep=1.0pt},
  ep/.style={rectangle, draw, thick, fill=white, inner sep=1.3pt},
  slbl/.style={font=\tiny},
  lbl/.style={font=\scriptsize},
  oplbl/.style={font=\scriptsize, align=center, black!70}
]
\begin{scope}
  \draw[axisline] (-1.6,0) -- (2.4,0) node[below, lbl] {$x$};
  \draw[axisline] (0,-0.25) -- (0,1.9);
  \draw[wall, blue!70!black] (0,0) -- (0,1.5) (1.25,0.875) -- (1.25,1.5);
  \draw[fplot] (0,0) -- (1,0.5) -- (1.25,0.875);
  \node[lbl, blue!70!black] at (1.65,1.1) {$f$};
  \node[slbl, blue!70!black] at (1.02,0.22) {$\tfrac12$};
  \node[slbl, blue!70!black] at (1.6,0.62) {$\tfrac32$};
  \draw[wall, orange!85!black] (-1,1) -- (-1,1.6) (1,1) -- (1,1.6);
  \draw[gplot] (-1,1) -- (0,0) -- (1,1);
  \node[lbl, orange!85!black] at (-1.3,0.85) {$g$};
  \node[slbl, orange!85!black] at (-1.0,0.4) {$-1$};
  \node[slbl, orange!85!black] at (0.38,0.68) {$1$};
  \node[bp, orange!85!black] at (0,0) {};
  \node[ep, orange!85!black] at (-1,1) {};
  \node[ep, orange!85!black] at (1,1) {};
  \node[lbl] at (0.55,-1.3) {(a)\; operands};
\end{scope}
\draw[-{Stealth[length=2mm]}, thick, black!60] (3.4,0.85) -- (4.8,0.85)
  node[midway, above, oplbl] {$(\cdot)^{*}$};
\begin{scope}[shift={(7.4,0)}]
  \draw[axisline] (-2.1,0) -- (2.4,0) node[below, lbl] {$p$};
  \draw[axisline] (0,-0.25) -- (0,1.9);
  \draw[black!60] (-1,0) -- (-1,-0.14) (0.5,0) -- (0.5,-0.14)
                  (1,0) -- (1,-0.14) (1.5,0) -- (1.5,-0.14);
  \node[slbl, black!60] at (-1,-0.46) {$-1$};
  \node[slbl, black!60] at (0.5,-0.46) {$\tfrac12$};
  \node[slbl, black!60] at (1,-0.46) {$1$};
  \node[slbl, black!60] at (1.5,-0.46) {$\tfrac32$};
  \draw[fplot] (-1.9,0) -- (0.5,0) -- (1.5,1) -- (1.95,1.5625);
  \node[lbl, blue!70!black] at (1.5,1.42) {$\conj f$};
  \draw[gplot] (-1.9,0.9) -- (-1,0) -- (1,0) -- (1.9,0.9);
  \node[lbl, orange!85!black] at (-1.1,1.15) {$\conj g$};
  \node[bp, orange!85!black] at (0.75,0) {};
  \node[ep, orange!85!black] at (-1.55,0.55) {};
  \node[ep, orange!85!black] at (1.55,0.55) {};
  \node[lbl] at (0.05,-1.3) {(b)\; conjugates};
\end{scope}
\draw[-{Stealth[length=2mm]}, thick, black!60] (10.1,0.85) -- (11.7,0.85)
  node[midway, above, oplbl] {$+$}
  node[midway, below, oplbl] {merge};
\begin{scope}[shift={(14.3,0)}]
  \draw[axisline] (-2.1,0) -- (2.4,0) node[below, lbl] {$p$};
  \draw[axisline] (0,-0.25) -- (0,1.9);
  \draw[hplot] (-1.6,0.6) -- (-1,0) -- (0.5,0) -- (1,0.5) -- (1.5,1.5)
               -- (1.67,1.88);
  \node[lbl, violet!75!black] at (-1.05,1.45) {$\conj f+\conj g$};
  \draw[very thick, orange!85!black] (-1,-0.16) -- (-1,-0.44);
  \draw[very thick, blue!70!black]   (0.5,-0.16) -- (0.5,-0.44);
  \draw[very thick, orange!85!black] (1,-0.16) -- (1,-0.44);
  \draw[very thick, blue!70!black]   (1.5,-0.16) -- (1.5,-0.44);
  \node[lbl] at (0.15,-1.3) {(c)\; dual sum};
\end{scope}
\draw[-{Stealth[length=2mm]}, thick, black!60] (17.0,0.85) -- (18.4,0.85)
  node[midway, above, oplbl] {$(\cdot)^{*}$};
\begin{scope}[shift={(20.4,0)}]
  \draw[axisline] (-1.6,0) -- (3.0,0) node[below, lbl] {$x$};
  \draw[axisline] (0,-0.25) -- (0,2.1);
  \draw[wall, violet!75!black] (-1,1) -- (-1,1.6) (2.25,1.875) -- (2.25,2.1);
  \draw[hplot] (-1,1) -- (0,0) -- (1,0.5) -- (2,1.5) -- (2.25,1.875);
  \node[lbl, violet!75!black] at (1.4,1.78) {$f\infconv g$};
  \node[slbl, violet!75!black] at (-1.0,0.4) {$-1$};
  \node[slbl, violet!75!black] at (1.02,0.22) {$\tfrac12$};
  \node[slbl, violet!75!black] at (1.75,0.95) {$1$};
  \node[slbl, violet!75!black] at (2.62,1.6) {$\tfrac32$};
  \node[lbl] at (0.9,-1.3) {(d)\; result $f\infconv g=(\conj f+\conj g)^{*}$};
\end{scope}
\end{tikzpicture}%
}
\caption{One infimal convolution evaluated in the dual, on the two operands
$f$ (blue, solid) and $g$ (orange, dashed), whose four piece slopes
$-1,\tfrac12,1,\tfrac32$ are all distinct. Conjugation exchanges the two
coordinates of every derivative record: the slope written on each piece in
(a) becomes an axis position in (b); following $g$, its kink (dot) becomes
the flat piece of $\conj g$ (dot), and the endpoints of
$\operatorname{dom}g$ (squares) become the asymptotic slopes of $\conj g$
(squares). Addition of the conjugates is an ordered merge of the two record
lists: in (c) the ticks mark the four records of the merged list, coloured
by operand --- the two lists interleave. Conjugating back yields
$f\infconv g$, whose pieces replay the merged records in slope order.
Throughout, the objects are finite ordered lists of breakpoints and slopes,
never grids or samples.}
\label{fig:flt-pipeline}
\end{figure}

\subsection{Closure and operator cost}
\label{subsec:complexity}

\begin{proposition}[Closure]
\label{prop:closure}
The five operations map closed proper convex PWL functions to closed proper
convex PWL functions whenever the result is proper, and the infimum defining
$f\infconv g$ is then attained for every point of its domain. The same
statements hold for the implemented convex PWQ class.
\end{proposition}

\begin{corollary}[Closure under positive combinations]
\label{cor:poscomb}
Let $p_w>0$ and let $f_w$ be closed proper convex PWL with $m_w$ derivative
records, $w=1,\dots,K$, with a common point in the intersection of their
domains. Then $\sum_{w=1}^{K}p_wf_w$ is closed proper convex PWL. It is
computed by one $K$-way ordered merge of the derivative record lists and
carries at most $\sum_{w}m_w+O(1)$ records.
\end{corollary}

\begin{proposition}[Output-sensitive operator cost]
\label{prop:opcost}
Let $f$ and $g$ be given by canonical ordered representations with $m$ and
$n$ derivative records. Conjugation costs $O(m)$. Pointwise addition and
infimal convolution cost $O(m+n)$ with array-backed storage, and
$O\bigl((m+n)\log(m+n)\bigr)$ with an ordered-map container. Every output
carries at most $m+n+O(1)$ records.
\end{proposition}

The three proofs are collected in \cref{app:proofs}.

\begin{remark}[Horizon totals and the optimal serial algorithm]
\label{rem:tseng-luo}
\Cref{prop:opcost} bounds one operation. Over a horizon of $T$ steps in
which the running value function carries $M_t$ records and the step cost
$k_t$ records, the incremental strategy implemented in \pkg{} --- rewrite
the value function at every step --- costs
$O\bigl(\sum_t(M_{t-1}+k_t)\bigr)$. Since $M_t\le M_{t-1}+k_t+O(1)$, the
worst case is $O(TN)$ with $N=\sum_tk_t$. Tseng and Luo~\cite{Tseng1996}
attain $\Theta(N\log N)$ for the complete serially nested expression with a
global balanced structure, together with a matching comparison lower bound;
we do not claim their bound for the incremental container. In the
experiments of \cref{sec:experiments}, the envelope restrictions interleaved
with every convolution prune $M_t$, and measured totals scale near-linearly
in $T$; \cref{sec:package} describes the array backend that realises the
merge costs above. The same accounting is what the computational study calls
the catalogue law: one chain pays its carried record counts, essentially flat
in the size of the step-cost catalogue, while the planar two-store step of
\cref{subsec:two-store-conjugate} pays the face count of a fixed arrangement,
$F\approx|\mathrm{cat}|^2$ --- \cref{fig:catalogue-law} measures both regimes
across two orders of magnitude of catalogue.
\end{remark}

\section{An exact solver for longitudinal convex programs}
\label{sec:longitudinal}

\subsection{Problem class}
\label{subsec:class}

Let $x_t$ be a local decision with separable convex cost, and let its
running sum be bounded from below and above:
\begin{equation}
\begin{aligned}
\min_{x\in\R^T}\quad &\sum_{t=1}^{T}c_t(x_t),\\
\text{s.t.}\quad &\underline x_t\le x_t\le\overline x_t,\\
&\underline s_t\le s_0+\sum_{\tau\le t}x_\tau\le\overline s_t,
\qquad t=1,\ldots,T.
\end{aligned}
\label{eq:longitudinal}
\end{equation}
We call this class \emph{longitudinal}: arbitrarily many variables, coupled
only through one bounded cumulative state. The solver of this section is
reused unchanged across operations planning; only the data and their
interpretation change:
\begin{itemize}
\item \emph{energy storage} --- $x_t$ a charging power, $s_t$ a state of
charge, its operating regime carried by the step cost (below);
\item \emph{inventory and production smoothing} --- $x_t$ production minus
demand, $s_t$ an inventory; production, overtime and procurement costs;
safety stocks and warehouse limits as envelopes;
\item \emph{reservoir release} --- $x_t$ a release net of known inflows,
$s_t$ a reservoir level; PWL benefit or penalty curves; rule curves as
envelopes;
\item \emph{flexible demand and charging} --- $x_t$ a consumption, $s_t$
the energy served; deadline envelopes from release and due dates; a
terminal singleton for exact delivery;
\item \emph{path flow with inventories} --- $x_t$ a local supply, $s_t$
the flow carried between consecutive nodes; arc capacities as state
bounds, the simplest series--parallel network reading~\cite{Tseng1996}.
\end{itemize}
Storage is the richest recurring example, not the boundary of the method.

\paragraph{Operating physics as step costs.}
For storage the operating regime of a period lives entirely in the shape of
$c_t$, a convex PWL (or PWQ) function of the charging power $x_t$: a linear
$c_t(x)=\pi_t x$ prices a \emph{price taker} at the spot price $\pi_t$; a
\emph{round-trip loss} is the two-slope V with a kink at $x=0$, charging at
$\pi_t/\eta_{\mathrm{in}}$ and discharging at $\pi_t\eta_{\mathrm{out}}$; a
\emph{network tariff} adds one piece on each side of $x=0$ --- a per-MWh
charge on withdrawal and on injection --- still convex PWL; and \emph{price
making} is itself convex PWL when the period's supply is a stack of units with
given marginal costs and available capacities --- charging or discharging
moves storage along the merit order, the exact injection curve of a network LP
(\cref{sec:injection}) --- reducing to the convex quadratic
$c_t(x)=\pi_t x+\tfrac12\delta_t x^2$ (PWQ) only for a smooth affine
inverse-demand law. Self-discharge is the exception:
a geometric decay $s_t=\rho\,s_{t-1}+x_t$ ($\rho\in(0,1]$) is not a cost but
an affine twist of the state recursion, reduced to~\eqref{eq:longitudinal} by
the change of variable $s_t\mapsto s_t/\rho^{t}$, under which the
value-function readout of \cref{subsec:joint-bundle} carries over unchanged.

In the resource-allocation literature the continuous class coincides with
RAP--NC, the allocation problem with nested lower and upper
constraints~\cite{Vidal2019}.

\subsection{Forward recursion}
\label{subsec:forward}

For problem~\eqref{eq:longitudinal}, define the cost of reaching state $s$
after step $t$:
\begin{equation}
V_t(s)=\inf\left\{\sum_{\tau=1}^{t}c_\tau(x_\tau):
s=s_0+\sum_{\tau=1}^{t}x_\tau,\;
\underline x_\tau\le x_\tau\le\overline x_\tau,\;
\underline s_\tau\le s_\tau\le\overline s_\tau\right\}.
\end{equation}
Let $\tilde c_t=c_t+\delta_{[\underline x_t,\overline x_t]}$ and let
$\mathcal R_t f=f+\delta_{[\underline s_t,\overline s_t]}$. Then
\begin{equation}
  V_t=\mathcal R_t\bigl(V_{t-1}\infconv\tilde c_t\bigr),
  \qquad V_0=\delta_{\{s_0\}},
  \label{eq:longitudinal-recursion}
\end{equation}
with terminal optimum $\min_s V_T(s)$. This is exact variable elimination,
whether one calls it dynamic programming or not: every step is one
restriction and one fast infimal convolution, so \cref{prop:closure} keeps
each $V_t$ closed convex PWL (or PWQ) and \cref{prop:opcost} prices each
step. \Cref{fig:recursion} shows one step. The convolution widens the
reachable interval and merges the slopes of the step cost into the value
function; the restriction cuts the result back to the envelope, which is
also what keeps representation sizes small in practice
(\cref{rem:tseng-luo}).

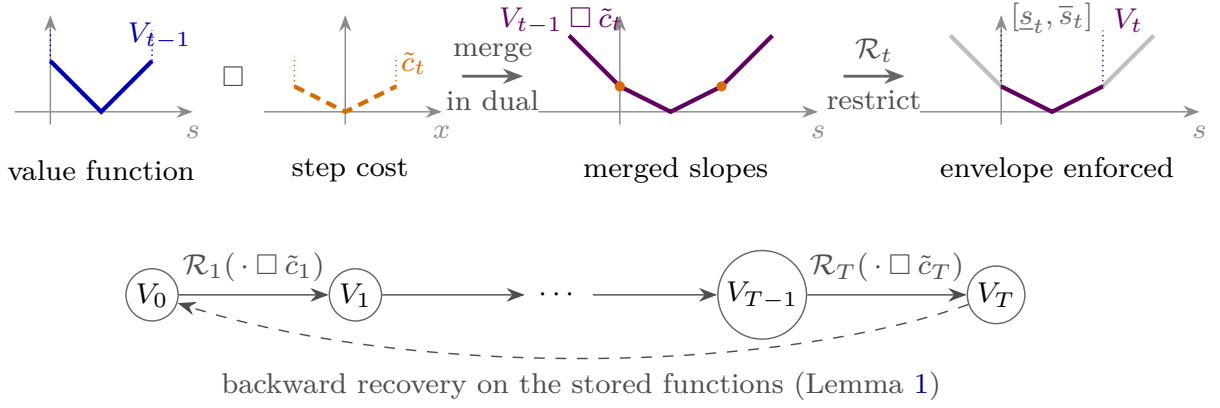
\begin{figure}[t]
\centering
\resizebox{\linewidth}{!}{%
\begin{tikzpicture}[
  x=0.5cm, y=0.5cm,
  axisline/.style={-{Stealth[length=1.6mm]}, thin, black!45},
  vplot/.style={very thick, blue!70!black},
  cplot/.style={very thick, orange!85!black, dash pattern=on 3pt off 2pt},
  rplot/.style={very thick, violet!75!black},
  ghost/.style={very thick, black!25},
  wall/.style={densely dotted, thin},
  bp/.style={circle, fill, inner sep=1.0pt},
  lbl/.style={font=\scriptsize},
  oplbl/.style={font=\scriptsize, align=center, black!70},
  chainnode/.style={font=\scriptsize, draw=black!60, circle, inner sep=1.2pt}
]
\begin{scope}
  \draw[axisline] (-0.7,0) -- (2.8,0) node[below, lbl] {$s$};
  \draw[axisline] (0,-0.25) -- (0,1.9);
  \draw[wall, blue!70!black] (0,1) -- (0,1.6) (2,1) -- (2,1.6);
  \draw[vplot] (0,1) -- (1,0) -- (2,1);
  \node[lbl, blue!70!black] at (2.15,1.5) {$V_{t-1}$};
  \node[lbl] at (1.0,-1.15) {value function};
\end{scope}
\node[oplbl] at (3.6,0.7) {$\infconv$};
\begin{scope}[shift={(5.8,0)}]
  \draw[axisline] (-1.6,0) -- (1.9,0) node[below, lbl] {$x$};
  \draw[axisline] (0,-0.25) -- (0,1.9);
  \draw[wall, orange!85!black] (-1,0.5) -- (-1,1.1) (1,0.5) -- (1,1.1);
  \draw[cplot] (-1,0.5) -- (0,0) -- (1,0.5);
  \node[lbl, orange!85!black] at (1.35,0.95) {$\tilde c_t$};
  \node[lbl] at (0.1,-1.15) {step cost};
\end{scope}
\draw[-{Stealth[length=2mm]}, thick, black!60] (8.1,0.7) -- (9.3,0.7)
  node[midway, above, oplbl] {merge}
  node[midway, below, oplbl] {in dual};
\begin{scope}[shift={(11.2,0)}]
  \draw[axisline] (-1.6,0) -- (3.9,0) node[below, lbl] {$s$};
  \draw[axisline] (0,-0.25) -- (0,1.9);
  \draw[rplot] (-1,1.5) -- (0,0.5) -- (1,0) -- (2,0.5) -- (3,1.5);
  \node[bp, orange!85!black] at (0,0.5) {};
  \node[bp, orange!85!black] at (2,0.5) {};
  \node[lbl, violet!75!black] at (-1.15,1.8) {$V_{t-1}\infconv\tilde c_t$};
  \node[lbl] at (1.1,-1.15) {merged slopes};
\end{scope}
\draw[-{Stealth[length=2mm]}, thick, black!60] (15.6,0.7) -- (16.8,0.7)
  node[midway, above, oplbl] {$\mathcal R_t$}
  node[midway, below, oplbl] {restrict};
\begin{scope}[shift={(18.7,0)}]
  \draw[axisline] (-1.6,0) -- (3.9,0) node[below, lbl] {$s$};
  \draw[axisline] (0,-0.25) -- (0,1.9);
  \draw[ghost] (-1,1.5) -- (0,0.5) (2,0.5) -- (3,1.5);
  \draw[wall, violet!75!black] (0,0.5) -- (0,1.6) (2,0.5) -- (2,1.6);
  \draw[rplot] (0,0.5) -- (1,0) -- (2,0.5);
  \node[lbl, violet!75!black] at (2.5,1.8) {$V_t$};
  \node[lbl, black!60] at (1.0,1.85) {$[\underline s_t,\overline s_t]$};
  \node[lbl] at (1.1,-1.15) {envelope enforced};
\end{scope}
\begin{scope}[shift={(0,-3.6)}]
  \node[chainnode] (v0) at (2.0,0) {$V_0$};
  \node[chainnode] (v1) at (6.0,0) {$V_1$};
  \node[font=\scriptsize] (vd) at (10.0,0) {$\cdots$};
  \node[chainnode] (vT1) at (14.0,0) {$V_{T-1}$};
  \node[chainnode] (vT) at (18.6,0) {$V_T$};
  \draw[-{Stealth[length=1.8mm]}, black!70] (v0) -- (v1)
    node[midway, above, oplbl] {$\mathcal R_1(\,\cdot\infconv\tilde c_1)$};
  \draw[-{Stealth[length=1.8mm]}, black!70] (v1) -- (vd);
  \draw[-{Stealth[length=1.8mm]}, black!70] (vd) -- (vT1);
  \draw[-{Stealth[length=1.8mm]}, black!70] (vT1) -- (vT)
    node[midway, above, oplbl] {$\mathcal R_T(\,\cdot\infconv\tilde c_T)$};
  \draw[-{Stealth[length=1.8mm]}, black!70, dashed]
    (vT) .. controls (14,-1.6) and (7,-1.6) .. (v0)
    node[midway, below, oplbl]
    {backward recovery on the stored functions (\cref{lem:recovery})};
\end{scope}
\end{tikzpicture}%
}
\caption{One step of recursion~\eqref{eq:longitudinal-recursion}. The value
function $V_{t-1}$ is convolved with the restricted step cost $\tilde c_t$
--- in the dual, an ordered merge of slopes --- which widens the reachable
interval and inserts the new slopes $\pm$ those of $\tilde c_t$ (middle);
the envelope restriction then cuts the result back to
$[\underline s_t,\overline s_t]$ (right, cut parts in grey). Below: the
forward sweep stores every $V_t$; the backward pass rereads them to split
each state difference optimally (\cref{lem:recovery}).}
\label{fig:recursion}
\end{figure}

\subsection{Deterministic trajectory recovery}
\label{subsec:recovery}

\begin{lemma}[Backward recovery]
\label{lem:recovery}
Fix $t\le T$ and $s\in\operatorname{dom}V_t$. The set
\begin{equation}
U_t(s)=\argminop_{u}\bigl\{V_{t-1}(u)+\tilde c_t(s-u)\bigr\}
\end{equation}
is a nonempty closed interval, and every $u\in U_t(s)$ satisfies
$V_t(s)=V_{t-1}(u)+\tilde c_t(s-u)$. Consequently, starting from any
$s_T^\star\in\argminop_sV_T(s)$ and selecting, for $t=T,\ldots,1$, any
$s_{t-1}^\star\in U_t(s_t^\star)$, the increments
$x_t^\star=s_t^\star-s_{t-1}^\star$ are optimal
for~\eqref{eq:longitudinal}. In particular, any fixed endpoint rule ---
\pkg{} selects the endpoint nearest to $s_t^\star$, i.e.\ the increment of
smallest magnitude --- is optimal, and no tie-breaking perturbation is
needed on flat optimal faces.
\end{lemma}

The proof is in \cref{app:proofs}.
The forward pass stores the record list of every $V_t$; recovery then costs
one search of a stored list per step. \Cref{alg:longitudinal} summarises the
complete solver.

\begin{algorithm}[t]
\caption{Longitudinal solve by restricted fast infimal convolutions}
\label{alg:longitudinal}
\begin{algorithmic}[1]
\Require costs $c_t$, increment bounds $[\underline x_t,\overline x_t]$,
envelopes $[\underline s_t,\overline s_t]$, initial state $s_0$
\State $V_0\gets\delta_{\{s_0\}}$
\For{$t=1,\dots,T$}\Comment{forward pass: exact variable elimination}
  \State $\tilde c_t\gets c_t+\delta_{[\underline x_t,\overline x_t]}$
  \State $V_t\gets\bigl(V_{t-1}\infconv\tilde c_t\bigr)
         +\delta_{[\underline s_t,\overline s_t]}$
         \Comment{conjugate, merge, conjugate back (\cref{prop:opcost})}
  \State store the record list of $V_t$
\EndFor
\State $s_T^\star\gets$ deterministic selection in $\argminop_s V_T(s)$
\For{$t=T,\dots,1$}\Comment{backward pass (\cref{lem:recovery})}
  \State $s_{t-1}^\star\gets$ deterministic selection in $U_t(s_t^\star)$;
         \quad $x_t^\star\gets s_t^\star-s_{t-1}^\star$
\EndFor
\State \Return trajectory $x^\star$, value functions $V_{0},\dots,V_T$
\end{algorithmic}
\end{algorithm}

\subsection{Forward--backward factorisation}
\label{subsec:forward-backward}

The forward sweep is not the only reading of~\eqref{eq:longitudinal}. Cutting
the horizon at any $t$ splits it into a prefix and a suffix that communicate
only through the state at the cut. Write $F_t=V_t$ for the cost-to-arrive of
\cref{subsec:forward}, and define the cost-to-go
\begin{equation}
B_t(s)=\inf\Bigl\{\sum_{\tau>t}c_\tau(x_\tau):
x_{t+1:T}\text{ feasible from }s_t=s\Bigr\}.
\label{eq:cost-to-go}
\end{equation}
The backward recursion initially contains $s+x$ rather than $s-x$;
reflecting the state, $\widehat B_t(u)=B_t(-u)$, restores exactly the same
restriction and infimal-convolution operations. Both passes therefore run on
the same FLT kernel and are priced by the same \cref{prop:opcost}.

\begin{proposition}[Forward--backward identity]
\label{prop:forward-backward}
Let $J$ be the optimal value of~\eqref{eq:longitudinal}. For every cut $t$,
\begin{equation}
J=\min_{s}\{F_t(s)+B_t(s)\},
\qquad
s_t^\star\in\argminop_{s}\{F_t(s)+B_t(s)\},
\end{equation}
the minimiser set is a nonempty closed interval, and pinning the state at any
such $s_t^\star$ leaves the optimal value unchanged.
\end{proposition}

The identity is variable elimination in two directions. It also gives a
useful interpretation of slopes: a subgradient of $F_t$ is the marginal cost
of arriving with an additional unit of state, a subgradient of $B_t$ the
marginal future cost of that unit, and their sum measures the pressure of the
two halves of the horizon on the state at the cut. Both sequels consume this
factorisation --- \cref{sec:stochastic} averages cost-to-go functions over
scenarios at a cut, \cref{sec:investment} reads capacity marginals off the
slopes of the two curves.

One consequence is worth isolating, because it is what makes the value
function an \emph{operational} object and not only a computational one.
Production-grade studies rarely solve a year at once: they operate it segment
by segment, pricing each device at the segment boundary with a usage value.

\begin{corollary}[Zero truncation error]
\label{cor:rolling-exact}
Chain segment solves over a partition of $\{1,\dots,T\}$, giving each segment
the true cost-to-go $B_t$ at its right boundary as terminal valuation. The
concatenated schedule attains the optimal value of~\eqref{eq:longitudinal}.
\end{corollary}

Both proofs are in \cref{app:proofs}.

The corollary is a statement about \emph{error decomposition}, and that is
its use. The gap of a rolling policy driven by an estimate $\widehat B$
splits into an \emph{information} error --- $\widehat B$ against $B$ --- and
a \emph{truncation} error, and \cref{cor:rolling-exact} says the second term
is exactly zero. A policy whose water values are fitted rather than computed
cannot exhibit that split, because it has no exact $B$ to measure against.

\begin{remark}[The usage-value curve is a by-product, and it is isotone by
construction]
\label{rem:usage-value}
$B_t$ \emph{is} the usage-value curve that adequacy studies estimate: its
one-sided derivatives are the state-dependent marginal values of stored
energy, in currency per unit of energy, and they decrease in the state by
convexity. Practice obtains such curves by regressing a cloud of
perfect-foresight duals and then imposing monotonicity and a scarcity floor
by hand~\cite{Wolfgang2009}; here monotonicity is not imposed but inherited
from convexity, and the curve is the output of the same backward pass that
\cref{prop:forward-backward} already requires. \Cref{sec:stochastic} replaces
the scarcity prior by an expectation over futures, which is the principled
version of the same correction.
\end{remark}

\subsection{Complexity and memory}
\label{subsec:solver-cost}

Each forward step is one merge, priced by \cref{prop:opcost}; the horizon
totals, and the pruning role of the interleaved envelope restrictions, are
those analysed in \cref{rem:tseng-luo}. When each period contributes a
bounded number of records, as in the price-taking and merit-order storage
models, the input size is $N=\sum_tk_t=O(T)$. Recovery rereads all forward
value functions, hence $O(\sum_tM_t)$ memory --- contiguous arrays in the
flat backend of \cref{sec:package}; the measured scaling of the complete
solver is reported in \cref{sec:experiments}.

\subsection{What the class excludes}
\label{subsec:exclusions}

Unit commitment, minimum up/down constraints and discrete lot sizes destroy
convexity. Ramping constraints couple consecutive increments in addition to
the cumulative state and require an enlarged state or a separate
coordination. Multiple interacting inventories increase the separator
dimension --- exactly by one per extra inventory, which is why two of them are
still eliminable and \cref{subsec:two-store-conjugate} does eliminate them,
while three are not addressed here. General meshed network constraints do not admit a
one-dimensional serial elimination. These problems may still use the
longitudinal solver as a subproblem --- \cref{sec:coordination} does exactly
this, and a ramping generator is the cleanest instance: its ramp increment
$\Delta g_t\in[-R_{\mathrm{dn}},R_{\mathrm{up}}]$ and cumulative output $g_t$
under box limits are \emph{themselves} a longitudinal chain, so the recursion
prices the unit as a block. That a single ramp-constrained unit admits an
exact dynamic-programming treatment is not new --- it is the setting of
Frangioni and Gentile~\cite{FrangioniGentile2006}, who solve the nonlinear
single-unit commitment problem with ramping, commitment binaries included;
what the observation buys here is only that such a unit enters the
decomposition as one more scalar-state chain, at no change of state
dimension. What the single-chain recursion cannot do is
\emph{couple} that block to storage or to the network --- and note that this is a
statement about the \emph{single} chain: a \emph{pair} of co-located stores is
coupled exactly, by carrying the bivariate value function
(\cref{subsec:two-store-conjugate}), which is elimination and not coordination. Co-located with a
storage, or joined to one across a line, the fleet and the store meet only
through the node's kinked merit-order cost, and the pair then exhibits a
\emph{block-coordinate lock}: each chain is exactly optimal given the other,
yet a fixed point of the alternation need not be the joint optimum --- the
joint optimum is itself a fixed point, but an alternation started cold can
settle at a strictly suboptimal one, because moving either unit alone need
not cross the merit-order kink that a joint move would cross
(\cref{subsec:alternation}). Pricing the block is therefore necessary but not
sufficient; the joint refinement that removes the lock is the coordination
construction of \cref{subsec:masters}, not the recursion above, and it returns
a certificate rather than an unverified fixed point --- at a cost that confines
it to small groups (\cref{subsec:masters,subsec:ramps-bench}).

\begin{remark}[Where exactness fails without convexity]
\label{rem:nonconvex}
Convexity, not piecewise linearity, is the load-bearing hypothesis: the
recursion~\eqref{eq:longitudinal-recursion} is valid variable elimination
for arbitrary costs, and a nonconvex PWL function is still PWL, but the
kernel guarantees of \cref{sec:flt} fail in three distinct ways. First,
the canonical representation no longer exists --- a nonconvex derivative
is not a monotone position--slope relation --- and \pkg{} therefore
rejects such input at construction. Second, conjugation retains only the
convex envelope ($f^{**}=\operatorname{conv}f$); a single global
minimisation would survive, since
$\min f=\min\operatorname{conv}f$, but restrictions do not commute with
the envelope, so from the first restriction applied after a convexified
convolution the propagated value can fall strictly below the true optimum,
and \cref{lem:recovery} fails: $U_t(s)$ need not be an interval, and a
recovered trajectory can cost strictly more than the claimed
$\min_sV_T(s)$. Third, propagating the true nonconvex function instead
forfeits \cref{prop:opcost}: an infimal convolution of nonconvex PWL
operands is a pointwise minimum of shifted sums carrying up to $mn$
pieces, so representations grow multiplicatively along the horizon rather
than additively. The exclusions above are therefore boundaries of the fast
kernel, not of dynamic programming.
\end{remark}

\section{Investment in energy and power capacity}
\label{sec:investment}

The solver of \cref{sec:longitudinal} returns more than the minimiser of one
instance: at every cut it retains the complete convex cost of connecting the
past to the future. This section turns that output into investment oracles
for a single storage. Energy and power capacity enter
\eqref{eq:longitudinal} through right-hand sides, so their marginal
operating values are exact duals recoverable from one solve: the energy
marginal directly from the value-function slopes of the forward--backward
factorisation of \cref{subsec:forward-backward}
(\cref{subsec:param-value,subsec:capacity-parameter}), the power
marginal from a water-value chain read off the optimal trajectory
(\cref{subsec:joint-bundle}); sizing in the
energy capacity is then a Legendre transform along an explicit slope curve
(\cref{subsec:sizing}), and joint energy--power sizing is driven by exact
Benders cuts in a level bundle (\cref{subsec:joint-bundle}). Every object
exchanged is an ordered list of breakpoints and slopes --- precisely what
the engine of \cref{sec:flt} manipulates; the spatial counterpart of these
oracles, the injection-cost curve of one network hour, is developed in
\cref{sec:injection}.

\subsection{The parametric operating value}
\label{subsec:param-value}

Consider a storage with power $p$, energy capacity $E$ and convex per-step
costs $c_t$. Its optimal operating cost is
\begin{equation}
J(E,p)=\min_x\left\{\sum_{t=1}^{T}c_t(x_t):
-p\le x_t\le p,\quad 0\le\sum_{\tau\le t}x_\tau\le E\right\}.
\label{eq:param-value}
\end{equation}
The initial level can be included by shifting the cumulative bounds. This is
an instance of~\eqref{eq:longitudinal} in which the two capacities enter only
as right-hand sides --- $p$ in the increment bounds, $E$ in the cumulative
ones. The forward--backward factorisation of \cref{subsec:forward-backward}
therefore applies at every cut, with the two curves now carrying $(E,p)$
through those right-hand sides:
\begin{equation}
J(E,p)=\min_{s\in[0,E]}\{F_t(s)+B_t(s)\}
\label{eq:param-fb}
\end{equation}
by \cref{prop:forward-backward}. The oracles of this section read the
capacity marginals off the slopes of $F_t$ and $B_t$: a subgradient of $F_t$
is the marginal cost of arriving with an additional unit of energy, a
subgradient of $B_t$ the marginal future cost of that unit.

\subsection{Capacity as a right-hand-side parameter}
\label{subsec:capacity-parameter}

Both $E$ and $p$ enter~\eqref{eq:param-value} through right-hand sides. By
convex perturbation theory, $J$ is convex, non-increasing and PWL in $(E,p)$
when the local costs are PWL. Let $\mu_t^+\ge0$ denote the multiplier of the
upper state bound $s_t\le E$, and $\nu_t^+,\nu_t^-\ge0$ those of the power
bounds. At points of differentiability,
\begin{equation}
\frac{\partial J}{\partial E}=-\sum_t\mu_t^+,
\qquad
\frac{\partial J}{\partial p}=-\sum_t(\nu_t^++\nu_t^-).
\label{eq:rhs-partials}
\end{equation}
We call
\begin{equation}
M_E(E,p)=-\frac{\partial J}{\partial E}=\sum_t\mu_t^+
\label{eq:marginal-capacity}
\end{equation}
the marginal operating value of energy capacity. Convexity of $J$ implies
that $M_E$ is non-increasing in $E$.

The multipliers can be recovered without resolving a dual LP. Perturb only
the upper state bound at time $t$. By \cref{prop:forward-backward}, the
corresponding one-dimensional value is the minimum of $G_t=F_t+B_t$ on the
perturbed interval. The multiplier is thus the inward slope at the upper
endpoint.

\begin{theorem}[Marginal capacity from value-function slopes]
\label{thm:marginal-capacity}
Let $F_t$ and $B_t$ be built at capacity $E$. With one-sided derivatives
chosen consistently with the required subgradient,
\begin{equation}
M_E(E,p)=
\sum_{t:E\in\operatorname{dom}F_t\cap\operatorname{dom}B_t}
\max\{0,-F_t'(E)-B_t'(E)\}.
\label{eq:marginal-slope}
\end{equation}
Times at which $E$ is unreachable are excluded by the domain condition.
\end{theorem}

\begin{proof}[Proof sketch]
Fix $t$ with $E\in\operatorname{dom}F_t\cap\operatorname{dom}B_t$ and
perturb the single constraint $s_t\le E$ to $s_t\le E+\varepsilon$,
$\varepsilon\ge0$, leaving all other bounds unchanged. By
\cref{prop:forward-backward} applied at the cut $t$, the perturbed value is
$\min_{s\in[0,E+\varepsilon]}G_t(s)$, where on $(E,E+\varepsilon]$ the
function $G_t$ prolongs the pre-restriction pieces stored during the two
passes. This value is convex in $\varepsilon$, and its one-sided derivative
at $\varepsilon=0$ equals $-\max\{0,-F_t'(E)-B_t'(E)\}$: enlarging the
interval lowers the minimum exactly when $G_t$ still decreases at its upper
endpoint. This quantity is $-\mu_t^+$ for the multiplier of the constraint
$s_t\le E$, with the one-sided choice matching the selected subgradient at
kinks of $J$. Summing the per-constraint derivatives over $t$ --- exact at
points of differentiability of $J(\cdot,p)$, and with consistent one-sided
choices elsewhere --- gives \eqref{eq:marginal-slope}; unreachable times
contribute nothing, which is the domain guard.
\end{proof}

The domain guard is essential. A common ceiling does not imply that the
device can reach it at every time, and attributing a multiplier to an
unreachable endpoint creates spurious capacity value. Equation
\eqref{eq:marginal-slope} is evaluated on the same ordered slope maps as the
optimisation itself, so offsets discarded during the recursion are
irrelevant. \Cref{sec:experiments} validates the resulting marginal values
against finite differences and LP duals. One construction of this paper
escapes the forward--backward pair: the two-store block of
\cref{sec:pair}. Its marginals are obtained by differentiating the
bivariate recursion itself, once that recursion is available
(\cref{subsec:marginal-joint}).

\subsection{Sizing as a Legendre transform in the capacity variable}
\label{subsec:sizing}

Adding linear investment costs gives the convex sizing problem
\begin{equation}
\min_{E,p\ge0}\ J(E,p)+\kappa_EE+\kappa_Pp.
\label{eq:sizing-problem}
\end{equation}
Fix $p$ and write $J_p=J(\cdot,p)$, extended by $+\infty$ to negative
capacities. Then
\begin{equation}
\min_{E\ge0}\{J_p(E)+\kappa_EE\}=-\conj{J_p}(-\kappa_E),
\label{eq:sizing-conjugate}
\end{equation}
so sizing is one more Legendre transform, taken in the capacity variable
rather than the state; since $J_p$ is non-increasing, $\conj{J_p}$ is
finite precisely on slopes $-\kappa_E\le0$, the price cone. The inversion
rule of \cref{subsec:representations} is the sizing rule:
$E^\star\in\partial\conj{J_p}(-\kappa_E)$, i.e.\
$-\kappa_E\in\partial J_p(E^\star)$ --- build until the marginal operating
value $M_E$ falls to the investment price.

No closed form of $\conj{J_p}$ is available, since $E$ enters every
restriction of the recursion and each evaluation of $J_p$ is one FLT pass.
But a pass at $E$ returns $J_p(E)$ together with the supporting slope
$-M_E(E,p)$ of \cref{thm:marginal-capacity}, hence, by Fenchel's
equality~\cite[\S23]{Rockafellar1970}, one exact record of $\conj{J_p}$;
and intersections of successive support lines jump from one affine piece
to the next. A finite walk thus assembles $\conj{J_p}$ on the visited
slope range without sampling a capacity grid, and its output answers the
fixed-$p$ problem for \emph{every} investment cost in that range at once:
$\kappa_E\mapsto E^\star$ is the staircase inverse of $M_E$, its jumps are
the break-even prices, and a new $\kappa_E$ is a lookup --- the reuse
quantified in \cref{subsec:parametric-reuse}. In the lossless linear
arbitrage case with constant power the staircase is explicit, with layers
at $E=kp$; with nonlinear local costs or time-varying power the
breakpoints are data-dependent but the oracle is unchanged. Minimising
over $p$ afterwards is a second one-dimensional conjugate, of the convex
partial minimum $p\mapsto\min_E\{J(E,p)+\kappa_EE\}$, so the joint
problem~\eqref{eq:sizing-problem} stays inside the engine's univariate
class; \cref{subsec:joint-bundle} makes it algorithmic.

This route is not claimed to beat a capacity-co-optimising LP in every
regime. A large seasonal optimum may require several increasingly
expensive FLT passes, whereas capacity is only one additional variable in
the LP. Its advantage is instead compositional: it uses the same operating
model, identifies exact breakpoints of the represented problem, and
returns the complete price-parametric solution needed for investment
analysis.

\subsection{Joint sizing: exact Benders cuts and a level bundle}
\label{subsec:joint-bundle}

In both variables, \eqref{eq:sizing-problem} is a two-stage investment
problem: a master chooses $(E,p)$ and an operating subproblem prices the
choice, its value being exactly $J$. An oracle call at $(E_k,p_k)$ that
returns $J(E_k,p_k)$ together with a subgradient
$g_k\in\partial J(E_k,p_k)$ --- componentwise $(-M_E,-M_P)$
by~\eqref{eq:rhs-partials} --- yields the cut
\begin{equation}
J(E,p)\;\ge\;J(E_k,p_k)+\big\langle g_k,(E,p)-(E_k,p_k)\big\rangle,
\label{eq:benders-cut}
\end{equation}
tight at the query point: one support plane of the polyhedral function $J$,
the two-variable analogue of the records \cref{subsec:sizing} walks in one
dimension. Because the oracle is exact, these are Benders optimality
cuts~\cite{Benders1962} without relaxation gap, and finitely many of them
reproduce $J$ on any box. That oracle is the coordination scheme of
\cref{sec:coordination} in whichever regime the instance puts it: the exact
single-unit projection of \cref{prop:one-storage} where each group holds one
intertemporal unit, the joint block of \cref{subsec:masters} otherwise.

\pkg{} minimises \eqref{eq:sizing-problem} over the accumulating cut model
with a level bundle~\cite{LemarechalNemirovskiiNesterov1995}
(\cref{alg:level-bundle}): the model minimum is a lower bound, the best
evaluated point an upper bound, and the next query is the Euclidean
projection --- a two-variable QP --- of the incumbent onto a level set of
the model, which stabilises the search without any stepsize schedule.
Since $J$ is polyhedral, the two bounds meet after finitely many cuts.

\begin{algorithm}[t]
\caption{Level-bundle sizing on the exact operating oracle}
\label{alg:level-bundle}
\begin{algorithmic}[1]
\Require oracle returning $J$ and a subgradient, prices
$(\kappa_E,\kappa_P)$, box $[0,E^+]\times[0,p^+]$, level parameter
$\lambda\in(0,1)$, tolerance $\varepsilon$
\State evaluate the oracle at the start point; initialise the cut model
\Repeat
  \State $\mathrm{lb}\gets\min$ of $\kappa_EE+\kappa_Pp+\text{model}$
         over the box \Comment{LP}
  \State $\mathrm{ub}\gets$ best evaluated investment cost;
         incumbent $\gets$ its argument
  \State project the incumbent onto
         $\{\kappa_EE+\kappa_Pp+\text{model}\le
         \mathrm{lb}+\lambda(\mathrm{ub}-\mathrm{lb})\}$ \Comment{QP}
  \State evaluate the oracle there; add the cut~\eqref{eq:benders-cut}
\Until{$\mathrm{ub}-\mathrm{lb}\le\varepsilon\,|\mathrm{ub}|$}
\end{algorithmic}
\end{algorithm}

Two structural facts sharpen the method. For a price-taking store with
linear hourly prices, $J$ is positively homogeneous of degree one ---
scaling $(E,p)$ scales every feasible schedule --- so in the duration
$d=E/p$ the objective is $p\,[J(d,1)+\kappa_Ed+\kappa_P]$: the joint
problem collapses to the one-dimensional search of \cref{subsec:sizing} in
$d$, the scale sitting at zero or at the power bound according to the sign
of the bracket, and no two-dimensional method is needed. Homogeneity breaks
when storage moves the prices, but the same reduction on the unstored
system's hourly marginal prices still yields a duration estimate that
warm-starts the bundle. Second, the cuts must be exact, and $M_P$
resists the route of \cref{thm:marginal-capacity}: no sum of per-step
envelopes can equal $M_P$, because power bounds bind in charge--discharge
\emph{pairs}. With two steps, prices $c_1<c_2$ and the bang--bang optimum,
$M_P=c_2-c_1>0$, yet relaxing either step's bound alone gains nothing ---
the extra megawatt charged cannot be sold, the extra megawatt to sell could
not have been charged. Chord estimates of $M_P$ fail for the same reason,
and extreme $(E,p)$ corners turn them into invalid cuts.

The pairing is resolved by \emph{sharing} one dual selection across all
steps. Recover the balance multipliers $w_t$ --- the water values ---
from a solved trajectory by complementary slackness: $w$ is constant on
blocks of steps separated by binding state cuts, may jump only upward at
a cut where the store is full and only downward where it is empty, is
anchored by the terminal state ($w=0$ interior, $w\ge0$ empty, $w\le0$
full), is pinned into $\partial c_t(x_t^\star)$ by every
power-interior step, and satisfies $w_t\ge c_t'(p)$ at saturated charges
and $w_t\le c_t'(-p)$ at saturated discharges (one-sided derivatives
taken from inside the box, as in \cref{thm:marginal-capacity}). These
conditions are dual feasibility plus complementarity at $x^\star$; by
strong duality complementary slackness holds between \emph{every} optimal
primal and \emph{every} optimal dual, so any admissible chain extends to an
optimal dual, whatever tie-breaking produced $x^\star$. Two classical facts
in storage form: the balance dual is the marginal value of stored energy,
the \emph{water value} of reservoir scheduling~\cite{StageLarsson1961,%
PereiraPinto1991,CruiseFlatleyGibbensZachary2019}, and under primal
degeneracy the exact marginal is the entire optimal dual face rather than
one basis, whose extreme points return the one-sided shadow
prices~\cite{JansenDeJongRoosTerlaky1997,AdlerMonteiro1992}. Consequently
\begin{equation}
M_P=\!\!\sum_{t:\,x_t^\star=+p}\!\!\big(w_t-c_t'(p)\big)
+\!\!\sum_{t:\,x_t^\star=-p}\!\!\big(c_t'(-p)-w_t\big),
\qquad
M_E=\!\!\sum_{\text{full cuts}}\!\!\big(w_{t+1}-w_t\big),
\label{eq:water-value-readout}
\end{equation}
with a terminal jump added when the horizon ends at a bound, and
\emph{every} admissible chain yields a valid joint subgradient --- all a
cut needs, also at kinks of $J$, where the extreme chains produce the two
one-sided derivatives. Matched charge--discharge pairs inside a block
cancel $w$ and contribute their price spread; unmatched saturated legs
couple neighbouring blocks through the jumps, which are the $\mu$'s of
\eqref{eq:marginal-capacity}. A feasible chain is selected by one
backward--forward interval sweep, so the oracle is one value-function
pass plus $\mathcal O(T)$ postprocessing: pure longitudinal algebra, no
linear program. Against exact LP bound duals and one-sided finite
differences, the readout reproduces $(M_E,M_P)$ and both one-sided
derivatives on linear, lossy and multi-piece local costs, including the
degenerate corners that defeated the chord. It is oblivious to the operating
physics carried by the local costs (\cref{subsec:class}): under the
self-discharge change of variable there the water values inherit a geometric
$\rho^{t}$ factor and \eqref{eq:water-value-readout} applies unchanged in the
rescaled state, so exact $(M_E,M_P)$ carry over to lossy and self-discharging
devices --- all handled by \pkg{}.

\Cref{tab:joint-sizing} exercises \cref{alg:level-bundle} with three
operating oracles --- a warm-started single-node monolith (Gurobi or the
licence-free HiGHS), whose bound duals return $(M_E,M_P)$, and the FLT
value-function oracle of \eqref{eq:water-value-readout} --- against
capacity-and-power co-optimising LPs (again Gurobi and HiGHS), on a
price-responsive single node whose seasonal component crosses several merit
steps, so that one signal yields a daily and a seasonal regime from the
investment prices alone. These are research-harness measurements in the
sense of \cref{subsec:parametric-reuse}. All five methods agree to
\jointSizingRelAgreement{} relative (the bundle tolerance; the daily
instances to machine precision), near-flat optimal faces explaining the
residual spread of the reported optimisers. The value-function oracle is
the fastest instantiation on all \jointSizingInstances{} instances --- ahead
of a monolith operating oracle (Gurobi or HiGHS) and of both co-optimising
LPs --- at \jointSizingFltFractionRange{}$\times$ Gurobi's wall time (a
\jointSizingFltSpeedupRange{}$\times$ speed-up), while returning the exact
dual pair $(M_E,M_P)$ at every visited point; isolating the oracle from the
solver, even a Gurobi-backed monolith oracle is
\jointSizingOracleSpeedupRange{}$\times$ slower than the value-function
readout. These are bounded-catalogue single-node instances; against richer
traced curves the same oracle loses the time comparison outright
(\cref{subsec:sizing-scope}), and the speed-up is not claimed beyond this
regime.

\begin{table}[t]
\centering
\caption{Joint $(E,p)$ sizing on a price-responsive single node (four-technology merit order, value of lost load), one signal, two investment-price regimes. Wall seconds, median of 3 (research harness); bundle times include oracle construction, cut counts in parentheses. Objectives agree to $\le5.8\,10^{-7}$ relative. The three rightmost columns all run \cref{alg:level-bundle} and differ only in the operating oracle: a warm-started Gurobi or HiGHS monolith whose bound reduced costs return $(M_E,M_P)$, versus the FLT value-function readout of~\eqref{eq:water-value-readout}. The first two columns instead solve the joint $(E,p)$ co-optimising LP directly (Gurobi, HiGHS). (\pkg{} 0.27.0.)}
\label{tab:joint-sizing}
\begin{tabular}{llccccc}
\toprule
& & \multicolumn{2}{c}{co-optimising LP} &
\multicolumn{3}{c}{\cref{alg:level-bundle}} \\
\cmidrule(lr){3-4}\cmidrule(lr){5-7}
$T$ & regime ($d^\star$) & Gurobi & HiGHS &
\shortstack{Gurobi\\monolith} & \shortstack{HiGHS\\monolith} &
\shortstack{FLT value\\function} \\
\midrule
5\,000 & daily (4.0\,h) & 0.22 & 0.47 & 0.27 (9) & 0.29 (9) & 0.07 (15) \\
5\,000 & seasonal (143\,h) & 0.24 & 0.47 & 1.05 (14) & 0.57 (14) & 0.07 (14) \\
10\,000 & daily (3.8\,h) & 0.48 & 0.86 & 0.82 (16) & 0.91 (16) & 0.14 (16) \\
10\,000 & seasonal (323\,h) & 0.49 & 2.87 & 1.94 (16) & 1.32 (16) & 0.12 (12) \\
\bottomrule
\end{tabular}
\end{table}

\subsection{Scope: few capacities, not a system expansion model}
\label{subsec:sizing-scope}

The oracles above price the capacities of a \emph{small} number of devices, and
that is the claim. Nothing here asserts that they scale to the many capacities of
a full multi-node, multi-commodity expansion model, and the reader should not
read one into them.

The obstruction is not the oracle but the master it feeds. A capacity master over
a system expansion model is driven by the number of Benders iterations, and that
number grows with the number of capacities faster than the per-iteration saving
compensates: the same arithmetic that governs the cross-node block of
\cref{subsec:masters} governs it here, because both replace one large program by
many solves of a program that is not much smaller. On ring expansions of a dozen
areas we measure a frontal barrier winning outright against either decomposition
variant. Where the decomposition does earn its keep is the direction this section
develops --- a handful of intertemporal units, priced exactly, with the complete
capacity curve as a by-product rather than one sizing point. The
\emph{scenario} axis obeys the same catalogue law as everything else. A
multi-scenario sizing LP is block-diagonal with only the design variables
coupling its blocks, so the frontal barrier scales linearly in the number of
weather years wherever it fits in memory; against traced network-node curves,
whose catalogues run to several hundred slopes, each oracle call pays the
arrangement and the secant-plane route does not close the time gap at any
measured size. On a price-taking instance with a bounded catalogue the same
route wins, and the verdict survives a tight certificate --- provided the
master carries two ingredients a plain cutting-plane loop lacks. Optimality
cuts alone plateau at a gap exactly equal to the capex of the reservoir, the
asset whose \emph{operating} marginal is zero at the optimum: the cuts
faithfully report a flat $\partial\mathrm{Op}/\partial E$, the master
extrapolates and collapses the capacity to zero, and no optimality cut can
exclude the infeasible point it lands on. One feasibility frontier --- $E$ at
least the minimal feasible tank of the electrolyser rating, available in
closed form --- closes the certificate, and a level-bundle stabilisation
makes the closure robust to the starting point. So equipped, the route
reaches a certified \kCurveGapTarget{} gap in \kCurveCutsRangeAllSeeds{} cuts
\emph{independently} of $K$ --- a factor sixteen in scenarios --- and of the
seed (\kCurveSeedsCount{} independent draws), at a per-cut cost linear in
$K$ (\kCurvePerSolveRange~s per block solve throughout). It holds no speed
advantage against a fairly configured frontal, and the comparator lesson is
worth the print. In default
\kCurveFrontalSolver~\kCurveGurobiVersion{} settings the frontal loses its
lead as $K$ grows (\kCurveRatioKFour{} at $K=4$, \kCurveRatioKSixteen{} at
$K=\kCurveCrossoverKAllSeeds$, \kCurveRatioKThirtyTwo{} at $K=32$) and fails
outright at $K=64$ on all \kCurveWallSeedsReproduced{} seeds
(\kCurveFrontalFailModeKSixtyFour): the barrier converges to the optimum in
\kCurveDefaultBarrierIters{} iterations and \kCurveDefaultBarrierSeconds~s,
then the crossover forced by the concurrent optimiser runs away
(\kCurveCrossoverPushIters{} push iterations) into its memory limit.
Disabling it (\texttt{Method=2, Crossover=0}) repairs the frontal
completely: the crossover-free barrier returns the same optimum (relative
agreement \kCurveBarrierPureObjAgreement) in
\kCurveBarrierPureSecondsKSixteen{}/\kCurveBarrierPureSecondsKThirtyTwo{}/\kCurveBarrierPureSecondsKSixtyFour~s
at $K=16/32/64$, ahead of the route by
\kCurveBarrierPureAdvantageRange{} throughout. What survives the fair
comparator is not speed but shape: the route's certificate cost is flat in
$K$ and seed, it degrades predictably where the default frontal's wall is
one setting away, and it never performs the crossover that builds that
wall. The remaining caveat is stated as measured: at $K=64$ under default
settings no external judge survives, and exactness there rests on the
certified bound, on warm and cold starts agreeing to
\kCurveWarmColdGapKSixtyFour, and on the crossover-free barrier reproducing
the incumbent.

That is the regime the experiments of \cref{sec:experiments} report, and the
system-scale question is left open rather than answered badly.

\section{Operation under uncertainty: exact recourse and scenario trees}
\label{sec:stochastic}

The solver of \cref{sec:longitudinal} does not return the optimal cost of one
boundary state; it returns the cost of \emph{every} boundary state at once,
as an explicit convex piecewise-linear function. Once that object is
available for each of finitely many futures, taking an expectation over them
requires no new operation: by \cref{cor:poscomb} a positively weighted sum of
cost-to-go functions is again a function of the same class, obtained by one
ordered merge. The recourse function of two-stage stochastic programming is
therefore not \emph{approximated} here --- by cuts, by a state grid, or by
sampling the linking state --- but \emph{assembled}. This section states that
consequence, its exactness proof and its cost at one cut
(\crefrange{subsec:two-stage-problem}{subsec:two-stage-relation}), and then
shows that the same argument survives the repetition: on a scenario tree,
expectation and elimination alternate all the way down the horizon and
exactness propagates through the alternation (\cref{subsec:trees}).

\subsection{Two-stage recourse with a scalar boundary state}
\label{subsec:two-stage-problem}

Stage~1 covers periods $t=1,\dots,T_1$ under known data. Stage~2 covers
periods $t=T_1+1,\dots,T_1+T_2$ under one of $K$ scenarios $w$ --- prices,
inflows, availabilities, bounds --- occurring with probabilities $p_w>0$,
$\sum_wp_w=1$. Both stages are longitudinal in the sense
of~\eqref{eq:longitudinal}: the same scalar cumulative state carries through,
confined to $[0,E]$ (a time-varying envelope changes nothing below). Write
\begin{equation}
  \sigma=s_0+\sum_{t\le T_1}x_t
  \label{eq:boundary-state}
\end{equation}
for the \emph{boundary state} reached at the end of stage~1, and let
\begin{equation}
  V_2^w(\sigma)=\min\Bigl\{\sum_{t>T_1}c_t^w(x_t):\
  x_{T_1+1:T_1+T_2}\text{ feasible under scenario }w\text{ from }
  s_{T_1}=\sigma\Bigr\}
  \label{eq:stage2-value}
\end{equation}
be the optimal stage-2 cost of scenario $w$ started from $\sigma$. The
two-stage program with recourse is
\begin{equation}
  \min_{x_{1:T_1}}\ \sum_{t\le T_1}c_t(x_t)
  \;+\;\E_w\bigl[V_2^w(\sigma)\bigr],
  \label{eq:two-stage}
\end{equation}
the minimisation running over stage-1 decisions whose running state stays
inside the envelope. Non-anticipativity is structural: the stage-1 decisions
carry no index $w$, and the scenarios communicate with them only
through~\eqref{eq:boundary-state}.

\subsection{Assembling the expected cost-to-go}
\label{subsec:two-stage-algorithm}

Let $F=V_{T_1}$ be the forward value function of stage~1 produced by
recursion~\eqref{eq:longitudinal-recursion}, and let $B^w$ be the cost-to-go
of stage~2 under scenario $w$, produced by the mirrored backward recursion on
that scenario's data --- the same restrictions and infimal convolutions after
reflecting the state (the general forward--backward factorisation is
developed in \cref{subsec:forward-backward}). Write
$\Qexp=\sum_wp_wB^w$ for the expected cost-to-go.
\Cref{alg:two-stage} is then the whole method.

\begin{algorithm}[t]
\caption{Exact two-stage recourse on a longitudinal chain}
\label{alg:two-stage}
\begin{algorithmic}[1]
\Require stage-1 data $(c_t)_{t\le T_1}$, scenario stage-2 data
$(c_t^w)_{t>T_1}$ and probabilities $(p_w)_{w\le K}$, envelopes, $s_0$
\State $F\gets$ forward value function of stage~1 at the cut $T_1$
       \Comment{one forward pass (\cref{alg:longitudinal})}
\For{$w=1,\dots,K$}\Comment{embarrassingly parallel}
  \State $B^w\gets$ cost-to-go of stage~2 under scenario $w$, as a function
         of $\sigma$ \Comment{one backward pass}
\EndFor
\State $\Qexp\gets\sum_{w}p_wB^w$
       \Comment{one $K$-way ordered merge (\cref{cor:poscomb}); no cuts, no
       grid}
\State $\sigma^\star\gets$ deterministic selection in
       $\argminop_{\sigma\in[0,E]}\{F(\sigma)+\Qexp(\sigma)\}$
\State recover the stage-1 schedule pinned at $\sigma^\star$, and each
       scenario's stage-2 schedule out of $\sigma^\star$
       (\cref{lem:recovery})
\State \Return $\sigma^\star$, the schedules, the curves
       $\{B^w\}_{w\le K}$ and $\Qexp$
\end{algorithmic}
\end{algorithm}

\begin{proposition}[Exactness of the assembled recourse]
\label{prop:two-stage-exact}
The optimal value of~\eqref{eq:two-stage} equals
\[
\min_{\sigma\in[0,E]}\ \bigl\{F(\sigma)+\Qexp(\sigma)\bigr\},
\]
the minimiser set is a nonempty closed interval, and the schedules recovered
by \cref{alg:two-stage} are optimal for~\eqref{eq:two-stage}.
\end{proposition}

The argument is the one the whole paper turns on, and it is worth stating in
one sentence before the bookkeeping: the boundary state is the \emph{only}
variable linking the two stages, so conditioning on it makes the objective
separate into $F(\sigma)+\sum_wp_wB^w(\sigma)$, and the expectation of
finitely many convex PWL functions is again one of them --- obtained by a
single ordered merge (\cref{cor:poscomb}), with no approximation anywhere.
The proof, which checks attainment and properness at each step, is in
\cref{app:proofs}.

The result is exact in the strict sense used throughout the paper: no state
grid, no sampling of $\sigma$, no outer convergence loop, and no tolerance in
the definition of $\Qexp$. The only approximation left in the model is the
one the modeller chose --- the finite scenario set itself.

\paragraph{Cost.}
\Cref{alg:two-stage} performs $K+1$ value-function passes --- one forward
over $T_1$ periods and one backward over $T_2$ periods per scenario, each
priced by \cref{prop:opcost} with the horizon totals of
\cref{rem:tseng-luo} --- plus one $K$-way ordered merge whose output carries
at most $\sum_wm_w+O(1)$ records, where $m_w$ is the record count of $B^w$
(\cref{cor:poscomb}), plus one pointwise sum and one minimisation. The $K$
scenario passes share no data and are embarrassingly parallel. The
comparison point is the deterministic equivalent of~\eqref{eq:two-stage},
whose constraint matrix carries one copy of the stage-2 block per scenario
and therefore grows by a factor $K$; here the growth is confined to the merge
that assembles the recourse function.

\begin{remark}[Per-scenario water values]
\label{rem:scenario-water-values}
The algorithm returns the curves $B^w$, not only their weighted sum. The
one-sided derivatives $(B^w)'(\sigma^\star)$ are the marginal water values of
the individual scenarios at the optimal boundary state, in currency per unit
of stored energy; the derivative of $\Qexp$ at $\sigma^\star$ is their
$p$-weighted average. Their dispersion across $w$ measures how much the
scenarios disagree about the worth of a stored unit at the cut --- a
value-of-information signal available at no extra cost, and one that
cut-based solvers do not return, since a cut aggregates the scenarios before
the marginal is read.
\end{remark}

\subsection{Relation to cut-based stochastic programming}
\label{subsec:two-stage-relation}

\Cref{alg:two-stage} is the L-shaped
method~\cite{VanSlykeWets1969,BirgeLouveaux2011} --- and, on a scalar state,
the exact limit of stochastic dual dynamic
programming~\cite{PereiraPinto1991,Shapiro2011,Dowson2021} --- with the
recourse function computed exactly, one pass per scenario, instead of
approximated from below by iterated subgradient cuts: the master problem
disappears, because the exact expected recourse function of a scalar linking
variable is cheaper to build than a sequence of its supporting hyperplanes.
Non-anticipativity likewise needs none of the consensus iteration of
progressive hedging~\cite{RockafellarWets1991}, being enforced exactly by the
single linking state $\sigma$.

Off the scalar state the positioning splits in two, and only the second half
fails. \emph{The composition stays exact one dimension up.} With two co-located
inventories the bivariate object of \cref{subsec:two-store-conjugate} plays the
role of $B^w$: each scenario's stage-2 cost-to-go is still an exact convex
piecewise linear function of the boundary state, of size bounded by the
arrangement rather than by the horizon, and $\Qexp=\sum_wp_wB^w$ is still
\emph{assembled} --- no cut, no state grid, no sampling of $\sigma$, no outer
loop and no tolerance. Nothing in the exactness argument of
\cref{prop:two-stage-exact} used the dimension of $\sigma$. It can be gated
against the deterministic equivalent exactly as in the scalar case; that
campaign is not yet among the generated tables.

\emph{The cost law does not.} Linearity in $K$ is a \emph{scalar state} result:
in one dimension the breakpoint set of a sum is the \emph{union} of the
breakpoint sets (\cref{cor:poscomb}), so the merge is linear in $\sum_wm_w$,
whereas in two dimensions the same pointwise sum is affine on the
\emph{overlay} of the scenarios' subdivisions, an \emph{arrangement} and not a
union, and the growth becomes quadratic in $K$. The mechanism is a duality one,
which is why the failure is structural and not an artefact of the data: a
breakline of the \emph{conjugate} is a merit-order slope times a storage gain,
both properties of the park, so the catalogue is shared across scenarios and no
scenario ever contributes a new slope or a new direction, while its Legendre
dual is a \emph{primal} edge supported by a level, that is by an integral of
the trajectory over the block, which is a property of the \emph{weather}: every
scenario contributes a new position, and the pairwise intersection of the
scenarios' primal line sets is measured empty.

Cut-based methods therefore regain their advantage in the \emph{cost law} and
not in exactness: they pay for dimension in cuts rather than in representation
size. In absolute terms the assembled object stays small at the $K$ a scenario
tree actually carries, so the construction remains usable there; what does not
transport is the claim of linear growth. The same caveat applies when block
constraints break the longitudinal structure of a stage, so that no exact
one-pass cost-to-go exists to be averaged.

\subsection{Exact multistage operation on scenario trees}
\label{subsec:trees}

The construction above merges once, at one cut, and the merged object is
consumed by a single minimisation. Multistage operation is not that:
information arrives repeatedly, decisions taken after each arrival must be
feasible for every continuation, and the expectation operator therefore has to
alternate with elimination all the way down the horizon. The question is
whether that alternation preserves exactness --- whether the object merged at
a node, itself the output of an elimination that already used a merge, still
is the true conditional expected cost-to-go. It does, for the scalar
longitudinal class, and the resulting sweep visits every node once.

\paragraph{Stage elimination on a tree.}
Data now live on a rooted tree: node $\nu$ carries a block of hours $H_\nu$
with its own scenario data and children $C(\nu)$ with conditional
probabilities $p_c$, $\sum_{c\in C(\nu)}p_c=1$. The state of charge flows
through the tree: the state reached at the last hour of $\nu$ is the state
entered by every child of $\nu$. Let $B_\nu(\sigma)$ be the optimal expected
cost of the subtree rooted at $\nu$, started from state $\sigma$ at the first
hour of $\nu$. Backward stage elimination from the leaves computes every
$B_\nu$ exactly: at each node, eliminate the node's hours by restricted
convolutions against the terminal function
$\Qexp_\nu=\sum_{c\in C(\nu)}p_c\,B_c$.

\begin{algorithm}[t]
\caption{Tree sweep: exact value functions at every node}
\label{alg:tree-sweep}
\begin{algorithmic}[1]
\Procedure{TreeSweep}{$\nu$}
  \If{$C(\nu)=\emptyset$}
    \State $\Qexp_\nu \gets$ terminal valuation (zero by default)
  \Else
    \For{$c\in C(\nu)$} \Call{TreeSweep}{$c$} \EndFor
    \State $\Qexp_\nu \gets \sum_{c\in C(\nu)} p_c\,B_c$
    \Comment{one ordered merge (\cref{cor:poscomb})}
  \EndIf
  \State $B_\nu \gets$ backward pass over $H_\nu$ with terminal function
  $\Qexp_\nu$ \Comment{\cref{subsec:forward-backward}}
\EndProcedure
\State \Call{TreeSweep}{root}; the root minimisation pinned at the initial
state gives the first-stage decision, and recovery descends the tree
(\cref{lem:recovery}).
\end{algorithmic}
\end{algorithm}

\begin{figure}[t]
\centering
\begin{tikzpicture}[
    >={Stealth[length=2mm]},
    tnode/.style={draw, rounded corners=1.5pt, inner sep=3pt, font=\small},
    prob/.style={font=\scriptsize, fill=white, inner sep=1pt},
    lbl/.style={font=\small},
    sub/.style={font=\footnotesize, gray!90}]
  \node[tnode] (r)  at (0,0)      {$\nu_0$};
  \node[tnode] (a)  at (2.6,1.2)  {$\nu_1$};
  \node[tnode] (b)  at (2.6,-1.2) {$\nu_2$};
  \node[tnode] (aa) at (5.2,1.8)  {$\nu_3$};
  \node[tnode] (ab) at (5.2,0.6)  {$\nu_4$};
  \node[tnode] (ba) at (5.2,-0.6) {$\nu_5$};
  \node[tnode] (bb) at (5.2,-1.8) {$\nu_6$};
  \draw[->] (r) -- (a) node[prob, midway] {$p_{\nu_1}$};
  \draw[->] (r) -- (b) node[prob, midway] {$p_{\nu_2}$};
  \draw[->] (a) -- (aa) node[prob, midway] {$p_{\nu_3}$};
  \draw[->] (a) -- (ab) node[prob, midway] {$p_{\nu_4}$};
  \draw[->] (b) -- (ba) node[prob, midway] {$p_{\nu_5}$};
  \draw[->] (b) -- (bb) node[prob, midway] {$p_{\nu_6}$};
  \draw[->, gray] (5.4,2.55) -- (0.2,2.55)
    node[midway, above, sub] {one backward sweep, leaves $\to$ root};
  \node[lbl] at (2.6,-2.6)
    {$\Qexp_\nu=\sum_{c\in C(\nu)} p_c\,B_c$ \ (merge), then $B_\nu\gets$
     backward pass over the hours $H_\nu$};
\end{tikzpicture}
\caption{One backward sweep of the tree: at each node, the children's
value functions merge into $\Qexp_\nu$ (weighted pointwise sum), then the
node's hours are eliminated backward. Every node ends up carrying its
exact expected cost-to-go --- what SDDP approximates by sampled cuts.}
\label{fig:tree}
\end{figure}

\begin{proposition}[Exactness on the tree]
\label{prop:tree-exact}
Let the tree be finite, let every node block be longitudinal in the sense of
\cref{subsec:class} with closed proper convex PWL step costs, and let the
state be confined to $[0,E]$ throughout. Then for every node $\nu$ the
function $B_\nu$ returned by \cref{alg:tree-sweep} is closed proper convex PWL
and equals the optimal expected cost of the subtree rooted at $\nu$, for
\emph{every} entry state $\sigma\in[0,E]$ simultaneously. In particular the
optimal value of the extensive form started from $s_0$ is
$B_{\mathrm{root}}(s_0)$, and the schedules recovered by descending the tree
from the root are feasible, non-anticipative and optimal.
\end{proposition}

The proof --- an induction on the height of the tree, value part and recovery
--- is in \cref{app:proofs}.

\begin{remark}[Cost in tree size: an accounting, measured, not a theorem]
\label{rem:tree-cost}
Weighted sums concatenate breakpoint lists (output at most additive in the
inputs, \cref{cor:poscomb}) and each hour's restricted convolution adds $O(1)$
records (\cref{prop:opcost}), which suggests a running time of
$\widetilde O\bigl(\sum_\nu |H_\nu|\bigr)$ --- linear, up to logarithmic
factors, in the total number of tree-hours, the same growth as the
deterministic-equivalent LP's \emph{data} with the FLT constant per hour. We
state this as an accounting and not as a proposition: a tight worst-case
bound would need breakpoint accounting across the merges, which we do not
carry out, and the measured support (\cref{subsec:tree-bench}) is flat wall
time over trees whose absolute times sit below the resolution at which an
empirical growth exponent is meaningful.
\end{remark}

\begin{remark}[Per-node water values]
\label{rem:node-water-values}
The sweep does not return only the root decision: it leaves an explicit
$B_\nu$ at every node. Its one-sided derivatives are the marginal water values
conditional on the history that reaches $\nu$, in currency per unit of stored
energy, and the derivative of $\Qexp_\nu$ is their $p$-weighted average; their
dispersion across the children of $\nu$ measures how much the branches
disagree about the worth of a stored unit at that cut. These are exactly the
curves that \cref{cor:rolling-exact} consumes as terminal valuations, so the
tree sweep is also the machine that produces the policy a rolling operation
would run.
\end{remark}

\begin{remark}[Reuse]
\label{rem:tree-reuse}
The sweep leaves an exact $B_\nu$ at every node. Changing stage-1 data ---
costs, bounds, initial state --- reuses every downstream $B_\nu$ unchanged,
the tree analogue of the parametric reuse measured in
\cref{subsec:parametric-reuse}. The same holds for rolling re-solves as the
tree is reshaped near the root.
\end{remark}

\begin{remark}[Offsets and anchors]
\label{rem:offsets}
Slopes are exact as returned, and every minimisation performed by
\cref{alg:tree-sweep} needs slopes only. When the optimal \emph{value} itself
is required --- for a comparison against a deterministic-equivalent objective,
as in \cref{subsec:tree-bench} --- values are pinned per node with one anchor
solve.
\end{remark}

\paragraph{Positioning: what SDDP samples, the sweep assembles.}
Stochastic dual dynamic programming~\cite{PereiraPinto1991,Shapiro2011,Dowson2021}
replaces the cost-to-go of each stage by an outer polyhedral approximation,
built by alternating sampled forward passes with backward collection of
subgradient cuts and refined until a statistical upper bound and the
deterministic lower bound agree within a tolerance; its efficiency rests on
sharing cuts among nodes of the same stage, which is legitimate under stagewise
independence of the uncertainty or, in the policy-graph formulation, across the
states of a Markov lattice~\cite{Dowson2021}. On a scalar cumulative state none
of that machinery is needed to obtain the cost-to-go, because the cost-to-go
\emph{is} a convex piecewise-linear function of one variable and one backward
pass returns it in full: the polyhedral outer approximation becomes the
function itself, the sampling loop a single deterministic sweep, the cut pool
per stage one explicit value function per node; the statistical stopping rule
disappears, leaving no tolerance in the algorithm; and no stagewise
independence is required, because nothing is shared between nodes --- each node
pays for its own block and merges its own children (\cref{prop:tree-exact}).

The price is why the two methods do not compete on the same instances. The
sweep visits every node of the tree and stores its function, so its cost grows
with the tree itself --- for a fixed branching factor, exponentially in the
number of stages --- whereas an SDDP iteration costs the sampled paths rather
than the tree, which is what makes very deep horizons tractable there. The
sweep is exact on the tree it is given; building that tree, by scenario
generation or reduction, is the approximation the modeller keeps, and it is the
only one left in this subsection. Its error is quantified and controlled in its
own literature~\cite{HoylandWallace2001,Pflug2001,HeitschRomisch2003}, and the
exactness claimed here is conditional on the tree those methods deliver.

Two neighbouring filiations mark the boundary: value functions computed
offline and then used to coordinate subsystems, the mechanism of approximate
spatiotemporal decomposition for microgrid operation~\cite{Zhu2019STD}, from
which the difference is not the architecture but that the coordinating object
is computed exactly and in one pass rather than learned; and the mixed
spatial--temporal decompositions of Carpentier, Chancelier, De Lara and
Pacaud~\cite{CarpentierChancelierDeLaraPacaud2020}, which attack the same
product of a network and a horizon by decomposing along both axes and
exchanging price or resource signals between the pieces, tolerate vector states
and genuine multistage information structures --- which this construction does
not --- and pay for that generality with an approximation that
\cref{prop:tree-exact} does not need on a scalar state.

No head-to-head timing against an SDDP implementation~\cite{Dowson2021} is
claimed: on the finite trees of \cref{subsec:tree-bench} such a race would
oppose one exact sweep to a sampling loop with a statistical stopping rule,
and the two do not answer to the same tolerance semantics. The difference this
paper relies on is structural --- exactness on the given tree, and the per-node
curves of \cref{rem:node-water-values} as output where the
deterministic-equivalent LP returns a single primal--dual pair and SDDP a cut
pool --- and \cref{subsec:tree-bench} measures it as such.

\subsection{Where the free part stops: risk aversion, and the scope of the
whole section}
\label{subsec:risk-scope}

\begin{remark}[CVaR stays in the class, but leaves the single merge]
\label{rem:cvar}
Expectation is risk-neutral, and adequacy studies often want tail weighting.
This is where the free part of the construction stops, and the boundary is
sharp. Unlike the expectation, $\cvar{\alpha}$ is not a positive combination
of the $B^w$: in the Rockafellar--Uryasev variational
form~\cite{RockafellarUryasev2000} it is a \emph{maximum} over a set of such
combinations. It therefore stays in the class --- a pointwise maximum of
finitely many convex PWL functions is convex PWL, so
$\sigma\mapsto\cvar{\alpha}(B^w(\sigma);p)$ is again convex PWL in $\sigma$,
affine on the common breakpoint refinement --- but it does not come out of one
ordered merge: the identity of the $\alpha$-tail changes wherever scenario
curves cross, which is $O(K^2)$ candidate crossings per interval pair, and the
crossing structure must be maintained interval by interval. Implementable, and
heavier than the merge by exactly that factor. The algorithmic core of this
paper therefore stays with $\Qexp=\E_w[B^w]$.
The variational form makes the construction explicit, but we do not carry it
out: $\cvar{\alpha}$ is not implemented in \pkg{} at the version measured
here, and no risk-averse claim is made in this paper.
\end{remark}

\Cref{rem:cvar} is the risk-axis twin of the dimension-axis paragraph above:
in both, exactness survives and the \emph{cost law} does not. Stating the two
together fixes the scope of the section. What this paper does \emph{not}
develop, and what it does not claim: a rolling two-stage model-predictive
policy, for which no measurement exists here; and valued multi-node blocks,
where a PWL terminal valuation is attached per unit inside a coordinated
network solve. The second is not a gap in the exposition but an open
construction: the coordination layer of \cref{sec:coordination} carries
schedules, not per-unit valuations, and nothing here says how the two compose.

\section{Where the hourly curves come from: exact network injection costs}
\label{sec:injection}

The longitudinal solver of \cref{sec:longitudinal} looks along time. Its
spatial counterpart fixes one hour of a network dispatch and asks what a
scalar injection costs there; the answer is again a convex piecewise-linear
curve, emitted as the ordered breakpoints and slopes that the engine of
\cref{sec:flt} consumes directly. This section builds that spatial oracle
(\cref{subsec:injection-trace}), states the coupled dispatch problem it serves
(\cref{subsec:coupled-model}), and composes the two oracles
(\cref{sec:coordination}). The curves it traces are also the data on which the
catalogue law of \cref{subsec:catalogue-law} is measured: a year of hourly
network dispatch is the rich-catalogue end of that axis.

The composition is carried here to one intertemporal unit per coupled group
(\cref{prop:one-storage}), which is where exactness stops; the coordination
layer a system with several such units requires --- its lock, the certified
joint block that removes it, and the annual system-scale measurements --- is
recorded in \cref{app:coordination}.

\subsection{The scalar injection value and its parametric trace}
\label{subsec:injection-trace}

Let $z$ collect the hour's generation, flows, curtailment and load
shedding, and let a scalar $y$ modify the balance along direction $h$:
\begin{equation}
\phi(y)=\min_z\{q^\top z:Az=d+yh,\ \ell\le z\le u\}.
\label{eq:spatial-value}
\end{equation}
The parameter can be an injection at one electricity node or a sparse
direction touching several commodity balances. The value $\phi$ is convex
and piecewise linear; each affine piece is an optimality interval and its
slope is $-h^\top\lambda$, where $\lambda$ is an optimal balance
multiplier~\cite{Gal1995,AdlerMonteiro1992}. This scalar CPWL function is
the interface between the hourly network and the longitudinal solver.

A grid of independent LP solves can miss narrow critical regions and pays a
general-solver setup cost at every sample. \pkg{} instead walks adjacent
optimal bases with a bounded-variable parametric dual simplex
(\cref{alg:trace}). The idea is classical: related network-flow algorithms
parameterise shipped flow and augment along shortest
paths~\cite{AhujaBatraGupta1984}, while more recent methods compute complete
parametric minimum-cost-flow paths for strictly convex piecewise-quadratic
costs~\cite{KlimmWarode2022}. Our contribution is the implementation for
thousands of closely related bounded hourly LPs: arbitrary sparse
directions, basis transport across hours and coordination sweeps, and
predictive basis atlases, all realised as batched C++ kernels
(\cref{sec:package,app:simplex}). The trace emits the exact ordered
breakpoints and slopes that the Legendre engine consumes directly.
Hour-chunk parallelism (\cref{subsec:curve-engineering}) is an optional
accelerator layered on top of these serial kernels: it changes wall time
only, never the curve computed, so every exactness statement below and every
single-storage timing holds single-threaded.

\begin{algorithm}[t]
\caption{Parametric trace of the injection value $\phi$ on $[y^-,y^+]$}
\label{alg:trace}
\begin{algorithmic}[1]
\Require hourly LP data $(A,d,q,\ell,u)$, direction $h$, start basis $B$
(atlas entry, previous hour, or cold start)
\State reoptimise $B$ at $y\gets y^-$ with the bounded-variable dual simplex
\While{$y<y^+$}
  \State record the slope $-h^\top\lambda(B)$ from the basis multipliers
  \State ratio test along $h$: largest $\Delta$ keeping $B$ primal feasible
         on $[y,y+\Delta]$
  \If{$y+\Delta\ge y^+$}
    \State record the final piece; \textbf{break}
  \EndIf
  \State $y\gets y+\Delta$; record the breakpoint; rank-one basis update
         (dual pivot)
\EndWhile
\State \Return ordered breakpoints and slopes of $\phi$, consumed directly
by \cref{alg:longitudinal}
\end{algorithmic}
\end{algorithm}

For the transport matrices used below, total unimodularity makes the
rank-one basis update especially simple and provides a structural check on
the trace~\cite{Schrijver1986}. No polynomial claim is made for arbitrary
LPs: the number of critical regions can be exponential~\cite{Murty1980}.
The method is output-sensitive and effective here because the annual 8-node
curves contain only $2.35$ affine pieces per hour on average.

\subsection{The coupled dispatch model}
\label{subsec:coupled-model}

The isolated longitudinal problem is exact and fast because one scalar state
separates past from future. A realistic energy system adds a different
coupling: at each hour, generation and transport link the injections of many
nodes, while storage links successive hours. We now embed the FLT solver as a
temporal oracle, consuming the spatial curves of
\cref{subsec:injection-trace}.

Let $z_t$ collect the spatial variables of hour $t$ and let $x_{r,t}$ be the
reservoir-side increment of storage $r$. The co-optimised dispatch is
\begin{equation}
\begin{aligned}
\min_{z,x}\quad &\sum_{t=1}^{T}q_t^\top z_t,\\
\text{s.t.}\quad&A_tz_t=d_t+\sum_{r\in\mathcal S}h_rL_r(x_{r,t}),
\qquad \ell_t\le z_t\le u_t,\\
&s_{r,t}=s_{r,t-1}+a_{r,t}+x_{r,t},\\
&0\le s_{r,t}\le E_r,\qquad
\underline p_r\le x_{r,t}\le\overline p_r.
\end{aligned}
\label{eq:coupled-dispatch}
\end{equation}
Here $a_{r,t}$ is an inflow and $L_r$ maps reservoir-side control to grid
load. For a device with round-trip efficiency $\eta_r$,
\begin{equation}
L_r(x)=\begin{cases}x,&x\ge0,\\ \eta_rx,&x<0.\end{cases}
\end{equation}
The spatial model couples nodes but not hours; each storage model couples
hours but has a scalar state. A frontal LP exposes both structures to one
simplex or barrier solve. Our alternative alternates exact projections onto
one structure at a time and exchanges convex value curves.

\subsection{Composing the two oracles}
\label{sec:coordination}
\label{subsec:alternation}

The two oracles compose in one direction immediately. Suppose the system
contains one storage $r$. Trace the exact hourly value~\eqref{eq:spatial-value}
in the direction $h_r$ and compose it with the efficiency map,
$\psi_t(x)=\phi_t(L_r(x))$; with free curtailment $\phi_t$ is non-decreasing,
so $\psi_t$ is convex PWL, and the dispatch projected onto $x$ is
$\min_x\sum_t\psi_t(x_t)$ under the storage bounds --- an instance
of~\eqref{eq:longitudinal}.

\begin{proposition}[Exactness for one storage]
\label{prop:one-storage}
If the hourly curves are traced over the full feasible injection interval,
one parametric spatial pass followed by one FLT solve returns a global
optimum of~\eqref{eq:coupled-dispatch} with one storage. Both eliminations
are exact projections: $z_t$ appears in no other hour's constraints, so
eliminating it hour by hour is exact and leaves $\sum_t\psi_t(x_t)$, which
\cref{alg:longitudinal} then solves exactly (\cref{lem:recovery}).
\end{proposition}

This is the cleanest demonstration of composability: a generic network LP is
eliminated hour by hour into PWL costs, and the temporal chain is then
eliminated by Legendre transforms. \Cref{subsec:one-node-bench} instantiates
it at annual scale.

Beyond one intertemporal unit per group the composition leaves that regime. The
curve seen by each unit becomes conditional on the schedules of the others, so
a sweep is block-coordinate descent on a convex but \emph{nonsmooth} reduced
objective, whose coordinatewise optima need not be globally
optimal~\cite{Tseng2001}; co-located high-power devices lock on a shared
merit-order kink, because moving either alone need not cross a kink that a
joint move would cross. A certified joint block removes that lock exactly on
eligible groups, but it buys correctness rather than speed. The construction
and the system-scale measurements are recorded in \cref{app:coordination}.
This paper's own answer for the smallest such block is different in kind:
\cref{sec:pair} \emph{eliminates} the pair instead of coordinating it.

\begin{remark}[Solver-agnostic and function-specific layers]
\label{rem:oracle-agnostic}
\Cref{prop:one-storage} depends on an exact longitudinal solver, not on
Legendre transforms specifically: once the spatial model is projected into the
separable hourly costs $\psi_t$, the temporal block is an RAP--NC instance
that the MDA algorithm~\cite{Vidal2019} could also solve
(\cref{subsec:related}). The interchangeability extends to every layer that
consumes only an optimal \emph{schedule}, and stops at the layers that consume
a \emph{function} --- the forward--backward capacity sensitivities of
\cref{sec:investment} and the expected cost-to-go of \cref{sec:stochastic}.
\end{remark}

\subsection{Building and consuming exact curves efficiently}
\label{subsec:curve-engineering}

Three implementation details determine whether the decomposition is useful.
First, the parametric simplex traces critical regions rather than repeatedly
calling an LP solver on sampled injections. Second, bases are carried across
hours and traces can be chunked over independent hour ranges. Third, exact
spatial curves may contain many nearly collinear breakpoints that have no
economic effect but inflate subsequent temporal value functions. An
$\varepsilon$-simplification removes a breakpoint only when the induced
vertical error is below a stated tolerance; feasibility is never simplified.
The approximation status and its error budget are therefore separate from
the exact curve trace.

The mechanics of basis reuse --- a shared hourly atlas, storage-specific
anchor bases, a predictor certified or repaired by dual simplex --- and a
pass-by-pass profile of where native simplex time goes are engineering
rather than algorithm design; \cref{app:simplex} records them, together with
the measured effect of incremental primal maintenance. Their headline
consequence is that curve tracing, not the FLT solve, dominates coordination
time as the network grows (\cref{app:coordination}).

\part{The transform in dimension two and beyond}

\section{Carrying a pair of states: the bivariate block}
\label{sec:pair}

Everything so far has carried one cumulative state. The engine of
\cref{sec:flt} conjugates functions of one variable, the solver of
\cref{sec:longitudinal} eliminates a scalar separator exactly, and the three
payoffs that followed --- capacity sensitivities, assembled recourse, and the
composition with a network's hourly curves --- are all readings of a
one-dimensional value function. This part asks what survives when the carried
state becomes a \emph{pair}.

The instance is two storages at one node, coupled only through the hourly
cost of their aggregate injection. The coupling can be \emph{eliminated}, by
carrying the value function of the pair rather than of each store --- the
same variable elimination as \cref{sec:longitudinal}, one dimension up,
tractable because conjugation keeps it combinatorially finite. The competing
treatment --- keeping two one-dimensional recursions and paying a
coordination layer --- is the subject of \cref{app:coordination}, and
\cref{rem:two-store-scope} below bounds where each reading applies.

Two things make this more than an extension exercise, and they are the
results of this part. The object the recursion carries is
\emph{combinatorially rigid}: its breakline directions are fixed in advance
by a theorem, in the plane and, under a common charge--discharge
coefficient, in any dimension. And what that rigidity costs is measurable
and adverse --- the step pays a face count quadratic in the slope catalogue
it reads, which is why the same curve is fast read as a chain and slow read
as a block. Throughout, $\psi_t$ is a convex piecewise-linear hourly cost of
net injection at the node; the traced network curves of \cref{sec:injection}
are its system-scale source, and \cref{subsec:catalogue-law} measures both
readings on the same data.

\subsection{The carried pair and its hourly step}
\label{subsec:two-store-conjugate}

Write $W_t(\ell)$ for the cost-to-go of the pair at levels
$\ell=(\ell_1,\ell_2)$ and $W^\star_t$ for its conjugate, a convex piecewise
linear function of a pair of \emph{water values} $s=(s_1,s_2)$ whose gradients are
levels. The hourly step is exactly
\begin{equation}
W^\star_{t+1} \;=\; \bigl(W^\star_t \text{ clamped on each axis into } [0,E_i]\bigr)
\;+\; G^\star_t ,
\label{eq:bivariate-step}
\end{equation}
because the level box is separable, so $(W+\iota_{\mathrm{box}})^\star
= W^\star\,\square\,\sigma_{\mathrm{box}}$ factorises into two directional
gradient clamps; and because $(W\,\square\,G_t)^\star = W^\star + G^\star_t$, so the
hourly gadget is an \emph{addition}. No planar transform is taken inside the loop.

\begin{remark}[The fixed arrangement, and what it rests on]
\label{hyp:stationarity}
Every breakline that \eqref{eq:bivariate-step} creates belongs to three
fixed pencils: the verticals $\{s_1 = p\,c_1\}$ and horizontals
$\{s_2 = p\,c_2\}$ for $p$ a slope of some $\psi_t$ and $c_i$ a charge or
discharge coefficient, and at most four rays through the origin
$\{s_2 = (c_2/c_1)\,s_1\}$ --- the loci where splitting one extra \si{\mega\watt\hour}
between the two stores is indifferent. \emph{Once the position alphabet is
complete} --- which the implementation decides by a membership guard rather
than assumes --- the carried subdivision is permanently a sub-complex of one
\emph{fixed} line arrangement, and the hourly step is a labelling of its
faces: two directional clamps and one addition per face, with nothing cut,
clipped or rebuilt. The two halves of that picture have different status,
and the paper keeps them apart throughout: the \emph{directions} of the
three pencils are a theorem (\cref{thm:directions}), while the
\emph{positions} of the verticals and horizontals remain a per-instance
measured quantity --- and \cref{prop:two-closed} shows that this asymmetry
is structural rather than a matter of proving harder.
\end{remark}

\subsection{The primal reading, and the direction theorem}
\label{subsec:pair-theorem}

The recursion is carried on the conjugate side, where the objects are
subdivisions of the price plane. Their breakline alphabet, however, is
legible on the \emph{primal} side, and that is what makes it provable. Write
$Z$ for the primal function whose conjugate the recursion carries --- the
cost-to-arrive of the pair in the forward sweep, the cost-to-go in the
backward one --- so that one hour of \eqref{eq:bivariate-step} reads
\begin{equation}
Z'\;=\;\bigl(Z+\iota_{[0,E_1]\times[0,E_2]}\bigr)\;\infconv\;\check G_t ,
\label{eq:pair-primal}
\end{equation}
$\check G_t$ being the hourly gadget, reflected in the sweep that traverses
time backwards. All primal objects along the recursion have compact
non-empty domains, so their conjugates are finite on the whole plane, and
two facts follow that do all the work: a piece of $\conj Z$ active on a set
of positive area is a \emph{vertex} of the primal subdivision, and a
realized breakline --- a positive-length frontier between two maximal affine
regions --- is a primal \emph{edge}, its supporting line being
$\{\langle u-v,\,s\rangle = Z(u)-Z(v)\}$ for the edge's two endpoints
(\cref{app:stat-bridge,app:stat-read}).

A question about the arithmetic of gradients has thereby become a question
about a polyhedral complex. Let
\[
R^{*}\;=\;\bigl\{\,(c_2^{\sigma_2},\,-c_1^{\sigma_1})\;:\;
\sigma_1,\sigma_2\in\{-,+\}\,\bigr\},
\]
at most four unoriented directions, none axis-parallel when the charge and
discharge coefficients are positive, and say that a primal object satisfies
the invariant $\mathbf J$ when \emph{(J1)} every edge of its subdivision is
parallel to $e_1$, to $e_2$, or to a member of $R^{*}$, and \emph{(J2)} the
object is \emph{constant} along every edge whose direction lies in $R^{*}$.
Under $\mathbf J$ the primal reading translates at once: an axis-parallel
edge gives a vertical or a horizontal breakline, while an edge of direction
$d\in R^{*}$ has equal values at its endpoints by (J2), so its breakline is
$\langle u-v,s\rangle=0$ --- a line through the origin \emph{exactly}, of
slope $c_2^{\sigma_2}/c_1^{\sigma_1}$. The constancy clause is therefore not
an accounting device: it is precisely what makes the rays pass through the
origin rather than merely share a direction.

\begin{theorem}[Directional stationarity of the two-store block]
\label{thm:directions}
Assume the hourly curves $\psi_t$ finite, convex, piecewise linear and
non-decreasing (free curtailment, the property already invoked before
\cref{prop:one-storage}), the charge and discharge coefficients
$0<c_i^-\le c_i^+$ fixed across the instance, and the power and level boxes
non-empty and compact. Then for every hour and for every object carried by
\eqref{eq:bivariate-step} --- after each one-sided gradient clamp and after
each gadget addition, in the forward as in the backward sweep, free or
pinned terminal --- every realized breakline of the carried subdivision is
\emph{(a)} vertical, \emph{(b)} horizontal, or \emph{(c)} borne by a line
passing exactly through the origin with slope
$c_2^{\sigma_2}/c_1^{\sigma_1}$, of which there are at most four. No
genericity is assumed, no viewport bound enters, and the statement holds for
any losses and any horizon.
\end{theorem}

The proof is an induction, and its three steps are the three operators of
the recursion; the lemmas are in \cref{app:stat}. The initial objects
satisfy $\mathbf J$: a point, or the level box, whose edges are axis-parallel
and whose (J2) is vacuous. The hourly gadget satisfies it because inside each
quadrant it is $\psi_t$ composed with a linear form, so its affinity regions
are cut by that form's level lines --- an edge of direction exactly
$(c_2^{\sigma_2},-c_1^{\sigma_1})$, along which the gadget is constant, equal
to $\psi_t$ at the kink. A level restriction refines the complex and the only
edges it creates lie on the cut line, which is axis-parallel. The
inf-convolution adds epigraphs, hence adds faces: an edge of the sum is the
sum of the faces its exposing price exposes in each operand, so its direction
is an edge direction of one of them, which is (J1); and exposure at a price
$s_0$ makes the directional slope of each object along $d$ equal
$\langle s_0,d\rangle$, so the operand that contributed the direction being
constant along it forces $\langle s_0,d\rangle=0$ and the sum is constant
along it too, which is (J2). Induction over the hours gives $\mathbf J$ for
every carried object, and the primal reading turns it into (a)--(c).

That last step is what closes the worry \cref{prop:two-closed}(iii) names
below. A clamp band is a \emph{new primal vertex} with coordinates absent
from its inputs, and it may well become co-active later against pieces born
after it --- but the breakline they then realize is an edge of the primal
subdivision, whose directions and constancies are governed by $\mathbf J$,
not by the arithmetic of the gradients that produced the vertex.

\begin{proposition}[Two closed channels, one open]
\label{prop:two-closed}
\emph{(i) The addition is closed.} Let $f,g$ be convex piecewise linear with
finitely many pieces, and write $R_f$, $R_g$ for maximal affine regions of $f$ and
of $g$. The maximal affine regions of $f+g$ are exactly the non-empty
intersections $R_f\cap R_g$ --- the common refinement of the two subdivisions ---
so $\partial(R_f\cap R_g)\subseteq\partial R_f\cup\partial R_g$, and every
breakline of the sum is carried by a breakline of $f$ or of $g$. No general-position
assumption is required, which matters because degeneracy is exactly where the
structural guards fire. \emph{(ii) The linear map is closed.} The gadget's
direction catalogue is fixed, and a per-step rescaling of a state axis is diagonal,
under which the three pencils are invariant: verticals stay vertical, rays through
the origin stay rays through the origin, and only positions move. \emph{(iii) The
clamp is open, and it is the only open channel.} Its closed form gives the
surviving band a gradient component equal to a cross determinant of two existing
gradients divided by a difference of their components --- a rational absent from
the input gradients, which a later difference against another piece can turn into a
breakline direction outside the pencils.
\end{proposition}

There is therefore \emph{no} closure theorem for the position alphabet in
terms of the park data, and none should be sought there --- while the
directions are closed by \cref{thm:directions}, which is exactly the split
\cref{prop:two-closed}(iii) mechanises: the clamp's open channel is a channel
of positions, never of directions. What survives as empirical is narrow, and
it is what the cost argument actually needs. The gradient of $W^\star_t$ is
constant on a face, so a clamp frontier $\{\partial_a W^\star_t = E_a\}$ is a
union of \emph{existing} edges: no face is ever split, and the region beyond
is replaced in closed form by the linear extension of the kept piece. The
position alphabet is measured small: over a full year of TYNDP data we
measure \pairAlphabetVLow{} to \pairAlphabetVHigh{} distinct verticals and
\pairAlphabetHLow{} to \pairAlphabetHHigh{} horizontals, saturating well
before the end of the horizon --- the recurring economic levels of the hourly
curves, which the slope audit of \cref{subsec:catalogue-law} finds again as
the physical content of the traced catalogues once tracer dispersion is set
aside.

\begin{remark}[What the measured alphabet contained, and a correction to it]
\label{rem:alphabet-artefact}
An earlier census of this alphabet reported a family of ray directions outside
the predicted set and read it as a structural gap. It was an \emph{artefact of
compilation}: the kernel is header-only, hence a single translation unit, and
the compiler contracts \texttt{a*b+c} into a fused multiply--add by default on
the target platform, so adding a diagnostic block changed the arithmetic of the
recursion itself. The duplicate-gradient reduction that keeps the subdivision
canonical tests \emph{exact} equality --- correct as soon as the arithmetic is
reproducible --- so under contraction two copies of one piece stop comparing
equal, survive as two pieces, and the sliver between them is censused as a
breakline. The residual excess that survived freezing the floating-point
semantics is refuted as a feature by the exact-rational audit corroborating
\cref{thm:directions}: over \statEdgeAuditTotal{} million realized edges,
classified in rational arithmetic across adversarial instance families, none
falls outside (a)--(c), and every through-origin ray carries an exactly zero
offset. The lesson is kept because it generalises: a geometric census is an
instrument reading, and this one was wrong in every floating-point build we
measured, by a margin that moves with a single compilation flag.
\end{remark}

\subsection{Beyond two stores: the $d$-dimensional block}
\label{subsec:pair-ddim}

Neither the primal reading nor the invariant is two-dimensional in essence.
When all charge and discharge coefficients equal one common constant --- the
regime that lossless devices realise --- the same three preservation steps
run in $\R^d$, and the monotonicity hypothesis is \emph{lifted}: with linear
load maps the composition is convex whatever the hourly curve does, so
negative prices without free curtailment cost the argument nothing.

What replaces $R^{*}$ is a lattice. Call a linear subspace \emph{admissible}
when its orthogonal complement has a basis of $0/1$ vectors with pairwise
disjoint supports; writing those supports as \emph{zeros} (singletons) and
\emph{blocks} (size $\ge2$), the remaining coordinates being \emph{free},
the decomposition is canonical. Blocks are the coalitions of storages whose
levels the recursion has tied. The invariant $\mathbf J_d$ then reads:
\emph{(J1)} the orientation of every face of the primal subdivision is
admissible, and \emph{(J2)} on every face the gradient is constant within
each block --- water values are equal inside a coalition, free coordinates
being exempt. The one-dimensional admissible subspaces are exactly the axes
$\langle e_i\rangle$ and the roots $\langle e_i-e_j\rangle$, the type-$A$
root system, and that classification is the theorem.

\begin{theorem}[Realized walls of the $d$-store block, common coefficients]
\label{thm:walls-d}
Assume all charge and discharge coefficients equal one common constant, the
hourly curves finite convex piecewise linear (no monotonicity required), and
the power and level boxes non-empty and compact. Then for every hour and
every object carried by the $d$-dimensional recursion --- after each of the
$2d$ one-sided clamps and after each gadget addition, in both sweeps, free
or pinned terminal --- every realized wall of the carried subdivision has
normal \emph{(a)} an axis $e_i$, the wall being a position hyperplane
$\{s_i=\text{const}\}$ with the constant in the per-instance alphabet, or
\emph{(b)} a root $e_i-e_j$, the wall being \emph{exactly} the indifference
hyperplane $\{s_i=s_j\}$ through the origin. No genericity is assumed, and
the statement holds in any window.
\end{theorem}

\begin{corollary}[Size: an arrangement of fixed directions]
\label{cor:arrangement}
The carried subdivision is at all times a sub-complex of the arrangement of
the $d$ position families $\{s_i\in A_i\}$ and of the $d(d-1)/2$
indifference hyperplanes $\{s_i=s_j\}$; with $N=\sum_i|A_i|+d(d-1)/2$
hyperplanes, at most $O(N^d)$ faces are carried at any time. The status of
the two parts differs, exactly as in two dimensions: the \emph{directions}
--- and the count $d(d-1)/2$ --- are a theorem; the position alphabets $A_i$
are a per-instance quantity, finite and decidable by the membership guard,
\emph{not} bounded by any theorem, so the complexity reading is conditional
on that guard. The corollary bounds the carried \emph{size}, not the work of
one step.
\end{corollary}

The induction is the one above, carried through a lattice rather than a
four-element direction set: admissibility is stable under intersection with
a coordinate hyperplane (a block containing that coordinate simply shrinks)
and under sums, and in a sum every pair inside a block is joined by a chain
of pairs each inside a block of one operand --- which is what transports
(J2) across the inf-convolution. The lemmas, and the reading that turns
$\mathbf J_d$ into (a)--(b), are in \cref{app:statd}. An exact rational
$d=3$ prototype of the recursion corroborates the theorem against a
duality-certified LP judge: zero discrepancy on \ddimProbeCount{} value
probes, and \ddimWallCount{} realized walls all on axes or exact-offset
roots.

At $d=2$ the theorem recovers the lossless case of \cref{thm:directions} ---
the single root $e_1-e_2$ is the diagonal ray $\{s_1=s_2\}$ --- while
lifting the monotonicity that \cref{thm:directions} needs to cover genuinely
kinked load maps; conversely, losses and kinks at $d=2$ are covered only by
\cref{thm:directions}, so neither statement subsumes the other. Once the
coefficients differ, no closure of this kind exists at all: the direction
algebra acquires a proven infinite multiplicative orbit from $d=3$ on, and
the two-store block is exactly the case whose direction theorem survives
losses (\cref{app:statd-seam}).

\subsection{What it costs, and what it returns}
\label{subsec:pair-cost}

The step is $O(F\sigma)$ with $F$ the \emph{face} count of the arrangement and
$\sigma$ the size of the hourly gadget --- not $O(C)$ in the live cell count $C$
--- the cells the recursion's own clipping actually touches ---
which the faces exceed by \pairFaceOverCellLow{} to \pairFaceOverCellHigh{}. The trajectory is recovered by the
same Fenchel decomposition as \cref{subsec:recovery}, one dimension up: for
$s^\star\in\partial W_t(\ell)$ the sum rule gives $\ell=a-u$ with
$a\in\partial \widetilde W^\star(s^\star)$ and $u\in\partial G^\star_t(-s^\star)$,
so every primal constraint holds by construction and no reachability envelope is
needed. Because the price is a vertex of a \emph{fixed} arrangement, the active
set at that vertex is a precomputed incidence table rather than a
tolerance-ranked neighbourhood.

\Cref{tab:two-store-block} hands one instance to four engines; its discussion
belongs to the computational study (\cref{subsec:two-store-bench}). Two
structural facts belong here. The per-hour step is $O(F)$ in the face count
of the fixed arrangement, so every comparison is governed by the slope
catalogue rather than by the pair --- \cref{subsec:catalogue-law} measures
the whole axis in both reader regimes. And the decisive difference is not the
factor: an LP block returns a cost and a schedule, while
\eqref{eq:bivariate-step} returns, at
$\twoStoreFamilyMbLossless$--$\twoStoreFamilyMbLossy$\,\si{\mega\byte}, the whole
family $V^\star_0,\dots,V^\star_{T-1}$ --- the object a stochastic outer layer
(\cref{sec:stochastic}) or a marginal-value readout needs, and one an LP block
cannot produce at any price.

Whether the catalogue the step pays for can be \emph{thinned} is a fair
question, and it has a certified answer: bracketing each hourly curve
between a dropped-piece minorant and a chordal majorant on the thin
catalogue makes the two sweeps bracket the true value functions pointwise,
and refinement terminates on the exact answer with a two-sided certificate
at every round (\cref{prop:bracket}). Measured, the scheme is correct and
finds almost nothing to compress --- the hourly curves are piecewise linear,
so their curvature sits at finitely many kinks and dropping one opens a
first-order gap --- which is why it is reported in \cref{app:proofs} as an
instrument and a certificate rather than a speed-up.

\begin{remark}[Scope, and why coordination is not superseded]
\label{rem:two-store-scope}
\Cref{thm:directions} is a statement about \emph{two} stores at \emph{one}
node in operation. With $S\ge3$ the pencils become pencils of planes, and the
fixed finite arrangement survives only in the common-coefficient regime of
\cref{thm:walls-d}; across a line the
stores no longer meet only through one node's kink. The construction of
\cref{subsec:masters} covers all of those and remains the general remedy. What
\eqref{eq:bivariate-step} shows is that the smallest coupled block --- the one
that motivates the coordination layer of \cref{sec:coordination} --- admits exact
\emph{elimination} rather than exact \emph{coordination}, and that it is
cheaper to eliminate than to coordinate while returning strictly more. The
two are measured against each other in \cref{app:coordination}, where the
joint block is constructed.
\end{remark}

\begin{remark}[How an arrangement size is reported in this paper]
\label{rem:arrangement-size}
Two quantities describe the size of the carried object, and only they are
quoted here: the \emph{face count} $F$ of the fixed arrangement, and the
\emph{deduplicated line alphabet}. Both are properties of the catalogue: the
arrangement is built once, combinatorially, from the pencils, so neither
passes through the polygon routines that clip cells against the viewport. A
census of the polygon edges of the clipped cells is available from the audit
tooling and is \emph{not} reported --- it counts every interior edge twice
and is an instrument reading rather than a property of the problem
(\cref{rem:alphabet-artefact}). Accordingly, \emph{every geometric count in
this paper is reported together with the floating-point semantics it was
measured under}, and a count quoted without them is not reproducible. The
same stamp applies to the live cell count, which the recursion's own clipping
does touch.
\end{remark}

\subsection{Marginals for the pair: differentiating the recursion}
\label{subsec:marginal-joint}

The investment layer of \cref{sec:investment} read each capacity
marginal of a one-dimensional device off a forward--backward pair
(\cref{thm:marginal-capacity}). The two-store block admits no such pair,
and its capacities are not all ceilings: a converter rating is a bound of
the \emph{control} box, and an electrolyser cap is an hourly one. We
therefore obtain its marginals by a different route --- forward-mode
differentiation of the conjugate recursion itself.

A capacity enters that recursion in exactly two places: an energy capacity is
the upper bound of the hourly gradient clamp, and a power rating is a bound of
the control box read by the hourly gadget. A design direction is thus described
by how it moves six bounds, and $\partial J/\partial\theta$ is the chain rule on
those arguments, carried on the affine pieces while the recursion runs and read
off the active piece at the end. One pass; no finite difference, and no dual
trajectory.

\begin{theorem}[Genericity implies exactness]
\label{thm:generic-exact}
Call $\theta_0$ generic for direction $e$ when no active set ties: no clamp
gradient lies exactly on its bound, no two control vertices tie in value, and
the conjugate is attained at a strictly unique point. Then every selection made
by the forward pass is strict and stable on $[\theta_0,\theta_0+\delta)$, and
the pass returns the exact one-sided derivative $\partial J/\partial\theta^+$,
which is an exposed subgradient of the convex $J$.
\end{theorem}

\begin{remark}[A superset is a support for the value, not for the derivative]
\label{rem:superset-derivative}
The gadget enumerates a finite \emph{superset} of the control vertices;
redundant points are harmless for the value because they can only be dominated.
They are \emph{not} harmless for the derivative: Danskin's rule requires the
argmax of the \emph{constrained} problem, and a retained point that ties in
value at $\theta_0$ contributes a rate although it maximises at no
$\theta_0+\varepsilon$. Restricting the tie-break to branches that remain
feasible under the perturbation is what makes the reading exact.
\end{remark}

Ties are non-generic, and each admissible kind is handled by the corresponding
one-sided selection --- at the gadget, at the clamp frontier, and at the readout
--- so exactness extends from the generic set to almost every capacity. The one
condition that is not repaired this way, and that the implementation detects, is
a non-unique conjugate maximiser. Parameters that move the \emph{arrangement}
rather than a bound --- round-trip efficiencies, the slope alphabet --- are
outside this route.

\part{Implementation, results and discussion}
\section{The FLToptim package}
\label{sec:package}

The algorithms of
\cref{sec:flt,sec:longitudinal,sec:investment,sec:injection}
are implemented in \pkg, an open-source Python package with a compiled C++ core.
This section records what a user installs, the two performance-critical
design decisions behind the measurements of \cref{sec:experiments}, and the
arithmetic contract and testing discipline that carry the exactness claims
of \cref{sec:flt,sec:longitudinal} into a floating-point implementation.

\paragraph{Architecture.}
Three layers mirror the structure of the paper.
\begin{enumerate}
\item \emph{Convex-function engine.} The C++ class templates
\texttt{cplfunction} and \texttt{cpqfunction} (with vector containers for
per-period collections) implement the canonical PWL and PWQ representations
of \cref{subsec:representations} as ordered maps from breakpoint position to
slope increment, exposing the five operations of \cref{subsec:operations}
to Python through a Cython wrapper. The templates are instantiated twice:
for double precision, and for exact rational arithmetic (lazy-exact GMP
rationals), which provides the certified evaluation layer of the arithmetic
contract.
\item \emph{Longitudinal solver.} \Cref{alg:longitudinal} runs on either the
historical ordered-map backend or a flat sorted-array backend, the default,
which stores each value function contiguously. Both produce bit-identical
schedules; on annual workloads the flat backend makes the
dynamic-programming stage a measured $4.5$--$6\times$ faster.
\item \emph{Parametric simplex and coordination drivers.} The
bounded-variable parametric dual simplex of \cref{alg:trace} runs as batched
C++ kernels; the \texttt{elec} and \texttt{mr\_decompose} drivers implement
the decompositions of \cref{sec:coordination} for electricity and coupled
multi-energy networks, with hour-chunk parallel curve tracing (Axis~1) and
Jacobi storage updates (Axis~2). By default the driver streams the traced
slope--breakpoint buffers directly into the flat DP arrays, without
materialising an intermediate curve object, and evaluates the exact
multi-resource objective from the committed trace, so that monitoring cannot
perturb the basis that seeds the next sweep. It also exposes a persistent
decomposition state and native batched nodal prices for nested
conditional-gradient algorithms, terminal-seeded backward value functions, and
per-step level pins on storages: a level pin fixes the cumulative state at
chosen periods and so models demand-side management as a shiftable-load band
conserved over each window (\cref{subsec:multienergy-bench}).
\end{enumerate}

\paragraph{Dependencies and interfaces.}
The runtime dependencies are \texttt{numpy} and \texttt{highspy}; HiGHS
supplies cold basis construction and independent spatial reference solves,
and Gurobi is used only as an optional frontal reference. The package
builds from source with any C++17 compiler on Python 3.9--3.13, and the
source distribution is the reference artefact: the build enforces
\texttt{-ffp-contract=off}, without which the geometric counts of
\cref{sec:pair} are not reproducible (\cref{rem:arrangement-size}). The
measurements in this paper were taken against a source build of the release
named in each caption rather than against a downloaded wheel.
Since version 0.10.0 the package also ships model-coupling adapters
(\texttt{FLToptim.interfaces}, lazily imported optional extras that leave
the \texttt{numpy}+\texttt{highspy} core unchanged): a mapping from
\texttt{pommes}~\cite{pommes} single-year operation datasets --- the path
behind the coupled case study of \cref{subsec:multienergy-bench} --- and a
bidirectional adapter for PyPSA/PyPSA-Eur electricity
networks~\cite{pypsa} (capacities pinned, network in transport relaxation,
dispatch and nodal prices written back), each validated against an exact
HiGHS monolith of the same case. The same module exposes a capacity-expansion
oracle, which consumes a planning dataset as data and returns the cost and the
capacity marginals so that a Benders master can drive it. Its behaviour at
system scale --- many capacities, many nodes --- is outside the scope of this
paper (\cref{subsec:sizing-scope}) and the package documents it separately.

\paragraph{Arithmetic and degeneracy contract.}
\label{subsec:contract}
An exact mathematical algorithm needs a precise numerical contract; the
word \emph{exact} is used at three distinct levels:
\begin{enumerate}
\item exact real-arithmetic identities and representation invariants, as
stated in \cref{prop:closure,prop:opcost};
\item certified evaluation for rational input data, provided by the
exact-rational instantiation of the same compiled templates;
\item floating-point evaluation with fixed comparison and merging rules,
used in all benchmarks of \cref{sec:experiments}.
\end{enumerate}
Equal slopes, coincident breakpoints and zero-width pieces are resolved by
the canonical merging of \cref{subsec:representations}, which leaves the
represented function unchanged; no perturbation or tie-breaking jitter is
invoked anywhere in the paper to justify exactness for an original
instance.

\paragraph{Why the stochastic layer is a package operation and not a recipe.}
The expectation of \cref{sec:stochastic} and the node merge of
\cref{alg:tree-sweep} are single calls on the weighted pointwise sum, and the
tree sweep itself is driver-level control flow over the terminal-seeded
backward pass listed above --- state handoff, terminal sum, minimisation per
node. That they are package operations rather than a documented recipe is a
correctness argument, not a convenience one. A first prototype assembled the
recourse from public accessors and did not honour the reflected-state
convention of \cref{subsec:forward-backward}: breakpoints were not negated on
the backward curves. The result was a boundary state $\sigma^\star$ wrong by
$8\cdot10^{-4}$ relative on one random seed and \emph{exact} on another ---
silent, seed-dependent, and invisible to any check that does not vary the
seed. Once reflection lives inside the operation, the failure is impossible by
construction rather than caught by testing.
The terminal-seeded backward pass the sweep composes ships with \pkg{}; the
per-node probability driver remains research code at \pkgVersion{},
exercised through the versioned campaign behind \cref{subsec:tree-bench}
rather than as a first-class package runner --- a packaging boundary, not a
correctness one, and it is stated as such.

\paragraph{Testing and degeneracies.}
The unit-test suite exercises the degenerate configurations of the contract
directly --- equal slopes, coincident breakpoints, zero-width pieces,
unbounded domains --- together with conjugation round-trips, high-precision
rational runs, parametric traces validated against independent LP solves,
and end-to-end decomposition regressions. Release gates rerun the installed
wheels on the benchmark instances (\cref{app:campaigns}).

\paragraph{Availability and reproducibility.}
\pkg{} is distributed under the MIT license; \pkgVersion{} is the current
release and the archived source distribution is its citable form. The
measurements reported here were taken at \pkgCampaignVersion{}, from which
\pkgVersion{} differs only in tests, documentation and packaging metadata ---
no numerical code changed between the two, which a reader can check in one
command against the two tags. The operation and reuse measurements of
\cref{sec:experiments} come from a single unified campaign against version
\pkgCampaignVersion{} (\cref{subsec:protocol,app:campaigns}), superseding the
earlier per-column archived versions. The coordination-refinement
measurements of \cref{subsec:ramps-bench} come from their own named release,
stated there.

The experiment driver, instance generators, seeds, archived campaigns and
table generators live in a separate benchmark repository. Its measurements are
a Snakemake workflow keyed by package version, and the tables of
\cref{sec:experiments} are \emph{generated} from the archived result files
rather than transcribed --- including the figures the prose quotes, which the
manuscript reads as macros, so that a number cannot drift between the campaign
that measured it and the sentence that cites it.
The release is archived at
DOI~\href{https://doi.org/10.5281/zenodo.21943668}{10.5281/zenodo.21943668}
--- one record carrying the package, the complete archive of the campaigns
behind every table, the exact-rational instruments corroborating
\cref{thm:directions,thm:walls-d}, and the bivariate-kernel harness --- and
the development repositories are public:
\url{https://git.persee.minesparis.psl.eu/energy-alternatives/FLToptim}
(tag \texttt{v\pkgVersion}),
\url{https://git.persee.minesparis.psl.eu/energy-alternatives/fltoptim_studies/benchmarks},
and
\url{https://git.persee.minesparis.psl.eu/energy-alternatives/fltoptim_studies/prototypes},
which carries the two exact-arithmetic instruments the appendices rely on ---
the rational audit of the two-store recursion and the $d=3$ prototype
(\texttt{stationarity\_exact/}), and the bivariate kernel harness whose
results \cref{tab:two-store-block} reads
(\texttt{bivariate\_kernel/}).

\section{Computational study}
\label{sec:experiments}

The study is organised by the state dimension the reader carries, mirroring
the two parts of the paper. A first group of six experiments stays
one-dimensional: the longitudinal kernel in isolation, parametric families
and reuse, two-stage recourse, multistage scenario trees, a single reservoir
on real data, and one storage inside an annual network dispatch. A second
group raises the state dimension to two with the coupled two-store block.
The catalogue-size law then reads the \emph{same} annual curves with both
readers, which is why it comes after the two groups rather than between
them; the section closes with the validation of the oracles and the observed
failure modes. The one system-scale case that
remains here is the exact one --- a single storage inside an annual network
dispatch (\cref{subsec:one-node-bench}); the system-scale coordination study,
where past one node coordination rather than the recursion sets the runtime,
is in \cref{app:coordination}. The synthetic experiment of
\cref{tab:intro-storage} isolates the kernel over horizons long enough to
expose its scaling; every other experiment uses one full year
($T=8{,}760$) of demand, renewable availability and network data derived
from ENTSO-E TYNDP~\cite{TYNDP2024}. This separation keeps the opening claim
simple while making the system-level comparison representative of energy
applications.

\subsection{Protocol}
\label{subsec:protocol}

The common protocol reports end-to-end wall time, peak resident memory,
objective quality against the frontal optimum, and physical feasibility of
the returned schedules. The frontal references use the natural sparse
state-of-charge formulations in HiGHS 1.14.0 and Gurobi 13.0.2; every FLT
call uses the flat-array value-function backend, and the ramp-free network
cases use the predictive shared hourly basis atlas, a monotonicity guard and
storage-local pruning continuation (\cref{sec:package,app:simplex}).
The system-scale tables --- the one-node row here, and those of
\cref{app:coordination} --- come from a
single unified campaign against the \pkg{} \pkgCampaignVersion{} release,
timed by a common driver that measures Gurobi, HiGHS and the decomposition
with identical accounting, end-to-end and solve-only both recorded; the
remaining tables read their own explicitly dated campaign files, named in
each caption (\cref{app:campaigns}); the machines and the
reference-objective transfer to the license-free server are recorded in
\cref{app:campaigns}. Every table and figure in this section is generated from an archived
result file rather than transcribed (\cref{sec:package}).

\subsection*{One-dimensional results}

The six subsections that follow are the experiments in which the
value-function recursion does \emph{all} the work: the state is a single
storage, so the recursion of \cref{alg:longitudinal} and the merges built on
it carry the whole computation, with no coupled block anywhere and --- except
in the last one, which adds one exact spatial trace --- no network. Three of
them sit deliberately in that regime with no network at all: the isolated
kernel (\cref{subsec:kernel-bench}), the scenario trees
(\cref{subsec:tree-bench}) and the single reservoir operated without
foresight (\cref{subsec:water-value-bench}).

\subsection{The longitudinal kernel in isolation}
\label{subsec:kernel-bench}

\Cref{tab:intro-storage} compares \pkg{} and Gurobi on the same sparse
price-taking storage instances with continuous random prices and identical
physical bounds. The optima agree to \kernelRelGap{} relative; the
end-to-end median advantage grows from $\kernelSpeedupLow\times$ at
$T=\kernelHorizonLow$ to $\kernelSpeedupHigh\times$ at
$T=\kernelHorizonHigh$, and annual hourly instances take milliseconds.
The gain comes entirely from replacing a global LP solve by the ordered
merges of \cref{alg:longitudinal}: no approximation, no state grid, no
relaxed physics, no specialised hardware.

\Cref{fig:kernel-scaling} extends the same kernel over two decades of
horizon and separates the two questions a speed-up figure conflates. The left
panel is the wall time, which reaches $T=\kernelScalingMaxN$ in
$\kernelScalingMaxTime$~seconds --- the measurement behind the abstract's
claim. The right panel is the per-step cost, which is the informative one: it
stays between $\kernelScalingUsLow$ and $\kernelScalingUsHigh$
microseconds per step across a $\kernelScalingSpan\times$ range of $T$, so
the recursion is near-linear with a small constant rather than merely fast at
one size. Both panels also carry the two certified arithmetic layers of the
contract in \cref{subsec:contract} --- $128$-bit floating point and exact
rationals --- which cost a factor of
$\kernelCertifiedFloatLow$--$\kernelCertifiedFloatHigh$ and
$\kernelCertifiedExactLow$--$\kernelCertifiedExactHigh$ respectively, and
preserve the same shape.

\begin{figure}[t]
\centering
\includegraphics[width=\linewidth]{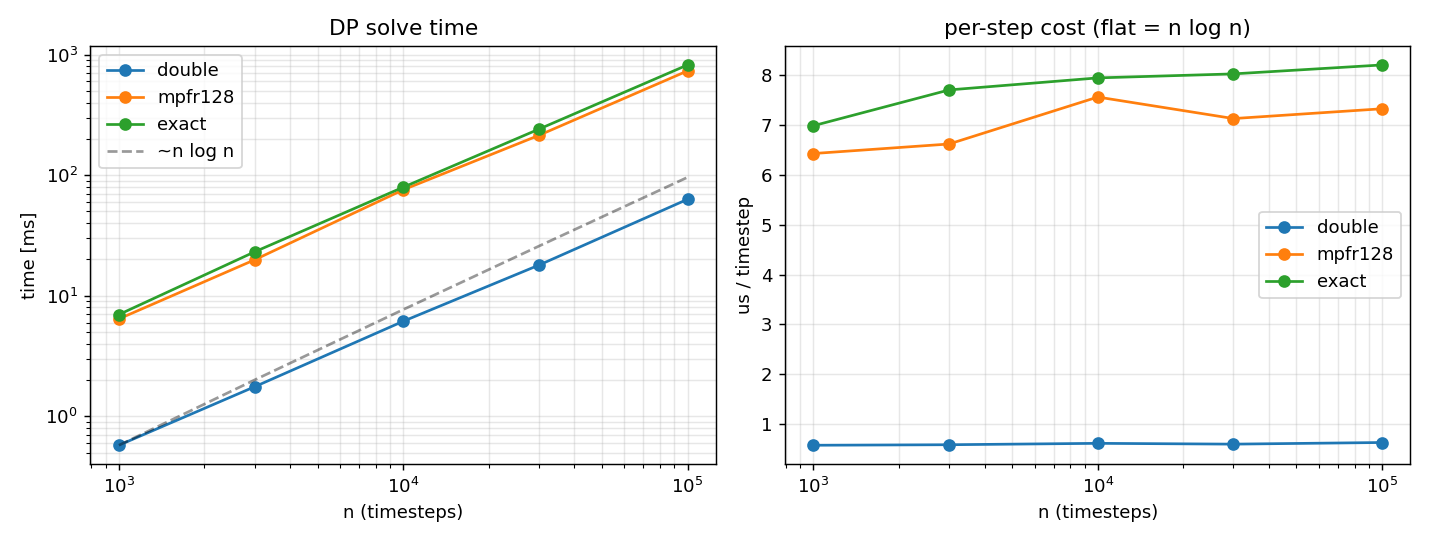}
\caption{Longitudinal kernel against horizon, at the three arithmetic layers of \cref{subsec:contract} (double, exact, mpfr128). Left: solve time, log--log, against an $n\log n$ guide. Right: microseconds per step, which is the panel that decides the question --- a \emph{flat} curve is the signature of the bound, where the left panel alone cannot separate near-linear scaling from being merely fast at one size. Double precision stays inside $[0.576, 0.631]$~$\mu$s/step over a $100\times$ range of $T$; the two certified layers cost a bounded multiple of it and keep the same shape, so exactness changes the constant and not the asymptotics. (\pkg{} 0.27.0, measured 2026-08-15.)}
\label{fig:kernel-scaling}
\end{figure}

The most informative specialised comparator is the monotonic
divide-and-conquer algorithm (MDA) of Vidal et al.~\cite{Vidal2019}: for the
linear case with one nested constraint per period both methods carry an
$O(T\log T)$ bound, with different mechanisms and different outputs
(\cref{subsec:related}). The benchmark uses identical prices and
bounds; the free terminal state converts to their fixed-total-resource
convention by one bounded zero-cost slack $z\in[-E,0]$ with
$\sum_tx_t+z=0$, which changes neither the feasible trajectories nor the
objective; and it reports separately (i) time and memory to one optimal
schedule and (ii) time and memory to the complete terminal value function,
so that a schedule-only solve is not credited with the richer output. No
public implementation is associated with the article and a written request to
the authors (2026-07-15) has gone unanswered, so the comparison runs against
a \emph{verified} reimplementation carried by the benchmark repository: a
compiled MDA implemented from the article, exact on piecewise-linear costs,
validated on several hundred randomised RAP--NC instances against an
independent LP judge with zero monotonicity violations, shipped with a
fidelity ledger recording every interpretation choice --- and superseded by
the authors' code the day it becomes available. On the storage instances of
\cref{tab:intro-storage}, between \mdaHorizonLow{} and \mdaHorizonHigh{}
periods, the specialised comparator is genuinely strong ---
\mdaVsGurobiRange$\times$ faster than the Gurobi solve, at zero objective
deviation --- and the division of outputs is what the measurement isolates:
the FLT chain returns the optimal schedule
\fltScheduleVsMdaRange$\times$ faster than MDA returns its own, and returns
the \emph{complete} terminal value function --- the object every construction
after \cref{subsec:parametric-reuse} consumes --- still
\fltValueFunctionVsMdaRange$\times$ faster than MDA's schedule alone. The
specialised race sharpens rather than weakens the division of labour: on the
schedule, everyone is exact and the chain is fastest; on the function, the
chain is alone.

\subsection{Parametric families and reuse}
\label{subsec:parametric-reuse}

The value-function output matters when the required answer is a function or
a family rather than one schedule. We first trace the complete capacity
curve of an annual price-taking store at all \reuseCurveSolves{} layer boundaries
($p=1$, $E\in[0,\reuseLayers]$), then solve \reuseFamilyK{} independent annual
price scenarios.
The capacity experiment validates every layer value against a warm Gurobi
model and every layer marginal against the duals of explicit capacity
constraints. Maximum relative value error is $4.9\,10^{-15}$ and maximum
absolute marginal error is $4.5\,10^{-11}$.

\begin{table}[t]
\centering
\caption{One parametric solve reused, against re-solving at a grid of capacities. Top: the complete 60-layer capacity curve of a single storage ($T=8{,}760$), where a uniform grid must trade missed breakpoints against wall time. Bottom: 100 independent price scenarios, Gurobi's best case for warm starting since the model shape never changes. Agreement is exact to solver tolerance throughout (max relative gap $6.5\,10^{-15}$, marginal-value curve within $2.3\,10^{-10}$ of the parametric one). (\pkg{} 0.27.0, machine \texttt{cyrius-2.home}.)}
\label{tab:parametric-reuse}
\begin{tabular}{lrrr}
\toprule
task and method & solves / instances & wall [s] & missed layers \\
\midrule
complete 60-layer curve, \pkg{} flat backend & 61 & \textbf{0.320} & 0 \\
Gurobi warm uniform grid & 11 & 0.384 & 55 \\
Gurobi warm uniform grid & 41 & 0.636 & 20 \\
Gurobi warm uniform grid & 161 & 1.000 & 0 \\
Gurobi cold uniform grid & 41 & 13.530 & 20 \\
\midrule
100 price scenarios, \pkg{} & 100 & \textbf{0.559} & -- \\
100 price scenarios, Gurobi warm & 100 & 1.713 & -- \\
100 price scenarios, Gurobi cold & 100 & 32.357 & -- \\
\bottomrule
\end{tabular}
\end{table}

Thus the complete exact FLT curve is $\reuseCurveVsExactGrid\times$ faster than
the smallest tested warm LP grid (\reuseGridExactSize{} points) that resolves all
\reuseLayers{} layers; coarse LP grids are not
equivalent outputs. Once this curve exists, a new investment cost is answered
in under $4.6\,10^{-5}$ seconds. Without tracing the curve, tangent search
locates the same three tested optima in 11--16 FLT evaluations
(0.202--0.785 seconds), versus 0.915--0.932 seconds for one Gurobi
co-optimisation per investment cost. Across \reuseFamilyK{} price scenarios,
FLToptim is $\reuseSpeedupWarm\times$ faster than the favourable warm-Gurobi
comparison and $\reuseSpeedupCold\times$ faster than cold rebuilds; objectives
agree to \reuseRelGapMax{} relative. These experiments support the reuse claim without
asserting that tracing a complete curve is necessary for a single sizing
query.

\paragraph{Joint energy--power sizing.}
Two instances exercise \cref{subsec:joint-bundle} at $T=8{,}760$; they are
research-harness measurements on the same machine, not yet folded into the
unified campaign of \cref{subsec:protocol}. For a
price-taking store on the same annual prices, homogeneity collapses the
joint problem to the duration: the one-dimensional search returns
$d^\star=34$\,h in 13 FLT solves and 0.53 seconds, and the
capacity-and-power co-optimising LP agrees exactly (0.29 seconds) --- the
reduction, not a two-dimensional method, is the right algorithm there. On
a price-responsive single node (four-technology merit order), where
homogeneity fails, the level bundle with the warm-started monolith oracle
matches the co-optimising optimum to a relative gap of $4\,10^{-7}$ in 17
cuts and 18 oracle solves: 0.48 seconds against 1.29 seconds for the
sparse co-optimising LP (HiGHS). The duration warm start ($\hat d=4$\,h
for a true $2.45$\,h) is the right order of magnitude but saves at most
one cut --- the level stabilisation is already nearly start-insensitive.

\subsection{Two-stage recourse: exactness and scaling in $K$}
\label{subsec:stochastic-validation}

The 100-scenario family above is solved scenario by scenario; the two-stage
algorithm of \cref{alg:two-stage} instead consumes such a family as the
recourse stage of a single decision. Two measurements validate it. The
instance splits a price-taking store over $T=\twoStageHorizon$ steps ---
$T_1=\twoStageSplitTOne$ deterministic hours, $T_2=\twoStageSplitTTwo$
recourse hours; power \twoStagePowerCap{} and energy \twoStageEnergyCap{}
(a four-hour store started half full); stage-2 prices \twoStagePriceGen{} ---
into a deterministic stage~1 and $K$ equiprobable
stage-2 price scenarios, $p_w=1/K$, for $K$ up to $\twoStageMaxK$, the
deterministic equivalent solved and judged at every tested $K$.

First, \emph{exactness}. The boundary state $\sigma^\star$ and the optimal
value returned by \cref{alg:two-stage} are compared against the
deterministic equivalent of~\eqref{eq:two-stage} --- one linear program with
$K$ copies of the stage-2 block sharing the stage-1 variables --- built and
solved by Gurobi under the protocol of \cref{subsec:protocol}. The maximum
relative objective difference over the tested $K$ is \twoStageMaxRelGap{},
and the boundary states agree \emph{exactly}
($\max|\Delta\sigma^\star|=\twoStageMaxSigmaGap$), i.e.\ agreement at machine
precision, as \cref{prop:two-stage-exact} predicts.

Second, \emph{scaling in $K$}. \Cref{tab:two-stage} reports, for each $K$, the
wall time, the record count of $\Qexp$, and the end-to-end time of the
deterministic equivalent, whose cost grows with the $K$-fold constraint
matrix. The striking entry is the record count: $\Qexp$ carries
\twoStageRecords{} breakpoints at every tested $K$, against the
$\sum_wm_w+O(1)$ bound of \cref{cor:poscomb} --- the merge of many scenario
cost-to-go functions produces an object no larger than one of them, because
the pointwise average of curves that agree on their slope ordering adds no
kink. The time columns move together: both the recursion and the
deterministic equivalent grow linearly in $K$ --- the multi-scenario LP is
block-diagonal --- so the time advantage the table shows is a roughly
constant factor at this horizon, not a widening one; the structural asset is
the constant record count, which is what the masters of
\cref{sec:investment} and the sweeps of \cref{subsec:trees} consume.
Scenario passes are run serially so that the comparison measures work
rather than parallelism.

\begin{table}[t]
\centering
\caption{Two-stage recourse at $T=8{,}760$ (deterministic stage~1 of $T_1=4{,}380$ steps, $K$ equiprobable stage-2 scenarios of $T_2=4{,}380$ steps). Times are end to end and serial; the FLT total is split into the $B^w$ passes (forward + $K$ backward) and the $K$-way ordered merge assembling $\Qexp$. ``records'' is the breakpoint count of $\Qexp$, bounded by $\sum_wm_w+O(1)$ (\cref{cor:poscomb}) --- observed constant in $K$, far below that bound. $\sigma^\star$ agrees with the deterministic equivalent exactly on every judged $K$.}
\label{tab:two-stage}
\begin{tabular}{rrrrrrr}
\toprule
$K$ & $B^w$ passes [s] & merge [s] & total [s] & records of $\Qexp$ & det.\ equivalent [s] & rel.\ objective gap \\
\midrule
10 & 0.063 & 0.000 & 0.063 & 5 & 1.699 & $5.9\,10^{-15}$ \\
20 & 0.120 & 0.000 & 0.121 & 5 & 3.180 & $10^{-14}$ \\
50 & 0.302 & 0.000 & 0.302 & 5 & 7.502 & $1.3\,10^{-14}$ \\
100 & 0.635 & 0.000 & 0.636 & 5 & 15.172 & $3.4\,10^{-15}$ \\
200 & 1.303 & 0.000 & 1.304 & 5 & 31.597 & $3.1\,10^{-14}$ \\
500 & 3.003 & 0.001 & 3.004 & 5 & 85.520 & $2.5\,10^{-14}$ \\
\bottomrule
\end{tabular}
\end{table}

\subsection{Multistage scenario trees: exactness and cost in tree size}
\label{subsec:tree-bench}

The two-stage measurement above merges once. \Cref{alg:tree-sweep} merges at
every node, and the question it raises is whether exactness and cost survive
the recursion. Weekly stages branch over depth$\times$branch
$\{3{\times}2,\,2{\times}3,\,4{\times}2,\,2{\times}5\}$, up to $31$ nodes and
$5{,}208$ tree-hours. Against the deterministic-equivalent LP the sweep agrees
to a worst relative objective difference of \treeExactRelGapMax{} --- machine
precision, as \cref{prop:tree-exact} predicts --- in
\treeWallLow--\treeWallHigh~ms, and the
wall time is \emph{flat} in tree size over that range while the LP grows with
it --- consistent with the accounting of \cref{rem:tree-cost}, though these
absolute times are below the resolution at which a growth
exponent can be read.
A separate control isolates the seam between two stages, which is the step the
induction of \cref{prop:tree-exact} actually turns on: seeding a stage with the
next stage's cost-to-go returns the same object as one backward pass over the
concatenation of the two, to a worst slope difference of exactly $0.0$.

This is the one regime in the paper where the value-function recursion does
all of the work. There is one storage, no network and no coupled block, so
neither the parametric trace nor a coordination sweep is involved, and what
the measurement exercises is the recursion of \cref{alg:longitudinal}
composed with the merges of \cref{cor:poscomb} and nothing else.

\subsection{Driving a year with the curves: a single reservoir on real data}
\label{subsec:water-value-bench}

\Cref{subsec:tree-bench} is synthetic. The complementary question is whether
the curves are good enough to \emph{operate} a year of real data without
foresight, which is what \cref{cor:rolling-exact} and \cref{rem:usage-value}
claim. The instance is the Norwegian reservoir of the TYNDP case taken
\emph{alone}, on its own frozen hourly dispatch-cost curves, scored against
its own perfect-foresight optimum on a scarce year; the policy prices the
reservoir at each weekly boundary with a terminal valuation and never sees
beyond it.

Four terminal valuations are compared. The exact in-year cost-to-go $B$
reproduces perfect foresight to a relative gap of
\waterValueExactGapSingleReservoir{} --- the identity of
\cref{cor:rolling-exact}, measured, and the reason the remaining columns
are pure information error. The expectation $\Qexp$ over
\waterValueEnsembleSize{} weather years
operates at \waterValueExpectedGap. In the leave-one-out test, where the
scarce year is
never seen by the expectation that prices it, the gap is \waterValueLooGap,
against
\waterValueWetYearGap{} for a single wet year --- whose water value is near
zero, so the
policy over-drains --- and \waterValueMyopicGap{} for a myopic policy with no
terminal
valuation at all. The expectation, not any single year, is the robust signal,
and this is the measured form of the argument of \cref{rem:usage-value}: the
scarcity floor that regressed policies add by hand is unnecessary once the
scarcity is priced by an expectation over futures.

The regime is again one storage, frozen curves, no cross-block coupling.

\subsection{One node inside a network: exact composition}
\label{subsec:one-node-bench}

The one-node case retains the French demand, generation and storage data
from the three-node TYNDP instance but removes interconnection. It is
therefore an annual system-dispatch problem rather than the price-only
synthetic test above, and it instantiates \cref{prop:one-storage}: one exact
spatial trace followed by one FLT solve returns the global optimum. This
composition reproduces the frontal optimum to numerical precision while
running $\annualOneNodeGurobiSpeedup\times$ faster than Gurobi and
$\annualOneNodeHighsSpeedup\times$ faster than HiGHS, in
$\annualOneNodeFltRss$~MB of peak memory. Both ratios are end-to-end; against
Gurobi's solver-only time the composition is still the faster of the two at
this single-node size. This is the last size at which the method is exact and
fastest at once: the row sits in \cref{tab:annual-comparison}, whose other
columns add storages and therefore leave the exact regime, and which is
reported with the coordination layer it measures in
\cref{app:coordination}.

\subsection*{Two-dimensional results}

Raising the state dimension to two replaces the chain by the planar
arrangement of \cref{subsec:two-store-conjugate}. The subsection below hands
one instance --- one node, one year, two co-located stores of different
duration --- to the same frontal references used above, and the block is
exact and fast on it. What that arrangement \emph{costs} is not settled by
one instance: it is the subject of the synthesis that follows.

\subsection{Eliminating a two-store block}
\label{subsec:two-store-bench}

Two co-located stores of \emph{different} duration are the smallest coupled
block, the one a plain alternation locks on, and the one
\cref{subsec:two-store-conjugate} eliminates rather than coordinates.
\Cref{tab:two-store-block} hands one instance to three engines: one node, a full
year of TYNDP FR data, a \SI{4}{\hour} battery and a \SI{100}{\hour} reservoir,
solved by the Gurobi monolithic LP, the HiGHS monolith, and the bivariate
conjugate recursion. All three return
the same optimum, the conjugate route to $\twoStoreUlpMax$ units in the last
place of the correctly rounded answer on the exact-rational test library and
$0$ relative deviation here. The frontal solvers are the comparators
throughout this paper; the LP joint block, which coordinates this same pair
instead of eliminating it, is measured against the recursion in
\cref{app:coordination}, where it is constructed.

\begin{table}[t]
\centering
\caption{The two-store joint block, one instance given to four engines: one node, a full year of TYNDP FR data, a 4\,h battery (20\,GWh/5\,GW) and a 100\,h reservoir (500\,GWh/5\,GW). The duration ratio is what breaks naive aggregation and the reason the joint block exists at all. \emph{Rel.\ dev.} is the relative deviation of the objective from the HiGHS monolith on the same case: every engine returns the same optimum. The Gurobi comparison is deliberately conservative --- its \emph{solve} time against the conjugate route \emph{end to end}, model construction, trajectory readout and value-function retention included; counting Gurobi's own model build would raise the lossless factor to $\times\twoStoreGurobiSpeedupWithBuild$. Only the conjugate route returns $V^\star_0..V^\star_{T-1}$, at $\twoStoreFamilyMbLossless$--$\twoStoreFamilyMbLossy$\,MB. Face counts $\twoStoreFacesLossless$ (lossless) and $\twoStoreFacesLossy$ (lossy), measured under \texttt{-ffp-contract=off} (\cref{rem:arrangement-size}). (\pkg{} 0.27.0, median of 3, measured T=8760.)}
\label{tab:two-store-block}
\begin{tabular}{lrrrl}
\toprule
engine & lossless [s] & lossy [s] & rel.\ dev. & returns \\
\midrule
Gurobi, monolithic LP (solve only) & 0.202 & 0.291 & $0$ & cost + schedule \\
\quad (incl.\ model build) & 0.799 & 0.490 &  &  \\
HiGHS monolith & 0.680 & 0.674 & $0$ & cost + schedule \\
\midrule
Bivariate conjugate, value only & 0.0073 & 0.0699 & $0$ & cost \\
\textbf{Bivariate conjugate, full} & \textbf{0.0132} & \textbf{0.138} & $\mathbf{2\,10^{-15}}$ & \textbf{cost + schedule + $V^\star$} \\
\bottomrule
\end{tabular}
\end{table}

Against Gurobi's \emph{solve} time --- counting the conjugate route end to end,
model construction, trajectory readout and value-function retention all included,
which is the conservative direction --- the factor is \twoStoreGurobiSpeedupRange{}
($\times\twoStoreGurobiSpeedupWithBuild$ if Gurobi's own model build is counted,
which we do not). The
spread between the lossless and lossy figures is the arrangement, not the data:
round-trip losses split each pencil in two and turn the origin rays into three,
so $F$ grows from \twoStoreFacesLossless{} to \twoStoreFacesLossy{} faces and
the step grows with it. Face counts are
the quantity quoted, for the reason given in \cref{rem:arrangement-size}: they are
a property of the fixed arrangement rather than of the clipping routines, and like
every geometric count in this paper they are reported under the floating-point
semantics they were measured with. They are also the price signal: the
per-hour step is $O(F)$, and the factor above does not survive arrangements
two orders of magnitude larger --- the synthesis of \cref{subsec:catalogue-law}
below measures the reversal across the whole catalogue axis,
\catLawTwoDLossRange$\times$ behind the frontal on traced annual network
curves. The comparison is governed by the catalogue, not by the pair.

The trajectory is exact --- the recovered schedule prices to the announced value
to $2\,10^{-15}$ relative, and the readout passes
\twoStoreReadoutInstances{} instances of the exact-rational feasibility and
optimality gate. What the table cannot show is the asymmetry that motivates the
construction: the three LP engines return a cost and a schedule, while the
conjugate route additionally returns
$V^\star_0,\dots,V^\star_{T-1}$ at
$\twoStoreFamilyMbLossless$--$\twoStoreFamilyMbLossy$\,\si{\mega\byte}.
\Cref{rem:two-store-scope} states the price of that: two stores, one node,
operation only.
Equal-duration co-located storages need no elimination and no refinement at
all: their aggregate is one storage with summed energy and power.

\subsection{The catalogue-size law: the two regimes on the same curves}
\label{subsec:catalogue-law}

This subsection is the synthesis of the two groups above: it hands the
\emph{same} curves to both readers, so it belongs after both have been
introduced. The two regimes are two readings of one variable carried by the
per-step cost --- the size of the slope catalogue. \Cref{fig:catalogue-law}
and \cref{tab:catalogue-law} read the same traced annual network curves
($T=8{,}760$, campaign of \catLawCampaignDate{} at
\pkg{}~\catLawCampaignVersion) once with the one-dimensional chain of
\cref{subsec:kernel-bench} and once with the planar block of
\cref{subsec:two-store-bench}, exactness checked at every point before any
timing is recorded.

\begin{figure}[t]
\centering
\includegraphics[width=\linewidth]{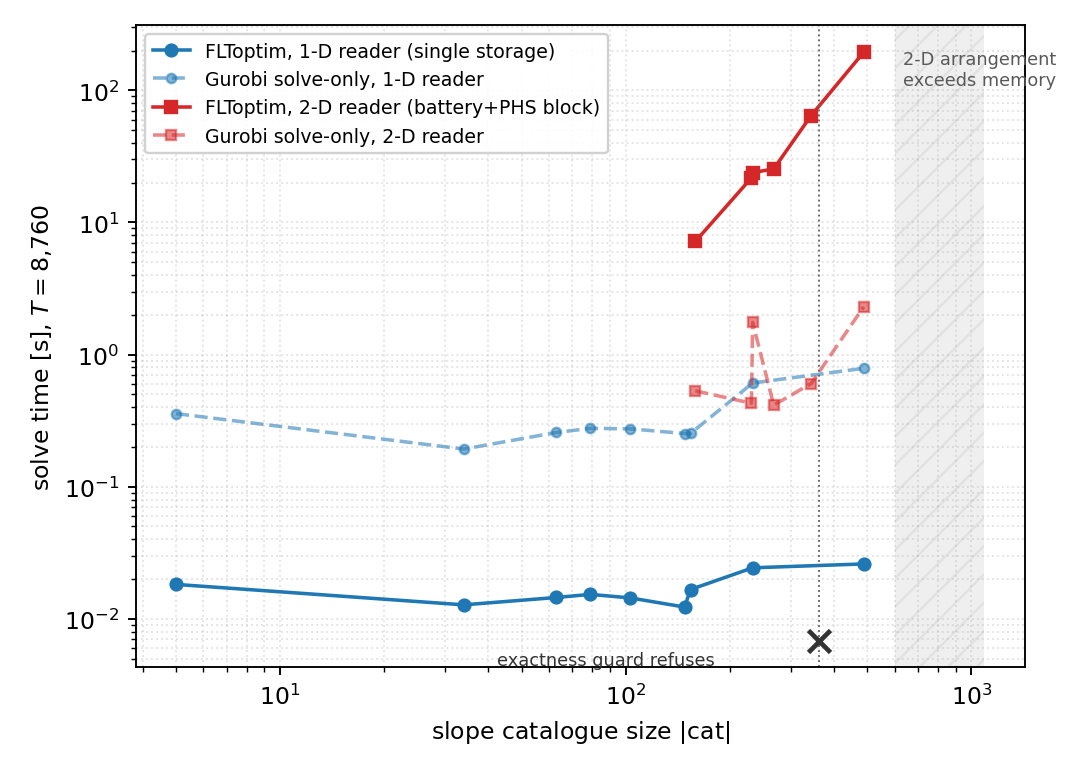}
\caption{Solve time against slope-catalogue size $|\mathrm{cat}|$ at $T=8{,}760$, log--log, two reader regimes on the same convex piecewise-linear curve family. FLToptim (1-D, single storage) beats Gurobi solve-only by $15.1$--$30.2\times$ over $|\mathrm{cat}|\in[5,488]$; the same kernel reading a 2-D battery+PHS joint block loses by $13.5$--$107\times$ over $|\mathrm{cat}|\in[159,488]$ -- at $|\mathrm{cat}|=233$ the same curve is 25.0$\times$ faster read as 1-D and 13.6$\times$ slower read as 2-D. The exactness guard refuses one 2-D point ($|\mathrm{cat}|=363$, joint solution $6.4\,10^{-9}$ suboptimal) rather than return an uncertified answer; the hatched band marks where the 2-D arrangement ($F\sim|\mathrm{cat}|^2$) exceeds this machine's memory. No line is drawn between the two regimes: the figure shows the measured envelope, not an extrapolated crossover. (\pkg{} 0.27.0, measured 2026-08-13.)}
\label{fig:catalogue-law}
\end{figure}

\begin{table}[t]
\centering
\caption{The catalogue-size law: the same traced annual network curves
($T=8{,}760$) read by a one-dimensional chain and by the two-dimensional
block, against the same frontal (Gurobi) solve. Exactness is checked at every
point before a timing is recorded. (\pkg{}~\catLawCampaignVersion, campaign of
\catLawCampaignDate.)}
\label{tab:catalogue-law}
\small
\begin{tabular}{@{}lll@{}}
\toprule
 & one-dimensional chain & two-store planar block \\
\midrule
state dimension read & $1$ (\cref{alg:longitudinal})
  & $2$ (\cref{subsec:two-store-conjugate}) \\
catalogue range measured & \catLawOneDCatMin--\catLawOneDCatMax{} slopes
  & \catLawTwoDCatMin--\catLawTwoDCatMax{} slopes \\
per-step cost law & carried record count (\cref{rem:tseng-luo})
  & constructed faces, $F\approx|\mathrm{cat}|^{2}$ \\
time vs.\ frontal solve & \catLawOneDGainRange$\times$ ahead, gain growing
  & \catLawTwoDLossRange$\times$ behind \\
the same \catLawDoubleCat-slope curve
  & \catLawDoubleOneDTime~ms vs.\ \catLawDoubleOneDGurobiTime~ms
    (\catLawDoubleOneDRatio$\times$ ahead)
  & \catLawDoubleTwoDTime~s vs.\ \catLawDoubleTwoDGurobiTime~s
    (\catLawDoubleTwoDRatio$\times$ behind) \\
worst exactness deviation & \catLawOneDExactnessMax
  & \catLawTwoDExactnessMax \\
\bottomrule
\end{tabular}
\end{table}

The quadratic law of the third row is checked directly on the constructed
arrangement: \catLawFCheckCatLow{} slopes give \catLawFCheckMeasuredLow{}
faces against \catLawFCheckPredictedLow{} predicted, and
\catLawFCheckCatHigh{} give \catLawFCheckMeasuredHigh{} against
\catLawFCheckPredictedHigh.

Four facts sharpen the table. \emph{(i) The boundary is a change of reader
dimension, not a crossover on the catalogue axis.} The same
\catLawDoubleCat-slope curve is read three orders of magnitude apart by the
two readers on identical data (penultimate row): no within-regime crossover
is observed on the measured axis and none is extrapolated. Nor can the two
regimes be brought to a common small-catalogue point at the year, because no
symmetric charge-and-discharge network curve below about
\catLawTwoDCatMin{} slopes, as the kernel counts them at float tolerance,
exists on the TYNDP nodes --- the charging side always traverses the thermal
merit order --- so the block's losing regime on these curves is not an
unlucky instance. \emph{(ii) The exactness guard refuses an instance rather
than return it}: at \catLawRefusedCat{} slopes the block detects a
\catLawRefusedSubopt{} suboptimality on a non-convex kink join and refuses.
\emph{(iii) The planar arrangement carries a memory wall of its own}: between
\catLawMemWallCatLow{} and \catLawMemWallCatHigh{} slopes its
$O(|\mathrm{cat}|^{2})$ storage demand (of order \catLawMemWallMachineGb~GB
here) exhausts the machine while the frontal solve answers in seconds. The
memory frugality reported elsewhere in this section belongs to the
one-dimensional value functions and to the decomposition built on them, not
to the planar block; the two statements are kept separate throughout.
\emph{(iv) The series is the block's best case, not its worst}: its
coefficients are unit, so it is exactly the lossless regime whose wall
directions \cref{thm:walls-d} closes. An A/B replay of the same traced curves
with genuine losses multiplies the carried face count by
\catLawLossyFaceRatioRange{} at matched horizons --- the split
charge/discharge pencils that are absent under common coefficients --- and
the warm kernel time by \catLawLossyTimeFactor$\times$ at the largest horizon
carried by both variants (campaign of \catLawLosslessCampaignDate); the
losing regime of \cref{tab:catalogue-law} is therefore a floor, not an
artefact of unfavourable efficiencies.

A posteriori, a slope audit of the raw traced curves locates both that floor
and its cause. Each catalogue splits into
\jitterEconLevelsLow{}--\jitterEconLevelsHigh{} recurring economic levels ---
the thermal merit order and the value of lost load, each present in hundreds
to thousands of hours --- plus a remainder that is sub-$10^{-4}$ tracer
dispersion on some curves and genuine congestion diversity on others: on the
most diverse curve, grouping slopes at $10^{-3}$ relative moves the objective
by \jitterDiverseGroupingCostPct{}, so the catalogue cannot in general be
compressed without changing the problem. The quadratic count, meanwhile, is a
property of the \emph{constructed} arrangement rather than of the object it
represents: on the same raw curves the realised subdivision of the lossless
block --- its maximal affine regions, measured by an independent active-set
prototype whose values match the frontal reference to \jitterGateExactMax{}
--- holds \jitterRealizedFLow{}--\jitterRealizedFHigh{} regions at the year
against \jitterBuiltFacesLow{}--\jitterBuiltFacesHigh{} constructed faces on
the identical instances, a factor of
\jitterBuiltOverRealizedLow{}--\jitterBuiltOverRealizedHigh{}. That prototype
still trails the frontal solve on these curves (measured outside the campaign
protocol), so the verdict of \cref{fig:catalogue-law} stands as measured;
what the audit changes is its attribution: the cost is the representation
this kernel carries --- exhaustive by construction, and the form on which its
exactness guard and the direction theorem are stated --- and the floor under
any symmetric network curve is the merit order's handful of levels, not the
float-tolerance count. The same split holds with losses: on the identical
lossy campaign instances the kernel constructs
\sandwichLossyBuiltLow{}--\sandwichLossyBuiltHigh{} faces where the realised
subdivision holds \sandwichLossyRealizedFLow{}--\sandwichLossyRealizedFHigh{}
regions, and a separate synthetic stress family with an all-distinct
catalogue of \sandwichLossyStressCat{} slopes realises
\sandwichLossyStressFMax{} --- so the loss multiplier of fact~(iv), like the
quadratic count, is a property of the constructed arrangement: measured on
the realised counts, against each curve's own loss-free twin at the same
horizon, losses cost only
\sandwichLossyTwinFactorLow{}--\sandwichLossyTwinFactorHigh{}. Constructed
faces and realised regions are never merged in this paper's counts.

\subsection{Oracle validation and failure modes}
\label{subsec:validation}

The lower-level operators are checked independently, and before any
composition. Forward--backward marginal capacity values
(\cref{thm:marginal-capacity}) agree with finite differences and LP duals at
approximately $10^{-13}$--$10^{-11}$. Traced hourly injection curves
(\cref{alg:trace}) are verified at random points against independent HiGHS
solves; on the coupled multi-commodity instance of \cref{app:coordination},
directional tracing is $37\times$ faster than the former 15-point sampled
builder while producing a smaller final gap.

The reproducibility package retains sampled-versus-traced,
serial-versus-parallel and alternation-versus-refinement ablations rather than
reporting only the favourable instances.

\section{Discussion and conclusion}
\label{sec:conclusion}

The operation of a single storage is a linear program, but it is not an
arbitrary linear program. Its variables form a long chain connected by one
cumulative state. By representing the cost of this state explicitly, the
Bellman recursion becomes a sequence of restrictions and infimal
convolutions; by Legendre--Fenchel conjugation, each convolution becomes an
ordered addition of convex functions. This structure explains why \pkg{}
solves long price-taking storage instances orders of magnitude faster, and
with far less memory, than a generic frontal solve.

The main contribution is therefore not a new application of dynamic
programming to storage. It is an implemented computational-convex-analysis
engine for longitudinal programs, connected explicitly to the serial
infimal-convolution algorithm of Tseng and Luo, nested resource allocation,
and monotropic network optimisation --- whose natural output is the complete
value function. That output is what the rest of the method consumes: it
supplies the marginal value of stored energy, exact capacity sensitivities
and sizing curves (\cref{thm:marginal-capacity}), and, through the
parametric dual simplex, the scalar cost curves that connect a temporal
device to any other optimisation block.

The organising question of the paper is then how far the carried state can be
widened before the construction stops paying, and the answer is sharper than
a benchmark. Two storages at one node need not be coordinated at all: the
pair can be eliminated, and the object that carries is combinatorially
rigid. Every breakline the recursion realizes is vertical, horizontal, or
borne by a line through the origin \emph{exactly}
(\cref{thm:directions}) --- no genericity, no viewport, any losses, any
horizon --- so the carried subdivision is permanently a sub-complex of one
\emph{fixed} arrangement, and the hourly step is a labelling of its faces
rather than a rebuild. For any number of co-located stores sharing one
charge--discharge coefficient the same argument runs in $\R^d$ with the
monotonicity hypothesis lifted, and the walls are position hyperplanes and
the type-$A$ indifference hyperplanes $\{s_i=s_j\}$ through the origin,
nothing else (\cref{thm:walls-d}). Both statements were arrived at
against an exact-rational arbiter rather than a floating-point census, and
that was not a formality: the census that first suggested a structural gap
in the two-store alphabet turned out to be an artefact of compiler
contraction (\cref{rem:alphabet-artefact}).

What the theorems do \emph{not} buy is speed, and the paper says so with a
measurement rather than a caveat. The directions are closed; the
\emph{positions} remain a per-instance measured alphabet, and
\cref{prop:two-closed} identifies the single operator responsible and shows
the asymmetry is structural rather than an unfinished proof. The face count
of the fixed arrangement grows as the square of the slope catalogue, and the
catalogue law of \cref{subsec:catalogue-law} reads the same curves in both
regimes: a one-dimensional chain wins the entire measured axis, ahead by
\catLawOneDGainRange{}$\times$ and growing, while the planar block trails by
\catLawTwoDLossRange{}$\times$ on the rich curves a real network produces
--- the same \catLawDoubleCat-slope curve read
\catLawDoubleOneDRatio$\times$ ahead of the frontal solve as a chain and
\catLawDoubleTwoDRatio$\times$ behind it as a block. The boundary is not a
crossover on the catalogue axis; it is the state dimension of the reader.
The planar block is therefore an exact object with a proof, valuable where
the catalogue is small and where the whole family of value functions is what
one needs --- and, at small $d$, an exact reference against which cut-based
methods on a few coupled reservoirs can be certified rather than merely
compared.

Uncertainty enters the same way, and it is the clearest case of the output
paying for itself. Because the solver returns the cost-to-go of \emph{every}
boundary state rather than the cost of one, taking an expectation over
finitely many futures is an ordered merge and not a new algorithm: the
recourse function is assembled instead of approximated, and the construction
survives its own repetition on a scenario tree, where one backward sweep
leaves the exact conditional cost-to-go at every node
(\cref{prop:tree-exact}). Against the cut-based tradition the difference is
not accuracy but the absence of a loop --- no sampling, no statistical
stopping rule, no tolerance in the definition of the expected cost-to-go ---
and the boundary is equally explicit: the merge is linear in the number of
scenarios only on a scalar state, and risk measures beyond the expectation
leave the single merge (\cref{rem:cvar}). The same curves are what a
foresight-free operation needs, and \cref{cor:rolling-exact} says a
window-by-window policy given the true cost-to-go has exactly zero truncation
error, so the residual gap of a rolling policy is an information error and
can be attributed as such --- measured at \waterValueExpectedGap{} over
eleven weather years, and \waterValueLooGap{} when the scarce year is
withheld from the expectation that prices it
(\cref{subsec:water-value-bench}).

Planning inherits both the method and its caveats. The forward--backward
capacity oracle of \cref{sec:investment} prices a small number of storages
inside a level-bundle master, and every capacity marginal comes out of one
decomposition solve, so a cut costs an operating solve whatever the number of
capacities. Scaling that to the many capacities of a full system expansion model
is deliberately not claimed here (\cref{subsec:sizing-scope}). The scenario
axis has now been measured on both sides of the catalogue law. Against traced
network-node curves the frontal barrier, block-diagonal and linear in the
number of weather years, is not caught at any measured size. On a
bounded-catalogue price-taking instance the same oracle route --- once its
master carries a feasibility frontier for the asset whose operating marginal
is zero, and a level-bundle stabilisation --- reaches certified
\kCurveGapTarget{} gaps in a cut count flat in $K$ and in the seed. It
claims no speed there: a crossover-free barrier stays
\kCurveBarrierPureAdvantageRange{} ahead. What it offers is shape --- a
predictable degradation where the default-configured frontal hits a
reproducible crossover wall at $K=64$, one setting away from repair
(\cref{subsec:sizing-scope}) --- and the catalogue decides which regime an
instance is in. On single-commodity
expansions a frontal barrier solve is simply the better tool.

Where a network is present, the method's exact reach is precisely one
intertemporal unit per coupled group: one parametric trace of the hourly
injection value followed by one solve is a globally exact solution of the
coupled dispatch (\cref{prop:one-storage}), and at annual scale it beats
both frontal references. Past that point the composition becomes
coordination, its iterates are feasible with measured gaps, and an opt-in
joint block removes the residual lock exactly on eligible groups while
buying correctness rather than speed --- its master carries every
intertemporal unit of the group over the whole horizon, so it approaches the
monolith it replaces just as the number of cutting-plane rounds grows. That
layer, and the annual multi-node, multi-energy and ramp-constrained studies
that measure it, are reported in full in \cref{app:coordination} rather than
claimed here. What those studies do establish, and it is the one durable
advantage of the composition at system scale, is peak memory: the
decomposition never forms the node-by-hour monolith, and at 30 nodes it runs
in \annualThirtyNodeMemFactorRange{}$\times$ less resident memory than the
frontal solvers while holding feasible gaps at or below $\annualGapMax$. It
does not win time there, and the paper does not claim it does. These results suggest a practical division of labour:
HiGHS and Gurobi remain the default for general or moderate one-off LP/QP
models; \pkg{} is preferable when the dominant structure is a very long
scalar cumulative chain, when many related instances must be solved, when
memory is limiting, when explicit marginal value functions are required, or
when a certified sizing loop must survive more weather years than a frontal
model can hold.

Several tasks remain. The direct MDA benchmark specified in
\cref{subsec:kernel-bench} awaits a public reference implementation; the
common driver must remove the residual timing asymmetry recorded in
\cref{app:campaigns}; and cached or incremental reduced-cost pricing and
full refactorisation have both been ruled out as meaningful end-to-end
levers by controlled ablations (\cref{app:simplex}). Promising algorithmic
directions include warm starts across capacity sweeps and exact elimination
on series--parallel compositions with small separators. On the widened
state, the road is barred by the position alphabet and the face count, never
by the directions: the $d$-store reference kernel of \cref{app:statd} now
runs byte-gated through $d=5$ with a bounded realised face law, and its
remaining wall is named --- the degenerate on-face sorting of the reduced
clamp --- so closing it is implementation work, not theory. Whether the
\emph{realised} directions stay finite once the coefficients differ, where
the direction \emph{algebra} is known to acquire an infinite multiplicative
orbit from $d=3$ on, is the open question the two theorems leave.

The broader conclusion is structural. Problem size along a longitudinal axis
need not be the difficult dimension. When the separator between past and
future remains one-dimensional and convex, exact value-function elimination
replaces a large global optimisation model by fast algebra on ordered
slopes; and when it widens to a pair, or to a coalition of co-located stores
sharing one coefficient, the elimination survives because the geometry it
generates is rigid rather than arbitrary. The fast Legendre transform makes
that algebra computationally practical. What decides whether it pays is not
the length of the chain but the size of the catalogue the step must read,
and the dimension of the state that reads it.

\printbibliography

\appendix

\section{Proofs}
\label{app:proofs}

This appendix collects the proofs of the results stated in the body, in their
order of appearance.

\paragraph{Proof of \cref{prop:closure}.}
Closed proper convex PWL functions on $\R$ are exactly the proper polyhedral
convex functions, and polyhedrality is preserved by all five
operations~\cite[\S19]{Rockafellar1970}. In detail: $f+g$ is polyhedral on
$\operatorname{dom}f\cap\operatorname{dom}g$ (refine the two piece lists and
add affine pieces), and a sum of closed functions is closed; restriction is
the sum with the polyhedral indicator $\delta_{[a,b]}$; reflection and shift
compose $f$ with an affine involution of the line, which changes neither
closedness, properness nor the piece structure. For the infimal convolution,
the conjugate of a proper polyhedral function is proper polyhedral, so
by~\eqref{eq:flt-infconv} $f\infconv g=(\conj f+\conj g)^*$ is polyhedral
and closed, and when it is proper the defining infimum is
attained~\cite[\S19]{Rockafellar1970}. Finally, a proper polyhedral function
bounded below attains its minimum on a closed interval, which is the object
returned by the minimisation operation. For the PWQ class the same steps
apply with one modification: conjugation acts by inverting the monotone
affine derivative $p=ax+b$ on quadratic pieces ($a>0$) and by exchanging
linear pieces with derivative jumps, which maps a finite ordered record list
to a finite ordered record list and hence the class to itself; sums refine
the two piece lists; and attainment in the convolution holds because on each
of the finitely many pieces the minimand is a closed convex quadratic.

\paragraph{Proof of \cref{cor:poscomb}.}
Positive scaling $f\mapsto pf$ multiplies every slope by $p>0$ and leaves
every breakpoint position unchanged, so it preserves the monotonicity of the
derivative and maps the class to itself; the pointwise sum is the first of
the five operations, covered by \cref{prop:closure}, and iterating it over
$w$ keeps the combination closed proper convex PWL, properness following
from the common domain point. For the count, apply \cref{prop:opcost}
$K-1$ times, or directly merge the $K$ position-sorted record lists in one
pass: each input record contributes at most one output record, plus $O(1)$
records for the domain endpoints, and canonical merging of coincident
positions only decreases the count. The bound is attained when the $K$
lists have pairwise distinct positions.

\paragraph{Proof of \cref{prop:opcost}.}
A canonical representation is sorted simultaneously in position and in
slope, because the derivative of a convex function is monotone. Conjugation
exchanges the two coordinates of each record and preserves their order
(\cref{subsec:representations}), so it is one linear pass, including the
canonical re-merging of records made coincident by the exchange. Addition
merges two position-sorted lists and adds slope information on overlapping
pieces: one ordered merge, at array-append cost $O(m+n)$, or with a
logarithmic factor when every insertion searches an ordered map. By
\eqref{eq:flt-infconv}, an infimal convolution is two conjugations and one
addition performed on the slope axis --- equivalently, the pieces of $f'$
and $g'$ are merge-sorted by slope. Each input record contributes at most
one output record, plus $O(1)$ records for the domain endpoints; canonical
merging of equal slopes or coincident positions only decreases the count.

\paragraph{Proof of \cref{lem:recovery}.}
For $s\in[\underline s_t,\overline s_t]$ the restriction adds zero, so
$V_t(s)=(V_{t-1}\infconv\tilde c_t)(s)$, and by \cref{prop:closure} the
infimum defining the convolution is attained. The minimand
$u\mapsto V_{t-1}(u)+\tilde c_t(s-u)$ is closed proper convex PWL
(\cref{prop:closure}, using reflection and shift on $\tilde c_t$), so its
minimiser set is a nonempty closed interval. For global optimality, select
$s_{t-1}^\star\in U_t(s_t^\star)$ and $x_t^\star=s_t^\star-s_{t-1}^\star$
backwards from $t=T$. Each $x_t^\star$ respects its bounds, since otherwise
$\tilde c_t(x_t^\star)=+\infty$; each $s_{t-1}^\star$ respects the envelope
at $t-1$, since $V_{t-1}$ already contains the restriction
$\mathcal R_{t-1}$ and $V_{t-1}(s_{t-1}^\star)<+\infty$. Telescoping the
identities $V_t(s_t^\star)=V_{t-1}(s_{t-1}^\star)+\tilde c_t(x_t^\star)$
from $t=T$ down to $t=1$ gives
$\sum_t\tilde c_t(x_t^\star)=V_T(s_T^\star)-V_0(s_0)=\min_sV_T(s)$, the
optimal value. Determinism costs nothing: $U_t(s)$ is an interval of the
represented function, not of its record list, and canonical merging
(\cref{subsec:representations}) makes endpoint selection independent of how
ties are stored.

\paragraph{Proof of \cref{prop:forward-backward}.}
Any feasible $x$ passes through some state $s$ at the cut: its prefix
$x_{1:t}$ is feasible with $s_t=s$ and costs at least $F_t(s)$, and its
suffix costs at least $B_t(s)$, so $J\ge\min_s\{F_t(s)+B_t(s)\}$. Conversely,
both partial minima are attained (\cref{prop:closure}), and the concatenation
of an optimal prefix into $s^\star$ with an optimal suffix out of $s^\star$
is feasible for~\eqref{eq:longitudinal} and achieves the right-hand side.
Finally $F_t+B_t$ is closed proper convex PWL by \cref{prop:closure}, so its
minimiser set is a nonempty closed interval.

\paragraph{Proof of \cref{cor:rolling-exact}.}
Apply \cref{prop:forward-backward} at each boundary $t$ of the partition: the
segment ending at $t$ solved with terminal valuation $B_t$ returns
$\min_s\{F_t(s)+B_t(s)\}=J$ and a state $s_t^\star$ at which pinning leaves
the optimal value unchanged; the next segment starts from that state with
its own true cost-to-go. Induction over the boundaries gives the claim.

\paragraph{Proof of \cref{prop:two-stage-exact}.}
The boundary state is the only variable linking the two stages: a plan is
feasible for~\eqref{eq:two-stage} if and only if its stage-1 part is feasible
and reaches some $\sigma\in[0,E]$ through~\eqref{eq:boundary-state}, and each
of its stage-2 parts is feasible under its scenario from $s_{T_1}=\sigma$.
Fix $\sigma$. By the definition of the forward value function
(\cref{subsec:forward}), the stage-1 cost of any such plan is at least
$F(\sigma)$, with equality for some feasible prefix, the infimum being
attained by \cref{prop:closure}. By the mirrored recursion, $B^w$ is the
value function of the stage-2 chain of scenario $w$ read from its starting
state, so $V_2^w(\sigma)=B^w(\sigma)$ for every $\sigma$ and every $w$ --- the
forward--backward identity of \cref{prop:forward-backward} applied at the cut
$T_1$ to each scenario chain --- again with attainment. Since the stage-1
decisions do not depend on $w$, the objective of~\eqref{eq:two-stage}
conditional on the boundary state is exactly
$F(\sigma)+\sum_wp_wB^w(\sigma)$. By \cref{cor:poscomb} the finite positively
weighted sum $\Qexp=\sum_wp_wB^w$ is closed proper convex PWL and its
representation is the merge of the $K$ record lists, so the represented
function is the pointwise expectation with no approximation --- proper on the
common domain, since a scenario with no feasible continuation from some
$\sigma$ enters as $+\infty$ there while any feasible plan supplies a
$\sigma$ at which all $K$ are finite; adding $F$ is
the pointwise sum of \cref{prop:closure}. Minimising the resulting closed
proper convex PWL function over $[0,E]$ is the minimisation operation, and
its minimiser set is a nonempty closed interval. Finally, concatenating an
optimal stage-1 prefix into $\sigma^\star$ with an optimal stage-2 suffix out
of $\sigma^\star$ in each scenario is feasible and attains the bound;
\cref{lem:recovery} produces both, deterministically and without tie-breaking
perturbation.

\paragraph{Proof of \cref{prop:tree-exact}.}
Induction on the height of $\nu$. At a leaf, $B_\nu$ is the cost-to-go of a
single longitudinal chain with terminal valuation $\Qexp_\nu$, exact by the
backward recursion of \cref{subsec:forward-backward} and in the class by
closure of the five operations (\cref{prop:closure}). For the induction step,
assume every child $c\in C(\nu)$ carries its exact $B_c$ in the class. The
children's entry state is the exit state of $\nu$, and the conditional
distribution of the child does not depend on the decisions taken inside
$H_\nu$, so the expected continuation cost seen from the exit state $u$ of
$\nu$ is the pointwise combination $\sum_c p_c B_c(u)$; by \cref{cor:poscomb}
that combination is closed proper convex PWL and its record list is the merge
of the children's, with no approximation (proper: any feasible subtree plan
supplies an exit state at which every child's continuation is finite). Appending it as terminal valuation
to the block $H_\nu$ and running one backward pass returns, by
\cref{prop:forward-backward} applied inside the block, the exact minimum over
feasible node-$\nu$ trajectories of block cost plus expected continuation
cost, i.e.\ $B_\nu$. Non-anticipativity is structural, exactly as in the
two-stage case: the decisions of $H_\nu$ are a single schedule, indexed by
$\nu$ and not by which child is realised, and they communicate with the
continuations only through the scalar exit state.
\emph{Recovery.} Fix a node $\nu$ and the entry state $\sigma_\nu$ reached
when it is visited ($\sigma_{\mathrm{root}}=s_0$). Pin the block at that
entry and append $\Qexp_\nu=\sum_{c\in C(\nu)}p_cB_c$ as terminal valuation:
writing $F^{\nu}$ for the cost-to-arrive functions of the pinned block, the
exit selection
\[
u_\nu^\star\in\argminop_{u}\bigl\{F^{\nu}_{|H_\nu|}(u)+\Qexp_\nu(u)\bigr\}
\]
minimises a closed proper convex PWL function over the confined domain, so
by the argument of \cref{prop:forward-backward} the minimum is attained, the
argmin is a nonempty closed interval, and any deterministic endpoint rule
selects from it --- the situation of \cref{lem:recovery}, whose backward
descent inside the block then returns increments of block cost exactly
$F^{\nu}_{|H_\nu|}(u_\nu^\star)$, with
$F^{\nu}_{|H_\nu|}(u_\nu^\star)+\Qexp_\nu(u_\nu^\star)=B_\nu(\sigma_\nu)$ by
the identity the induction step established. Every child $c$ is entered at
the same state $u_\nu^\star\in[0,E]$, and by induction on the height the
descent below $c$ recovers schedules of expected cost $B_c(u_\nu^\star)$
conditional on reaching $c$. Taking the expectation over the child realised
at the junction, the recovered policy costs
$F^{\nu}_{|H_\nu|}(u_\nu^\star)+\sum_cp_cB_c(u_\nu^\star)=B_\nu(\sigma_\nu)$
from $\nu$. Nothing in the block schedule or in $u_\nu^\star$ depends on
which child is realised, so the policy is non-anticipative, and descending
from the root it attains $B_{\mathrm{root}}(s_0)$, which the value part
identifies as the extensive-form optimum.

\subsection{Thinning the catalogue: the certified bracket, and what it finds}
\label{subsec:pair-sandwich}

The step of \cref{subsec:pair-cost} pays the slope catalogue, so the
question whether the catalogue can be thinned is a fair one. Uncertified
thinning is not an option --- the audit of \cref{subsec:catalogue-law}
exhibits a curve on which grouping slopes at $10^{-3}$ relative moves the
objective by \jitterDiverseGroupingCostPct{} --- but a \emph{certified}
scheme exists, and it is elementary in the present framework. Keep a
sub-catalogue (every tenth slope, say). Dropping the other pieces of each
hourly maximum yields $\psi_t^{-}\le\psi_t$; interpolating each $\psi_t$
chordally on the knots of the kept slopes, over the whole window the block
reads, yields $\psi_t^{+}\ge\psi_t$: two convex piecewise-linear curves on
the thin catalogue.

\begin{proposition}[Certified catalogue bracket]
\label{prop:bracket}
Let every hourly curve satisfy $\psi_t^{-}\le\psi_t\le\psi_t^{+}$ pointwise
on a window covering the range the recursion reads, all three convex
piecewise linear and finite there. Then \emph{(i)} the sweeps run on the
two envelopes bracket the true value functions pointwise ---
$W^{-}_t\le W_t\le W^{+}_t$ at every hour, at every state, in both
directions of time; \emph{(ii)} the bracket width obeys
$\sup\,(W^{+}_T-W^{-}_T)\le\sum_t\sup\,(\psi_t^{+}-\psi_t^{-})$;
\emph{(iii)} any refinement loop that, while the bracket is open somewhere,
reinserts at least one dropped slope of some hourly curve terminates in
finitely many rounds on the exact value function --- and every intermediate
round is a two-sided certificate.
\end{proposition}

\begin{proof}
(i) Each operator of the sweep is monotone for the pointwise order:
addition of a fixed function and restriction trivially, and infimal
convolution because $f\le f'$ and $g\le g'$ give
$\inf_u f(u)+g(x-u)\le\inf_u f'(u)+g'(x-u)$; induction along the hours.
(ii) The same operators are nonexpansive in the uniform norm --- an infimum
moves by at most the uniform distance of its argument, and a restriction
does not increase it --- so the hourly envelope gaps add at worst.
(iii) The catalogue is finite and every open round strictly grows the kept
set, so after finitely many rounds $\psi_t^{-}=\psi_t=\psi_t^{+}$ at every
hour, and the bracket is exact by (i).
\end{proof}

Where the two value functions coincide, the thin solution is
\emph{certified} exact; where they differ, the readout of the two sweeps
names the hours and segments responsible, and the median dropped slope of
each active window is reinserted. One caution on the rate, because the
classical one does not apply: for twice-differentiable convex data,
sandwich approximation converges as $O(1/n^{2})$ in the number of kept
knots, but the hourly curves here are piecewise linear --- their curvature
is concentrated at finitely many kinks, and dropping a kink opens an
envelope gap proportional to its slope jump times its breakpoint spacing,
first order in the dropped content. This is not a defect of the scheme but
the geometry of the data, and it is why the bracket, run on measured
curves, finds so little to compress: the audit of
\cref{subsec:catalogue-law} shows the economic content of a traced
catalogue to be a handful of sharp levels, and sharp levels are exactly
what an envelope cannot smooth over without opening a first-order gap.

Measured, the scheme is correct and finds nothing to compress --- on each
of the three instance families it was run on, named below with their own
counts, because the families live at different horizons and catalogue sizes
and no ratio across two of them is licit. On a park-catalogue control it
terminates exactly in \sandwichRoundsControl{} rounds. On the raw traced
network curves --- the audit family of \cref{subsec:catalogue-law} --- it
reproduces the exact answer with an active alphabet of
\sandwichActiveAlphabetNetwork{} slopes, the realised subdivision holding
\jitterRealizedFLow{} to \jitterRealizedFHigh{} regions; termination
matches the reference to \sandwichGateWorst{} throughout, with and without
losses. On all-distinct \emph{stress} catalogues --- synthetic, built so
that no slope repeats --- the active alphabet \emph{ramps}
(\sandwichStressActiveShare{} of the catalogue loss-free,
\sandwichLossyRampShareLow{} to \sandwichLossyRampShareHigh{} with losses),
and the realised count still saturates: a plateau of
\sandwichFvsKSatLow{}--\sandwichFvsKSatHigh{} while the loss-free catalogue
grows to $320$, and \sandwichLossyStressFMax{} regions on a lossy stress
catalogue of \sandwichLossyStressCat{} slopes. On the traced
\emph{campaign} instances of \cref{subsec:catalogue-law}, finally, the
lossy variants realise \sandwichLossyRealizedFLow{} to
\sandwichLossyRealizedFHigh{} regions against
\sandwichLossyBuiltLow{}--\sandwichLossyBuiltHigh{} \emph{constructed}
faces on the identical instances --- a factor of
\sandwichLossyBuiltOverRealizedLow{} to
\sandwichLossyBuiltOverRealizedHigh{} --- and the loss multiplier, measured
against each curve's \emph{own} loss-free twin at the same horizon, is
\sandwichLossyTwinFactorLow{}--\sandwichLossyTwinFactorHigh{}. In every
family the same reading: the split charge/discharge pencils multiply the
\emph{construction}, not the object. What the bracket leaves behind is therefore not a speed-up but
an instrument and a certificate: it is the scheme that makes slope grouping
admissible where data would otherwise be silently degraded, and its exact
sweeps are what measured the realised counts above.

\section{Directional stationarity of the two-store subdivision}
\label{app:stat}

This appendix proves \cref{thm:directions}, whose statement and proof
architecture are given in \cref{subsec:pair-theorem}: it supplies the formal
setting, the two lemmas of the primal reading, the three preservation lemmas,
and the induction. What it does not bound, deliberately, are the vertical and
horizontal \emph{positions}: the clamp's new rational is a new position, never
a new direction (\cref{prop:two-closed}(iii)), and the position alphabet
remains the measured, per-instance quantity of
\cref{subsec:two-store-conjugate}.

\paragraph{Setting and standing hypotheses.}
Two storages $i=1,2$ share one node. Storage $i$ has reservoir-side control
$x_i$ in a box $B_i$, level $\ell_i\in[0,E_i]$, and a load map $L_i$ ---
piecewise linear through the origin with slopes $c_i^-$ for $x<0$ and $c_i^+$
for $x\ge0$. The hourly gadget of \eqref{eq:bivariate-step} is
\begin{equation}
G_t(x)\;=\;\psi_t\bigl(L_1(x_1)+L_2(x_2)\bigr)\;+\;
\iota_{B_1\times B_2}(x),
\label{app:stat-eq-gadget}
\end{equation}
where $\psi_t$ is the node's hourly curve in the pair's combined-load
direction. Throughout the appendix:
\begin{enumerate}
\item[(i)] each $\psi_t:\R\to\R$ is convex, piecewise linear, finite and
  \emph{non-decreasing}. Monotonicity is not decorative: it is what makes the
  composition \eqref{app:stat-eq-gadget} convex (convex non-decreasing
  composed with the convex $L_1+L_2$), and it is not a free fact of the
  model --- the traced curves are non-decreasing whenever curtailment is
  free, the same property already invoked before \cref{prop:one-storage} ---
  so without free curtailment it is a genuine hypothesis of
  \cref{thm:directions}, not a modelling remark;
\item[(ii)] $0<c_i^-\le c_i^+$ for both storages (the load maps are convex,
  equivalently the round-trip efficiency is at most one); the degenerate case
  $c_i^-=0$ is handled by \cref{app:stat-degenerate};
\item[(iii)] the power boxes $B_{i}$ (possibly per hour) are non-empty and
  compact, the level boxes are $[0,E_i]$ with $E_i\ge0$, and every statement
  is conditional on feasibility --- when a restriction empties a domain the
  recursion reports infeasibility and carries nothing;
\item[(iv)] the objects below are the \emph{exact} conjugates of
  \eqref{eq:bivariate-step}; a max-representation of them may carry redundant
  pieces, and every statement about breaklines refers to boundaries between
  maximal affine regions --- equivalently, by \cref{app:stat-bridge}, to
  pieces witnessed on sets of positive area;
\item[(v)] the coefficients $c_i^\pm$ are fixed across the instance.
\end{enumerate}

\begin{definition}[Realized breakline]
\label{app:stat-realized}
A \emph{realized breakline} of a convex piecewise-linear function of the
price plane is a segment of positive length on which two of its maximal
affine regions meet. Two distinct maximal regions of a convex function
necessarily carry distinct gradients, so the supporting line of the shared
segment is well defined: for regions with data $(g_i,c_i)$ and $(g_j,c_j)$ it
is $\{\langle g_i-g_j,\,s\rangle = c_j-c_i\}$.
\end{definition}

\paragraph{The primal reading.}
Write $Z$ for the primal function whose conjugate the recursion carries ---
the cost-to-arrive of the pair in the forward sweep, the cost-to-go in the
backward one --- so that one hour of \eqref{eq:bivariate-step} is, on the
primal side,
\begin{equation}
Z'\;=\;\bigl(Z+\iota_{[0,E_1]\times[0,E_2]}\bigr)\;\infconv\;\check G_t,
\label{app:stat-eq-primal}
\end{equation}
with $\check G_t=G_t$ in one sweep and $\check G_t=G_t(-\cdot)$ in the
other, the reflection matching the direction in which time is traversed. The conjugate identities
behind \eqref{eq:bivariate-step} --- the level box becoming two directional
gradient clamps, the inf-convolution becoming an addition --- are exact here
because all operands are polyhedral: the closure operation of the general
biconjugation theorem is not needed, the relevant infima being attained
(Rockafellar, Thm.~16.4 and Cor.~19.3.4), and the non-emptiness condition of
the latter is precisely the non-infeasible branch of
\cref{app:stat-restrict}. Every audited object of the recursion --- after
each one-sided clamp and after each addition --- is therefore the conjugate of
one of the primal objects $Z$, $Z+\iota_{\{\ell_1\le E_1\}}$, \dots,
$Z+\iota_{\mathrm{box}}$, $Z'$.

All primal objects have compact non-empty domains: $\{\ell_0\}$ (or the level
box, for the free-terminal backward start) is compact; a restriction to a
half-plane preserves compactness, non-emptiness being hypothesis~(iii); and
an inf-convolution with $\check G_t$ gives $\operatorname{dom} Z' =
\operatorname{dom} Z \pm B_1\times B_2$, a sum of compacta. Their conjugates
are finite on all of $\R^2$.

For a proper closed polyhedral convex $f$ with compact domain, the faces of
$\operatorname{epi} f$ split into \emph{graph} faces (recession cone not
containing the vertical direction) and vertical ones; write $D(f)$ for the
projected lattice of graph faces --- its $2$-cells are the maximal affinity
regions of $f$, its $1$-faces the \emph{edges} (domain-boundary edges
included), its $0$-faces the \emph{vertices}. Let
\[
R^{*}\;=\;\bigl\{\,(c_2^{\sigma_2},\,-c_1^{\sigma_1})\;:\;
\sigma_1,\sigma_2\in\{-,+\}\,\bigr\}
\ \text{purged of its axis-parallel members}
\]
--- at most four unoriented directions, none axis-parallel under~(ii). Say
that $f$ satisfies the invariant $\mathbf J$ when \emph{(J1)} every edge of
$D(f)$ is parallel to $e_1$, to $e_2$, or to a member of $R^{*}$, and
\emph{(J2)} $f$ is \emph{constant} along every edge of direction in $R^{*}$.

\begin{lemma}[Bridge: active pieces are vertices]
\label{app:stat-bridge}
Let $f$ be proper closed polyhedral convex with compact domain and let
$\max_k(\langle g_k,\cdot\rangle+c_k)$ be any exact representation of
$\conj{f}$. If a piece $(g,c)$ attains the max on a set of positive area,
then $g$ is a vertex of $D(f)$ and $c=-f(g)$.
\end{lemma}

\begin{proof}
First, $\conj{f}(s)=\max_v[\langle v,s\rangle-f(v)]$ over the vertices $v$ of
$D(f)$: the sup of $\langle\cdot,s\rangle-f$ over the compact domain is
attained, its attainment set lifts to the face of $\operatorname{epi}f$
exposed by $(s,-1)$, that face lies in the graph (at $t>f(x)$ the objective
is strictly dominated by $(x,f(x))$), and $\operatorname{epi}f$ is pointed
(its recession cone $\{0\}\times\{0\}\times\R_{\ge0}$ contains no line), so
the face contains a vertex of $\operatorname{epi}f$ --- necessarily a graph
point --- whose projection is a vertex of $D(f)$. Now cover the
positive-area set by the finitely many closed sets on which one vertex piece
attains: one of them meets it in positive area, and two affine functions
which agree on a set of positive area agree identically, giving $g=v$ and
$c=-f(v)$.
\end{proof}

\begin{lemma}[Reading realized breaklines in the primal]
\label{app:stat-read}
Let $f$ be as in \cref{app:stat-bridge}, and let two maximal affine regions
of $\conj{f}$, with affine data $(u,c_u)$ and $(v,c_v)$, meet along a segment
of positive length. Then $u\ne v$ are vertices of $D(f)$ with $c_u=-f(u)$ and
$c_v=-f(v)$, the segment $[u,v]$ is an \emph{edge} of $D(f)$, and the
supporting line of the realized breakline is
$\{\langle u-v,\,s\rangle=f(u)-f(v)\}$.
\end{lemma}

\begin{proof}
The gradients of two distinct maximal regions differ
(\cref{app:stat-realized}), and each region has positive area, so
\cref{app:stat-bridge} identifies both data with vertex pieces. On
the relative interior of the shared segment, discard the finitely many $s$
that are the exact gradient of a $2$-cell of $D(f)$ or a vertex of the dual
subdivision; at a remaining $s$ the exposed face
$F(s)=\operatorname{arg\,max}(\langle\cdot,s\rangle-f)$ has dimension exactly
one (dimension two would force $s$ to be a cell gradient, dimension zero
would contradict $u,v\in F(s)$), so it is a $1$-face of $D(f)$ containing $u$
and $v$. A vertex contained in a face is a vertex \emph{of} that face (the
relative interiors of the faces of a polyhedron partition it: no vertex is
interior to a $1$-face), and a compact $1$-face has exactly two vertices;
hence $F(s)=[u,v]$, an edge of $D(f)$. Equality of the two pieces on the
segment gives the supporting line.
\end{proof}

Under $\mathbf J$ the conclusion sharpens: an edge $[u,v]$ of direction
$d\in R^{*}$ carries $f(u)=f(v)$ by (J2), so its breakline is
$\langle u-v,s\rangle=0$ --- \emph{through the origin}, of slope
$c_2^{\sigma_2}/c_1^{\sigma_1}$; an edge parallel to $e_1$ (resp.\ $e_2$)
gives a vertical (resp.\ horizontal) breakline. No viewport enters: a
breakline realized in any window is a positive-length coincidence in the
plane.

\begin{lemma}[The gadget satisfies $\mathbf J$]
\label{app:stat-gadget}
Under (i)--(iii), $G_t$ is convex, and $G_t$ and $\check G_t$ satisfy
$\mathbf J$.
\end{lemma}

\begin{proof}
Convexity is (i) composed with the convex $L_1+L_2$, plus a box indicator.
The affinity regions of $G_t$ on the compact box are cut by: the box
boundary and the splitting lines $x_i=0$ (directions $e_1,e_2$); and, inside
each quadrant $Q(\sigma_1,\sigma_2)$, the level lines
$c_1^{\sigma_1}x_1+c_2^{\sigma_2}x_2=v_k$ of the kinks of $\psi_t$, since
there $G_t=\psi_t\circ(\text{linear form})$. A level-line edge has direction
$(c_2^{\sigma_2},-c_1^{\sigma_1})$ and $G_t$ is \emph{constant} along it
(equal to $\psi_t(v_k)$): that is (J2) when the direction lies in $R^{*}$,
and when it degenerates onto an axis (the purged case $c_i=0$) it is absorbed
by (J1) with no constancy owed. No other edge exists, which is (J1). The
reflection $w\mapsto-w$ reverses edges without changing unoriented directions
or constancy, and one-sided or degenerate boxes only remove quadrants or drop
dimension.
\end{proof}

\begin{lemma}[Level restriction preserves $\mathbf J$]
\label{app:stat-restrict}
Let $f$ satisfy $\mathbf J$ and let $H=\{\ell_a\le\beta\}$ (or $\ge$) be a
half-plane with axis-parallel boundary. If
$\operatorname{dom}f\cap H=\emptyset$ the step is infeasible and nothing is
carried; otherwise $f+\iota_H$ satisfies $\mathbf J$.
\end{lemma}

\begin{proof}
The affinity complex of $f+\iota_H$ is the refinement $\{F\cap H\}$. Its
edges are of two kinds: $e\cap H$ for $e$ an edge of $D(f)$ --- same
direction, constancy inherited on a subsegment --- and $C\cap\partial H$ for
$C$ a $2$-cell --- carried by the cut line, hence axis-parallel. Degenerate
domains (a segment when $E_i=0$, a point) leave only edges of those two
kinds.
\end{proof}

Two remarks locate \cref{prop:two-closed}(iii) inside this lemma. The clamp
band's gradient is the point of the segment $[g_K,g_D]$ whose $a$-coordinate
equals the bound: the ``new rational'' of \cref{prop:two-closed}(iii) is the
$b$-coordinate of a \emph{new primal vertex} --- a new position, never a new
direction. And
$(C_1\cap H)\cap(C_2\cap H)=(C_1\cap C_2)\cap H$: a restriction creates no
new adjacency between kept cells, so every new edge passes through the cut
line.

\begin{lemma}[Inf-convolution preserves $\mathbf J$]
\label{app:stat-infconv}
Let $f,g$ be proper closed polyhedral convex with compact domains, both
satisfying $\mathbf J$ with respect to the same direction set $R^{*}$. Then
$h=f\infconv g$ is exact, polyhedral, with compact domain, and satisfies
$\mathbf J$.
\end{lemma}

\begin{proof}
\emph{Exactness and epigraphs.} For fixed $x$ the map
$y\mapsto f(y)+g(x-y)$ is continuous on the compact set
$\operatorname{dom}f\cap(x-\operatorname{dom}g)$, so the infimum is attained
whenever that set is non-empty; hence
$\operatorname{epi}h=\operatorname{epi}f+\operatorname{epi}g$ (one inclusion
is the definition, the other decomposes $(x,t)$ through an attaining $y$).

\emph{Face additivity.} For polyhedra $P,Q$ and a direction $w$ along which
both suprema are attained: if $x=p+q$ maximises $\langle w,\cdot\rangle$ over
$P+Q$ then $\langle w,p\rangle\le\sup_P$ and $\langle w,q\rangle\le\sup_Q$
force equality in each, so $F(P+Q,w)=F(P,w)+F(Q,w)$; conversely sums of
maximisers maximise. For $w=(s,-1)$ on our epigraphs the suprema are
attained, the domains being compact.

\emph{Every edge is exposed.} Let $E$ be an edge of $D(h)$ and
$\operatorname{lift}(E)$ its graph $1$-face. Every support vector $w$ of an
epigraph has $\langle w,e_z\rangle\le0$; a face is exposed by any $w$ in the
relative interior of its normal cone, and if that $w$ had $w_z=0$ the exposed
face would contain the vertical direction and not be a graph face. So some
exposing $w$ has $w_z<0$; normalise it to $(s_0,-1)$.

\emph{(J1).} By the two previous paragraphs,
$\operatorname{lift}(E)=\operatorname{lift}(E_f)+\operatorname{lift}(E_g)$
with $E_f,E_g$ the faces of $D(f),D(g)$ exposed at $s_0$. Since $\dim E=1$,
each has dimension at most one and at least one is an edge; if both are,
they are parallel. Hence every edge direction of $h$ is an edge direction of
$f$ or of $g$.

\emph{(J2).} Exposure at $s_0$ makes $\langle s_0,\cdot\rangle-x$ constant on
the exposed face of each object $x\in\{f,g,h\}$, so the directional slope of
$x$ along a direction $d$ of its face is $\langle s_0,d\rangle$. If $E$ has
direction $d\in R^{*}$, the operand that contributed an \emph{edge} of
direction $d$ is constant along it by its own (J2), so
$\langle s_0,d\rangle=0$, and $h$ is constant along $E$.
\end{proof}

This is the step that closes the worry of \cref{prop:two-closed}(iii): a
clamp band --- a new primal vertex with new coordinates --- may well become
co-active later against pieces born after it, but the breakline they realize
is an edge of $D(Z)$, whose directions and constancies are governed by
$\mathbf J$, not by the arithmetic of the gradients.

\paragraph{Proof of \cref{thm:directions}.}
Hypotheses (i)--(v) above are the standing hypotheses of the theorem as
stated in \cref{subsec:pair-theorem}.
The initial primal objects satisfy $\mathbf J$: a point ($\{\ell_0\}$), or
the level box (axis-parallel edges only, (J2) vacuous). All domains stay
compact and non-empty along the recursion (hypothesis (iii) and the
compactness argument above). By induction, $\mathbf J$ passes through each of
the four level half-planes (\cref{app:stat-restrict}) and through the
inf-convolution with the hour's gadget
(\cref{app:stat-gadget,app:stat-infconv}, with the common $R^{*}$ of
hypothesis (v); the backward reflection is absorbed by
\cref{app:stat-gadget}). \Cref{app:stat-read} --- through
\cref{app:stat-bridge}, hypothesis (iv) --- translates $\mathbf J$ into
(a)--(c) for the realized breaklines of the conjugates.\hfill$\square$

\begin{corollary}[Degenerate coefficients]
\label{app:stat-degenerate}
If some $c_i^{\sigma}=0$ (a zero-yield branch --- in practice a zero
discharge yield), the corresponding members of
$\{(c_2^{\sigma_2},-c_1^{\sigma_1})\}$ are axis-parallel and are purged from
$R^{*}$; \cref{thm:directions} holds verbatim with the remaining set, the
degenerate family being absorbed by classes (a)--(b).
\end{corollary}

\begin{remark}[Inflows]
\label{app:stat-inflows}
With inflows, $\ell_t=\ell_{t-1}+a_t+x_t$, the state translation moves no
breakline at all: on the primal side it changes neither edge directions nor
value differences along edges, and on the conjugate side it maps pieces
$(g,c)$ to $(g+a_t,c)$, which leaves every supporting line
$\{\langle g_i-g_j,s\rangle=c_j-c_i\}$ unchanged. \Cref{thm:directions}
holds unchanged.
\end{remark}

What the theorem leaves open is exactly what
\cref{subsec:two-store-conjugate} measures: the vertical and horizontal
\emph{positions} are difference quotients of $Z$ along axis-parallel primal
edges, the clamp does create new primal vertices, and the position alphabet
therefore remains a per-instance measured quantity ---
\cref{prop:two-closed}(iii) survives as a statement about positions, never
about directions.

The corroborating audit quoted in \cref{rem:alphabet-artefact} is an
exact-arithmetic search over the recursion: rational predicates throughout,
two campaigns of $1\,180$ and $1\,198$ instances run in both sweeps, with
adversarial families including clamp bands kept alive across flat hours and
met by later menus, five audit points per hour, and a per-operator
differential audit. It classified \statEdgeAuditTotal{} million realized
edges and found none outside families (a)--(c), with every through-origin ray
carrying an exactly zero offset; the instruments and their verdicts are
\texttt{stationarity\_exact/} in the \texttt{prototypes} repository
(\cref{sec:package}).

\section{The $d$-store block: directional stationarity under common
coefficients}
\label{app:statd}

This appendix proves \cref{thm:walls-d} and \cref{cor:arrangement}, stated
in \cref{subsec:pair-ddim}: it supplies the formal setting, the lattice
lemma that replaces the four-element direction set of \cref{app:stat}, the
reading and preservation lemmas in dimension $d$, and the induction. What is
\emph{not} bounded here, deliberately and exactly as in two dimensions, are
the wall \emph{positions}: they remain the per-instance measured alphabet.

\paragraph{Setting and standing hypotheses.}
$d$ storages share one node; storage $i$ has reservoir-side control $x_i$ in
a box $B_{i}$ (possibly per hour) and level $\ell_i\in[0,E_i]$. With the
load-map conventions of \cref{app:stat} and one \emph{common} coefficient,
the hourly gadget is
\begin{equation}
G_t(x)\;=\;\psi_t\Bigl(c\,\textstyle\sum_i x_i\Bigr)
\;+\;\iota_{B_{1}\times\cdots\times B_{d}}(x)
\;=\;\tilde\psi_t\Bigl(\textstyle\sum_i x_i\Bigr)\;+\;\iota_{\mathrm{box}}(x),
\qquad \tilde\psi_t(z):=\psi_t(c\,z),
\label{app:statd-eq-gadget}
\end{equation}
and the hourly step is the $d$-dimensional form of \eqref{eq:bivariate-step},
\begin{equation}
W^\star_{t+1}\;=\;\bigl(W^\star_t\ \text{clamped on each axis into }
[0,E_i]\bigr)\;+\;G^\star_t,
\label{app:statd-eq-step}
\end{equation}
i.e.\ $2d$ one-sided directional gradient clamps followed by one addition,
each the exact conjugate of the primal step
$Z'=(Z+\iota_{\prod_i[0,E_i]})\infconv\check G_t$ (the box support function
is separable axis by axis; Rockafellar, Thm.~16.4 and Cor.~19.3.4, the
non-emptiness condition of the latter being the infeasibility branch of
\cref{app:statd-restrict}). Throughout:
\begin{enumerate}
\item[(i)] each $\tilde\psi_t:\R\to\R$ is convex, piecewise linear and
  finite. \emph{No monotonicity is assumed}: hypothesis~(i) of
  \cref{app:stat} was needed there because a genuinely kinked load map makes
  the composition convex only for non-increasing curves; under~(ii) below
  the load maps are linear and the composition is convex unconditionally;
\item[(ii)] all coefficients equal one common constant: $c_i^{\pm}=c\neq0$
  for every storage and both signs --- the lossless devices of
  \cref{subsec:coupled-model} realise this with $c=1$. With distinct
  coefficients no closure theorem of this kind exists: the direction
  \emph{algebra} acquires a proven infinite multiplicative orbit from
  $d=3$ on (\texttt{stationarity\_exact/ROOT\_VERDICT.md}, \cref{sec:package})
  --- whether
  the \emph{realised} directions remain finite is open, the same
  realised-versus-closure distinction as in two dimensions;
\item[(iii)] the power boxes are non-empty and compact, the level boxes are
  $[0,E_i]$, and every statement is conditional on feasibility. All primal
  domains along the recursion are compact and non-empty: the initial objects
  are a point or the level box, a half-space restriction preserves
  compactness (non-emptiness is the non-infeasible branch), and an
  inf-convolution adds a compact box;
\item[(iv)] the objects are the exact conjugates of \eqref{app:statd-eq-step};
  a \emph{realized wall} is the $d$-dimensional form of
  \cref{app:stat-realized}: a set of positive $(d{-}1)$-dimensional measure
  common to the boundaries of two maximal affine regions --- whose gradients
  are automatically distinct, by the same convexity argument --- so the
  supporting hyperplane $\{\langle g_i-g_j,s\rangle=c_j-c_i\}$ is well
  defined;
\item[(v)] the coefficients are fixed across the instance.
\end{enumerate}

\paragraph{The partial-partition lattice, and the block invariant.}
For $f$ proper closed polyhedral convex with compact domain, $D(f)$ is the
projected lattice of graph faces of $\operatorname{epi}f$
(\cref{app:stat}); on a face $F$, $f$ is affine and its slope
$\partial_\delta f|_F$ is defined for every $\delta$ in the tangent space
$\mathrm{or}(F)$ of $F$. Let
\[
\mathbb L_d \;=\; \bigl\{\,S\subseteq\R^d \text{ linear} \;:\;
S^{\perp}\ \text{has a basis of }0/1\text{ vectors with pairwise disjoint
supports}\,\bigr\}.
\]
Writing the basis supports as \emph{zeros} (singletons) and \emph{blocks}
(size $\ge2$), $S=\{x: x_z=0\ (z\in Z),\ \sum_{i\in B_m}x_i=0\ (m)\}$; the
uncovered coordinates are \emph{free}. Say that $f$ satisfies the invariant
$\mathbf J_d$ when \emph{($\mathrm J1$)} the orientation of every face of
$D(f)$ belongs to $\mathbb L_d$, and \emph{($\mathrm J2$)} on every face
$F$, for every pair $\{j,k\}$ contained in one \emph{block} of
$\mathrm{or}(F)$: $\partial_{e_j-e_k}f|_F=0$ --- the gradient of $f$ is
constant within each block. (Blocks are coalitions of storages whose levels
the recursion has tied; ($\mathrm J2$) says water values are equal inside a
coalition. Pairs of \emph{free} coordinates are exempt --- the distinction
that makes the invariant both true and sufficient, see
\cref{app:statd-restrict}.)

\begin{lemma}[The lattice]
\label{app:statd-lattice}
(a) The disjoint $0/1$ basis of $S^{\perp}$, when it exists, is unique, so
zeros, blocks and free coordinates are canonical. (b) $\mathbb L_d$ is
stable under $S\mapsto S\cap e_a^{\perp}$; a block containing $a$
\emph{shrinks} to $B\setminus\{a\}$, all other structure unchanged. (c)
$\mathbb L_d$ is stable under sums, and every pair $\{j,k\}$ in a block of
$S_1+S_2$ is joined by a chain of pairs each contained in a block of $S_1$
or of $S_2$. (d) $\mathbb L_d$ is finite. (e) Its one-dimensional members
are exactly the axes $\langle e_i\rangle$ and the roots
$\langle e_i-e_j\rangle$; in the canonical structure of
$\langle e_i-e_j\rangle$ the pair $\{i,j\}$ is a \emph{block}.
\end{lemma}

\begin{proof}
(a) If $G,G'$ are two such bases, every $1_B$ with $B\in G$ lies in
$\operatorname{span}G'$, hence is coordinatewise a disjoint union of blocks
of $G'$; the two partitions refine each other and coincide. (b)
$(S\cap e_a^{\perp})^{\perp}=S^{\perp}+\langle e_a\rangle$: if $a$ is free,
add the singleton; if $a\in B$, rebase $\{1_B,e_a\}$ as
$\{e_a,1_{B\setminus\{a\}}\}$. (c) $(S_1+S_2)^{\perp}=S_1^{\perp}\cap
S_2^{\perp}$; $v\in S^{\perp}$ iff $v$ is constant on each block, zero on
the free coordinates and arbitrary on the zeros, so the intersection
consists of the vectors constant on each class of the transitive closure of
the two block relations, zero on every class touching a free coordinate of
either operand, arbitrary on common zeros: a basis is given by the class
indicators and the common-zero singletons --- $0/1$, disjoint --- and every
pair in a surviving class is chained through operand blocks. (d) Members
correspond to families of disjoint blocks. (e) $\dim S=1$ forces $d-1$
disjoint constraints: either $d-1$ singletons and one free coordinate
($S=\langle e_i\rangle$), or $d-2$ singletons and one block of size two
($S=\langle e_i-e_j\rangle$).
\end{proof}

\begin{lemma}[Bridge and reading in dimension $d$]
\label{app:statd-read}
Let $f$ be proper closed polyhedral convex with compact domain. (a) Every
piece of an exact max-representation of $\conj{f}$ that attains the max on a
set of positive $d$-dimensional measure coincides with $(w,-f(w))$ for a
vertex $w$ of $D(f)$. (b) If two maximal affine regions of $\conj{f}$ share
a realized wall, their data are vertex pieces $(u,-f(u))$, $(v,-f(v))$,
$u\neq v$, the segment $[u,v]$ is an \emph{edge} of $D(f)$, and the wall's
hyperplane is $\{\langle u-v,s\rangle=f(u)-f(v)\}$.
\end{lemma}

\begin{proof}
(a) is \cref{app:stat-bridge} verbatim with area replaced by
$d$-dimensional Lebesgue measure. (b) transposes \cref{app:stat-read}; the
one argument that changes is the genericity exclusion, which becomes a
stratification count: the set of $s$ exposing a given $k$-face of $D(f)$ has
dimension $d-k$ (relative interior of its normal cone, sliced by
$w_z=-1$), so the $s$ exposing faces of dimension $\ge2$ form a set of
dimension $\le d-2$ --- of $(d{-}1)$-measure zero in the wall. At a generic
$s$ of the wall the exposed face is therefore of dimension exactly one,
contains $u$ and $v$, and the vacancy argument of \cref{app:stat-read}
(relative interiors of faces partition the polyhedron; a compact $1$-face
has exactly two vertices) applies verbatim: $F(s)=[u,v]$.
\end{proof}

Under $\mathbf J_d$ the conclusion sharpens exactly as in two dimensions:
an edge is a face, its orientation is one-dimensional, hence an axis or a
root by \cref{app:statd-lattice}(e); in the root case $\{i,j\}$ is a block
of the edge's own orientation, so ($\mathrm J2$) forces $f(u)=f(v)$ and the
wall is $\{s_i=s_j\}$ --- \emph{through the origin, exactly}. No viewport
enters.

\begin{lemma}[The gadget satisfies $\mathbf J_d$]
\label{app:statd-gadget}
Under (i)--(iii), $G_t$ is convex and $G_t$, $\check G_t$ satisfy
$\mathbf J_d$.
\end{lemma}

\begin{proof}
Convexity is (i) composed with a linear form, plus a box indicator ---
no monotonicity. Because the coefficients agree on both signs, the load maps
are linear and \emph{no splitting at $x_i=0$ exists}: the complex of $G_t$
on the compact box is cut only by the box facets and by the level
hyperplanes $\{\sum_i x_i=v_k\}$ of the kinks of $\tilde\psi_t$. Its faces
are (box face) $\cap$ (slab or level hyperplane), of orientations $V_I$
(free coordinates $I$; no block, ($\mathrm J2$) vacuous) or
$V_I\cap\mathbf 1^{\perp}$ (one block $I$); both lie in $\mathbb L_d$. For
a block pair $\{j,k\}\subseteq I$: $G_t$ depends on $x$ only through
$\sum_i x_i$, so $\partial_{e_j-e_k}G_t=0$ --- indeed its slope along
\emph{every} direction of $\mathbf 1^{\perp}$ vanishes. Reflection and
one-sided or degenerate boxes change nothing.
\end{proof}

\begin{lemma}[Level restriction preserves $\mathbf J_d$]
\label{app:statd-restrict}
Let $f$ satisfy $\mathbf J_d$ and $H=\{\ell_a\le\beta\}$ (or $\ge$). If
$\operatorname{dom}f\cap H=\emptyset$ the step is infeasible and nothing is
carried; otherwise $f+\iota_H$ satisfies $\mathbf J_d$.
\end{lemma}

\begin{proof}
The refined complex consists of the slices $F\cap H$ (orientation and
restricted slopes inherited) and the cut faces $F\cap\partial H$ for the
faces $F$ that the hyperplane traverses. A cut face has orientation
$\mathrm{or}(F)\cap e_a^{\perp}\in\mathbb L_d$
(\cref{app:statd-lattice}(b)), whose blocks are those of $\mathrm{or}(F)$,
unchanged or with $B\to B\setminus\{a\}$; every pair of a new block was a
pair of an old one, and the slope of $f$ along it, restricted from $F$, is
zero by ($\mathrm J2$) of $f$.
\end{proof}

Two remarks locate the $d$-dimensional difficulty. In two dimensions a cut
of a $2$-cell is a segment \emph{of the cut line}: new edges are
axis-parallel by a dimension count, and no constancy is owed
(\cref{app:stat-restrict}). From $d=3$ on this fails: cutting a $2$-face of
orientation $\operatorname{span}(e_i,e_j-e_k)$ at $a=i$, or
$\operatorname{span}(e_i-e_j,e_k-e_l)$ at $a\in\{i,j\}$, produces an edge of
\emph{root} direction, whose constancy is exactly the block-shrink
inheritance above. And a $2$-face of orientation
$\operatorname{span}(e_i,e_j)$ --- a \emph{free} pair, for which
($\mathrm J2$) promises nothing --- never produces a root edge: for
$a\notin\{i,j\}$ the plane lies inside $e_a^{\perp}$ (the cut keeps the face
whole or drops it), and for $a\in\{i,j\}$ the cut yields an axis edge. This
is why the invariant distinguishes blocks from free pairs.

\begin{lemma}[Inf-convolution preserves $\mathbf J_d$]
\label{app:statd-infconv}
Let $f,g$ satisfy $\mathbf J_d$, with compact domains. Then
$h=f\infconv g$ is exact, polyhedral, with compact domain, and satisfies
$\mathbf J_d$.
\end{lemma}

\begin{proof}
Exactness, $\operatorname{epi}h=\operatorname{epi}f+\operatorname{epi}g$,
face additivity $F_h=F_f+F_g$ at a common exposing $(s_0,-1)$, and the
exposure of every graph face by such a vector are the proofs of
\cref{app:stat-infconv}(a)--(c), written there for polyhedra in
$\R^{d+1}$ and unchanged. ($\mathrm J1$):
$\mathrm{or}(F_h)=\mathrm{or}(F_f)+\mathrm{or}(F_g)\in\mathbb L_d$ by
\cref{app:statd-lattice}(c). ($\mathrm J2$), \emph{the telescoping chain}:
on a face exposed by $s_0$ the slope of the object along any direction
$\delta$ of its orientation is $\langle s_0,\delta\rangle$. Let $\{j,k\}$
lie in a block of $\mathrm{or}(F_h)$. By \cref{app:statd-lattice}(c) there
is a chain $j=m_0,\dots,m_r=k$ whose every link is a pair in a block of
$\mathrm{or}(F_f)$ or of $\mathrm{or}(F_g)$; for each link, ($\mathrm J2$)
of the contributing operand gives
$0=\partial_{e_{m_t}-e_{m_{t+1}}}(\cdot)|_{F}=
\langle s_0,e_{m_t}-e_{m_{t+1}}\rangle$, i.e.\
$s_{0,m_t}=s_{0,m_{t+1}}$; telescoping, $s_{0,j}=s_{0,k}$ and
$\partial_{e_j-e_k}h|_{F_h}=0$. (The two-dimensional argument of
\cref{app:stat-infconv} is the case of chains of length one at the level of
edges; the chain is needed here because \cref{app:statd-restrict} consumes
$2$-faces.)
\end{proof}

\paragraph{Proof of \cref{thm:walls-d}.}
Hypotheses (i)--(v) above are the standing hypotheses of the theorem as
stated in \cref{subsec:pair-ddim}.
The initial objects satisfy $\mathbf J_d$ (a point, or the level box:
coordinate orientations, no blocks); all domains stay compact and
non-empty (hypothesis (iii)). By induction $\mathbf J_d$ passes through
each level half-space (\cref{app:statd-restrict}) and through the
inf-convolution with the hour's gadget
(\cref{app:statd-gadget,app:statd-infconv}; the backward reflection is
absorbed by \cref{app:statd-gadget}). \Cref{app:statd-read} --- under
hypothesis (iv) --- translates $\mathbf J_d$ into (a)--(b), the direction
classification being \cref{app:statd-lattice}(e) and the zero offset of the
root walls being ($\mathrm J2$) applied to the edge itself.\hfill$\square$

\paragraph{Proof of \cref{cor:arrangement}, and what it does not say.}
The subdivision is a sub-complex of the stated arrangement because, by
\cref{thm:walls-d}, every wall it realizes is one of that arrangement's
hyperplanes; the face bound is the standard count for an arrangement of $N$
hyperplanes in $\R^d$. The corollary bounds the carried \emph{size}, not the
work of one step: an inf-convolution of two subdivisions with $F$ faces may
generate up to $F^2$ candidate pairs before pruning. Both operands being
sub-complexes of one arrangement of fixed directions, their non-empty
intersections number $O(N^d)$ themselves, which an arrangement-guided
implementation can traverse without enumerating pairs --- but that is a
statement about implementations (the reference implementation clips
pairwise), not a theorem.

\begin{remark}[Inflows, and the seam with the two-store theorem]
\label{app:statd-seam}
Inflows translate the state and move no wall, exactly as in
\cref{app:stat-inflows}. At $d=2$ the present theorem recovers the lossless
case of \cref{thm:directions} --- the single root $e_1-e_2$ is the diagonal
ray $\{s_1=s_2\}$ --- while \emph{lifting} the monotonicity of the hourly
curves, which \cref{app:stat} needs because it covers genuinely kinked load
maps; conversely, losses and kinks at $d=2$ are covered only by
\cref{thm:directions}. Neither statement subsumes the other, and the
lossless two-store block with non-monotone curves was covered by neither
before this appendix. The lifted hypothesis is not decorative: with genuine
losses, a decreasing segment of $\psi_t$ --- negative marginal prices
without free curtailment --- makes simultaneous charging and discharging
profitable and the hourly gadget itself non-convex, a boundary of the convex
method and not of this proof; the common-coefficient regime is exactly the
one in which negative prices cost the convexity nothing.
\end{remark}

As corroboration, two independent instruments were run over this appendix's
scope. The referee's exact fuzz --- $33$ instances with non-monotone hourly
curves and varied in-scope coefficients, $410$ realized root walls --- found
every wall in families (a)--(b) with exactly zero offset on the roots. An
exact $d=3$ prototype of the recursion itself --- rational arithmetic
throughout, a global max-of-affine representation with no viewport --- then
replayed \eqref{app:statd-eq-step} against an independent, duality-certified
rational-simplex judge: over $100$ feasible instances, $1{,}600$ parametric
probes of the terminal value agree exactly (zero discrepancy, and the two
sides agree on every infeasible instance), and a census of $4{,}628$
realized walls across $238$ carried objects finds axes and roots only,
every root wall with exactly zero offset. On the same instances the carried
face count stays at least two orders of magnitude below the bound of
\cref{cor:arrangement} ($F\le29$ against $N^{3}\ge5{,}832$ at $d=3$); the
measured range is too narrow to fit an exponent and is reported as a
measurement, not a law. Prototype, campaigns and audits are
\texttt{stationarity\_exact/} in the \texttt{prototypes} repository
(\cref{sec:package}).

\paragraph{The reference implementations, measured.}
The remark closing \cref{cor:arrangement} --- that the work of one step is a
statement about implementations, not a theorem --- has been instantiated.
Three reference kernels were built in succession, each gated on its
predecessor before any timing: the exact rational prototype above (no
timing, by principle), a floating-point active-set port checked against it,
and a generic-$d$ C++ kernel (one template, instantiated $d=2..6$, floating
point contraction disabled) whose clamp generates its crossing candidates
in dimension $d{-}1$ on the bound hyperplane --- output byte-identical to
the full $O(F^2)$ generation across \ddimByteGateCells{} gate cells, with
the candidate count divided by \ddimGenReductionLow{} to
\ddimGenReductionHigh{}. On park-catalogue instances the realised face
count stays bounded through $d=5$: $F_{\max}=$ \ddimFmaxDTwo{},
\ddimFmaxDThree{}, \ddimFmaxDFour{}, \ddimFmaxDFive{} at three slopes for
$d=2,3,4,5$, a per-dimension amplification that \emph{decreases}
(\ddimAmpTwoThree{}, \ddimAmpThreeFour{}, \ddimAmpFourFive{}) --- far below
the bound of \cref{cor:arrangement} at every $d$. Measured against a frontal
LP outside any speed claim of this paper, the kernel answers the annual
$d=2$ instances around parity with the frontal --- ahead by up to
\ddimDTwoYearBSBest{} with model build included, behind at worst
\ddimDTwoYearBSWorst{} --- and trails it from $d=3$ on (solve ratios
\ddimDGeqThreeSolveLow{} to \ddimDGeqThreeSolveHigh{}); every timed
instance matches the integer LP optimum exactly (\ddimGateBCells{} cells).
The cost is the pruning of \emph{constructed} candidates, not the carried
object, and the remaining wall is engineering with a name --- the
degenerate on-face sorting of the $(d{-}1)$-dimensional candidate set ---
recorded, with the three campaigns, under \texttt{diagnostics/} in the
benchmark repository (\cref{sec:package}).

\section{Coordination of coupled storage blocks, and the system-scale study}
\label{app:coordination}

This paper's object is the exact value function, and it has a boundary: a
coupled group holding more than one intertemporal unit leaves it. Past it the
composition of \cref{sec:coordination} is a coordination scheme --- iterates
feasible with a measured gap, not globally optimal --- and a certified joint
block removes the lock exactly on small groups, buying correctness, not speed.
What the decomposition holds against a frontal barrier is memory and exactness,
not time; read every measurement below in that light. They are recorded so that
the boundary is documented rather than asserted.

\subsection{Block-coordinate alternation and its lock}

With several units per group each curve is conditional on the others'
schedules, so a sweep is block-coordinate descent on a convex but
\emph{nonsmooth} objective, whose coordinatewise optima need not be
global~\cite{Tseng2001}: co-located high-power devices share a merit-order kink
neither can cross alone. \Cref{prop:one-storage} is the exact case. Beyond it
comes the multi-storage stall (\cref{subsec:alternation}): devices at the same
or tightly connected nodes leave plain alternation several percent above the
monolith, and refining the spatial curves does not remove that floor, so the
gaps in \cref{tab:annual-comparison} are coordination gaps, not oracle errors.
The residual is a lock, not slow convergence --- \cref{subsec:ramps-bench}
exhibits a plateau, not a decreasing sequence --- and \cref{subsec:masters}
removes it with a certified gap.

Vocabulary: spatial curves and FLT subproblems are \emph{exact}; the
one-storage composition is \emph{globally exact}; a multi-storage alternating
iterate is \emph{feasible with a measured gap to the reference monolith}; a
master is exact only when its bound and convergence argument cover the
implemented algorithm.

\subsection{Masters for coupled storage blocks}
\label{subsec:masters}

When a group locks, refine the \emph{block}, not the curves: group the units
coupling through one kink and re-optimise them together. At one node a
two-sided parametric trace along the combined-injection ray gives the
merit-order epigraph, over which one small auxiliary LP returns the jointly
optimal split; across a line the hourly value is multi-dimensional and no
single ray traces it, so a cutting-plane master fed with exact tangent planes
assembles it and stops at a certified relative gap, its reported cost always
the exactly re-priced upper bound. Either way one exact trace plus one small LP
per coupled group keeps the method frugal, with no node-by-hour monolith, and a
genuine floor survives only at a precise eligibility boundary decided from the
topology before any solve (\cref{subsec:ramps-bench}).

The cross-node certificate reads too easily as a scaling claim, so it deserves
an honest bound. Two things grow with the group's connectivity: the master
carries every intertemporal unit over the whole horizon --- removing the
network takes out about a tenth of the problem, so its \emph{base} approaches
the monolith it replaces --- and the rounds closing the certificate grow with
the dimension of the hourly value the planes assemble. On three zones it closes
in \crossThreeZoneRounds{} rounds at parity with the same monolith and solver;
on eight zones the gap still falls geometrically but the master lands orders of
magnitude behind it (\cref{subsec:ramps-bench}), structurally rather than
numerically: the planes are hourly and their slopes do \emph{not} repeat
usefully across the year at that size. \emph{The joint block buys exactness and
a certificate, not speed}, and belongs on small groups; for two stores at one
node \cref{sec:pair} \emph{eliminates} the coupling instead, and
\cref{rem:two-store-scope} bounds where this master remains the general remedy.

\subsection{Working memory of the decomposition}

The decomposition never forms the full node-by-hour monolith: its working
memory is hourly spatial models, scalar traces and temporal value functions ---
what the sizing oracles of \cref{sec:investment} rely on, and what survives to
the sizes at which the monolith is unaffordable. The ramp-free backend streams
each trace into the dynamic-programming buffers and reuses the committed trace
for exact objective monitoring (\cref{sec:package}). Two optional parallel axes
--- hour-chunked curve construction, concurrent Jacobi storage updates ---
change wall time only: schedule, objective and feasibility gap are identical to
the single-threaded run, so parallelism is a deployment convenience, not part
of the method. Its campaign ran on a different machine from every other
measurement here and is therefore reported in the benchmark repository
(\cref{sec:package}) rather than in this paper's tables, where it would
compare to nothing.

\subsection{Annual multi-node electricity systems}
\label{subsec:multinode-bench}

\begin{landscape}
\begin{table}[p]
\centering
\caption{Annual TYNDP benchmark ($T=8{,}760$). FLT denotes the isolated one-node temporal solve and FLT coordination the complete curve-based method. From three nodes onward the timed row measures the coordination layer --- the traced injection-cost curves and the alternating one-dimensional solves --- which consumes the FLT temporal solve as its oracle; the one-node column is that FLT temporal solve in isolation. All three methods are timed end to end (build plus solve), both frontal solvers' solver-only time in parentheses. A dash marks a measurement that does not apply to that column. Ramp-constrained fleets are reported separately at operation scale in \cref{subsec:ramps-bench}; their re-gating at TYNDP scale with the joint cross block is left to future work. (\pkg{} 0.27.0, median of 3, machine \texttt{cyrius-2.home}.)}
\label{tab:annual-comparison}
\resizebox{\linewidth}{!}{%
\begin{tabular}{lrrrrrr}
\toprule
 & 1 node & 3 nodes & 8 nodes & 10 nodes & 30 nodes & multi-energy \\
\midrule
data & TYNDP--FR & TYNDP & TYNDP & TYNDP & TYNDP & TYNDP \\
horizon $T$ & 8,760 & 8,760 & 8,760 & 8,760 & 8,760 & 8,760 \\
method & FLT & coordination & coordination & coordination & coordination & coordination \\
\midrule
Gurobi build+solve (solve) [s] & 0.53 (0.19) & 1.13 (0.42) & 3.80 (1.60) & 5.58 (2.50) & 16.35 (8.32) & \textbf{2.17} (1.55) \\
HiGHS build+solve (solve) [s] & 0.57 (0.43) & 2.32 (1.55) & 30.78 (13.64) & 55.62 (16.53) & 195.38 (40.74) & 15.62 (8.32) \\
FLT / coordination wall [s] & \textbf{0.03} & \textbf{0.20} & \textbf{2.54} & \textbf{3.15} & \textbf{15.23} & 2.45 \\
Gurobi / FLT time & $\mathbf{16.1\times}$ & $\mathbf{5.62\times}$ & $\mathbf{1.50\times}$ & $\mathbf{1.77\times}$ & $\mathbf{1.07\times}$ & $0.885\times$ \\
HiGHS / FLT time & $17.3\times$ & $11.5\times$ & $12.1\times$ & $17.6\times$ & $12.8\times$ & $6.38\times$ \\
\midrule
Gurobi peak RSS [MB] & 216 & 478 & 1,197 & 1,566 & 3,673 & 680 \\
HiGHS peak RSS [MB] & 200 & 414 & 991 & 1,457 & 3,256 & 884 \\
FLT / coordination RSS [MB] & \textbf{59} & \textbf{133} & \textbf{321} & \textbf{264} & \textbf{550} & \textbf{300} \\
\midrule
objective gap to Gurobi [\%] & $-0.0000$ & $+0.0026$ & $+0.0459$ & $+0.0400$ & $+0.1045$ & $+0.0793$ \\
power/state violation & 0 & 0 & 0 & 0 & 0 & 0 \\
additional certificate & exact & -- & -- & -- & -- & -- \\
\bottomrule
\end{tabular}}
\end{table}

\end{landscape}

On the 3-, 8- and 10-node systems the single-storage advantage survives as a
cheap oracle, but coordination rather than the longitudinal solve sets total
runtime as the network grows. Against Gurobi's bare solver time (parenthetical
column) the frontal barrier is the faster beyond a few nodes; the end-to-end
ratios of \cref{app:campaigns} ($\annualThreeNodeGurobiSpeedup\times$,
$\annualEightNodeGurobiSpeedup\times$ and $\annualTenNodeGurobiSpeedup\times$
at 3, 8 and 10 nodes) are favourable only because the frontal model's build
dominates those runs, and we read no method property into them. What the
decomposition holds at every size is memory ---
\annualMemFactorThreeToTenRange{}$\times$ less --- at best feasible gaps at or
below $\annualThreeToTenGapMax$ with zero physical violations. No time ratio
against HiGHS is quoted: on identical code and models its wall clock varied by
factors of two to five across machines and solver builds in our campaigns, so
it serves as an objective cross-check and a memory point, not a timing
baseline.

The 30-node column is not synthetic: annual TYNDP profiles and capacities from
the capacity-expansion solution over thirty European zones, seventeen of which
retain a significant reservoir (list in the reproducibility package). On the
local machine the decomposition takes $\annualThirtyNodeFltWall$~s against
$\annualThirtyNodeGurobiWall$~s for a frontal Gurobi build-and-solve and
$\annualThirtyNodeHighsWall$~s for HiGHS on the same sparse monolith, returning
a feasible $\annualThirtyNodeGap$ solution. The end-to-end near-parity with
Gurobi ($\annualThirtyNodeGurobiSpeedup\times$) is a build artefact, not a
solver result: against the bare Gurobi solve the decomposition is roughly half
as fast at this size. Memory separates them more sharply:
$\annualThirtyNodeFltRss$~MB against $\annualThirtyNodeHighsRss$~MB for HiGHS
($\annualThirtyNodeHighsMemFactor\times$) and $\annualThirtyNodeGurobiRss$~MB
for Gurobi ($\annualThirtyNodeGurobiMemFactor\times$). Its 30 sweeps are still
descending, so this point is a fixed-budget comparison, not evidence of a
terminal oscillation.

\subsection{Multi-energy: coupled electricity--hydrogen dispatch}
\label{subsec:multienergy-bench}

The coupled electricity--hydrogen instance exercises sparse multi-commodity
directions in \cref{alg:trace}: a "node" is an (area, resource) pair, and a
hydrogen store couples its hydrogen and electricity balances through a
parasitic charging load. In three fresh local processes, directional traces,
the monotonicity guard and hour-level parallelism return a physically feasible
schedule in $\annualMultiEnergyFltWall$~s at the median, against
$\annualMultiEnergyGurobiWall$~s (Gurobi) and $\annualMultiEnergyHighsWall$~s
(HiGHS), at an empirical $\annualMultiEnergyGap$ gap. Gurobi is faster on this
small, quickly built model; \pkg{} uses
\annualMultiEnergyMemFactorRange{}$\times$ less peak memory
($\annualMultiEnergyFltRss$~MB versus $\annualMultiEnergyGurobiRss$ and
$\annualMultiEnergyHighsRss$~MB) and is $\annualMultiEnergyHighsSpeedup\times$
faster than HiGHS. The result is a high-quality feasible solution, not an
equally certified optimum.

\paragraph{An 8-node coupled system with demand management, from \texttt{pommes}.}
The case study scales the coupling to eight areas (FR, DE, BE, NL, CH, AT, NO,
SE) with demand-side management: \texttt{pommes}~\cite{pommes} builds the
ENTSO-E TYNDP ``Global Ambition'' 2050 operation case, \pkg{} slices it to a
plain multi-resource instance --- eight electricity nodes, seven hydrogen
nodes, $\mrEightStorages$ storages (hydro reservoirs, batteries, pumped hydro,
hydrogen salt caverns and tanks) --- and the same coupled decomposition solves
it. Demand flexibility is one shiftable-load band per electricity node, at
$\mrEightDsmRho\%$ of mean demand and \emph{energy-conserved over daily
windows}: each band is a storage whose cumulative state is pinned back to its
mid-point at every window boundary (the per-step level pins of
\cref{sec:package}). The returned schedule is feasible at a $+\mrEightGap\%$
coordination gap, with no power violation, a state violation of
\mrEightLevelViolation{}~MWh, and pins holding to \mrEightDsmPinError{}~MWh
over $\mrEightDsmWindows$ windows and $\mrEightDsmBands$ bands --- demand
management thus enters the exact longitudinal oracle as pure interval data. On
the complete storage fleet the flexibility is worth $\mrEightDsmValue\%$ of
system cost (roughly half generation-cost arbitrage, half avoided peak
scarcity), not inflated by a missing short-term fleet since batteries and
pumped hydro are present. The residual gap is the coupled-storage coordination
floor of \cref{subsec:alternation} --- the caverns' parasitic-charge kink ---
not a limitation of the demand-management encoding.

Trace-based objective monitoring is quantified here too: evaluating each
sweep's objective by 8,760 independent LPs cost about 33.5 seconds per sweep,
whereas reusing the committed exact trace (\cref{sec:package}) reduces it to
0.04 seconds, a same-driver A/B run giving 57.87 versus 10.24 seconds end to
end ($5.65\times$). An audit at every identical committed state bounds the
relative difference between the two evaluations by $1.3\,10^{-15}$; the earlier
monitor also shared solver state with the next sweep, a measurement failure
mode recorded in \cref{app:campaigns}.

\subsection{Ramp constraints and coupled-storage coordination}
\label{subsec:ramps-bench}

A ramp-constrained fleet is itself a longitudinal chain
(\cref{subsec:exclusions}), but coupled to a storage it triggers the
block-coordinate lock. We measure the lock and its removal by the joint block
of \cref{subsec:masters} on a small annual operation case
(\texttt{op\_ramp\_simple}, $T=8{,}760$, one laptop; the release is stated in
the caption of \cref{tab:ramp-simple}): a
four-node ring, node~3 carrying the demand, a co-located volatile renewable, a
battery ($200$\,MW\,/\,$1500$\,MWh) and a mid-merit ramp fleet ($2000$\,MW,
$R=40$\,MW/h). The renewable swings the residual load fast, so the ramp binds
($\max_t|\Delta g_t|=R$). Every gap is measured against an independent Gurobi
monolith replicating the multi-resource formulation exactly (equal to HiGHS to
$\sim\!10^{-15}$).

\Cref{tab:ramp-simple} has three steps. Storage only, the decomposition is
globally exact and solves in $\rampStorageOnlyWall$\,s, ahead of the Gurobi
monolith ($\rampGurobiWall$\,s) and HiGHS ($\rampHighsWall$\,s) ---
\cref{prop:one-storage} at operation scale. Adding the ramp fleet at the
\emph{same} node, a plain Frank--Wolfe coordination of the two chains loses
that lead entirely: it plateaus at \rampFwGap{} above the monolith in
$\rampFwWall$\,s, the pure block-coordinate lock (each chain optimal given the
other, their split of the node aggregate wrong), the iterates no longer moving.
The joint block restores exactness --- one two-sided trace and one small
epigraph LP return the exact optimum, at most \rampJointGap{} relative, in
$\rampJointWall$\,s; at these sub-second scales the wall clocks are
build-dominated and we read no speed ranking into them.

\begin{table}[t]
\centering
\caption{Ramp fleet coupled to storage at one node (\texttt{op\_ramp\_simple}, $T=8{,}760$, one laptop). Wall seconds and relative gap to an independent Gurobi monolith (equal to HiGHS to $\sim\!6.4\,10^{-16}$). The decomposition rows measure the cross-block coordination --- the fleet/storage alternation lock and its removal by the joint block --- with the FLT temporal solve as the oracle inside each block; the two monolith rows solve the same instance frontally. (\pkg{} 0.27.0.)}
\label{tab:ramp-simple}
\begin{tabular}{lrl}
\toprule
method & wall [s] & gap to monolith \\
\midrule
storage only, decomposition & \textbf{0.03} & exact ($\le2.6\,10^{-15}$) \\
\quad + ramp fleet, plain Frank--Wolfe & 9.1 & $+1.05\%$ \\
\quad + ramp fleet, joint block & \textbf{0.37} & exact ($\le1.1\,10^{-14}$) \\
\midrule
same instance, Gurobi monolith & 0.56 & (reference) \\
same instance, HiGHS monolith & 1.7 & (reference) \\
\bottomrule
\end{tabular}
\end{table}

The refinement generalises off the single node. With fleet and storage at
\emph{different} nodes joined by a line, plain alternation locks at about
$+1.0\%$ almost independently of line capacity: loose lines make the network a
copper plate with the same merit-order kink, while congestion only makes the
joint hourly value two-dimensional and rules out a single-ray trace. The
certified cross-group then returns the exact optimum ($\sim\!10^{-12}$
relative) in $0.7$--$1.0$\,s with $6$--$8$ Kelley planes. The same machinery
closes the multi-storage floor without any fleet: two co-located storages of
\emph{different} duration ($2$\,h and $12$\,h) leave plain alternation at
$+0.027\%$ after $8.1$\,s, and the joint block returns the exact optimum
($\sim\!10^{-13}$) in one provable sweep of $0.60$\,s; two network-connected
storages close exactly ($\sim\!10^{-13}$) by cross-group in about $0.7$\,s.
Equal-duration co-located storages need no refinement: their aggregate is one
storage with summed energy and power.

\paragraph{Coordinating the pair, against eliminating it.}
The smallest group this block handles --- two co-located stores of different
duration at one node --- is exactly the one \cref{sec:pair} eliminates
outright, by carrying the value function of the pair instead of coordinating
the two. On the identical instance of \cref{tab:two-store-block}, the
conjugate recursion is \twoStoreBlockSpeedupRange$\times$ faster than this
block, the spread between the two ends being the losses, which split each
pencil of the arrangement and multiply its face count. It also returns
$V^\star_0,\dots,V^\star_{T-1}$ at
$\twoStoreFamilyMbLossless$--$\twoStoreFamilyMbLossy$\,\si{\mega\byte}, which
an LP block cannot produce at any price. That is the sharpest form of the
boundary this appendix documents: where a coupled block is small enough to
eliminate, eliminating dominates coordinating on every axis, and the joint
block earns its place only on the groups that cannot be eliminated --- three
or more stores with distinct coefficients, or units spread across a line
(\cref{rem:two-store-scope}).

\paragraph{Which units a block can absorb.}
Eligibility is decided before any solve, from the topology alone, and three
classes fall outside it. A storage whose load does not lie along a single
traced direction --- an H2 store with a parasitic electricity draw, contributing
to two balances at once --- has no injection ray for the trace to follow. A
lossy storage at a node without curtailment makes the hourly cost non-monotone
in the injection, so the traced curve is no longer the storage's own cost. A
level-pinned storage is not free over the horizon the block optimises. Such
units stay on the per-unit path and the run is a hybrid, to be reported as one:
on the eight-zone system above, $7$ of $29$ storages are ineligible for the
first reason, so the ``joint block'' spans $22$ units and not the system.

\paragraph{Where the certified cross-group stops being worth its cost.}
The Kelley planes above close in single digits because the group is small;
outside that regime the bound of \cref{subsec:masters} is reached quickly. On a
real eight-zone system with \crossEightZoneUnits{} intertemporal units in one
connected group the certificate does still close --- the gap falls by a factor
\crossEightZoneGapFactor{} per round --- but needs of the order of
\crossEightZoneRoundsExtrap{} rounds, each solving a master grown by one plane
per hour whose base already carries every unit over the whole year, where the
monolith answers the same instance in \crossEightZoneMonolithSeconds\,s. Nor is
there hidden structure to exploit: the hourly tangent slopes were still
accumulating at \crossEightZoneNewSlopesRound{} new ones per round after
\crossEightZoneRoundsProbed{} rounds, \crossEightZoneSeenOnceShare{} of them
occurring at a single hour of the year, so the planes cannot be shared across
time. The joint block is therefore an \emph{exactness} device with a
certificate, on groups small enough for it to close --- not a way to make large
coupled systems faster.

\section{Parametric-simplex engineering}
\label{app:simplex}

This appendix records the engineering behind \cref{alg:trace} and
\cref{subsec:curve-engineering}, at the level of detail needed to reproduce
the kernel measurements.

\paragraph{Basis atlas and predictor.}
The implementation retains one common-dispatch basis per hour and the
storage--hour anchor bases produced in the preceding sweep. A predictor
chooses among that spatial atlas, the previous storage-specific anchor and
the basis carried longitudinally from hour $t-1$; the dual simplex then
certifies or repairs the choice before retracing the complete curve. This
reuse changes only the starting active set, never the curve being computed.
On the annual 8-node profile, the traced curves average $2.35$ affine pieces
per hour.

\paragraph{Where native simplex time goes.}
A pass-by-pass profile of the complete 30-sweep 8-node run assigns $58.3\%$
of native simplex work to the trace itself and $30.0\%$ to anchor
reoptimisation (inclusive). Maintaining the basic primal vector
incrementally reduces the complete run by $4.5\%$ with one thread and
$2.7\%$ with eight, with identical gap trajectories; basis refactorisation
accounts for only $6.0\%$ of native work. A dense reduced-cost cache was
slower ($+1.9$--$5.6\%$), while incrementally maintaining only the dual
vector was neutral ($+0.04\%$ at eight threads); both variants were removed.
The retained end-to-end lever instead fuses trace buffers directly into the
flat DP, reducing the annual 8-node eight-thread wall by $17.9\%$.

\section{Campaigns, versions and measurement conventions}
\label{app:campaigns}

\paragraph{Campaigns.}
Every table in this paper is emitted by a generator that reads one archived
result file, and each caption names the release and the machine that produced
it; this appendix records only what those captions share. The three-solver,
kernel, two-stage, catalogue-law and application rows come from unified
campaigns against the installed \pkg{} \pkgCampaignVersion{} release,
measured on an 8-thread macOS laptop; the parametric-reuse, joint-sizing and
ramp tables were re-measured at the same release on 2026-08-17 (the
joint-sizing rows moving from single runs to medians of three), and the
two-store block keeps its $0.23.0$ harness output together with a
reproduction control at the current release stated in its caption --- each
caption names its release, and any residual provenance mixing is visible
rather than implied. A small number of narrative figures in
\cref{sec:injection,sec:experiments} and \cref{app:simplex} quote one-shot
research-harness runs that predate the campaign contract (a profiling
breakdown, two small cross-group cases, a monitoring A/B); they carry no
headline claim, and they are recognisable by the absence of a generated
macro. The parallel-axes campaign ran on a different server
(Calcul-Planeterr, Xeon Gold 6354) and is recorded in the supplementary
material rather than in this paper's tables, precisely because cross-machine
rows compare to nothing else here.

Two qualifications, both stated in the captions concerned rather than inherited
from this paragraph. The joint-sizing measurements of \cref{subsec:joint-bundle}
are research-harness runs --- single runs rather than repeated medians --- so their
sub-second wall times carry a few percent, and the ranges derived from them are
rounded outward.
Re-running the suite therefore relabels its own tables, and a measurement that
no archived campaign supports is printed as a visible gap rather than a
plausible number.

\paragraph{Frontal references and timing conventions.}
The frontal references use HiGHS 1.14.0 (1.15.1 on the server) and Gurobi
13.0.2 on the natural sparse state-of-charge formulations. A common driver
(\texttt{benchmark\_common.py}) measures all three methods with identical
end-to-end accounting --- model build plus solve, excluding only the shared case
load, median of three repetitions, one fresh process per measurement for a clean
peak-memory reading. This removes the earlier asymmetry (Gurobi timed
solve-only, HiGHS build+solve, the decomposition end-to-end) that had favoured
Gurobi in the time rows; the Gurobi solver-only time is retained in parentheses
on the Gurobi row of \cref{tab:annual-comparison} for continuity.

\paragraph{Objective monitoring.}
The decomposition evaluates each sweep's objective from the committed exact
trace (\cref{sec:package}); optional direct checks use a separate spatial
solver so that monitoring cannot perturb the traced bases that seed the next
sweep. This avoids a subtle measurement failure mode, in which a direct
monitor sharing the seeding basis could change the subsequent simplex path.
The audit of \cref{subsec:multienergy-bench} compares the two evaluations at
every identical committed state and bounds their relative difference by
$1.3\,10^{-15}$.

\end{document}